\documentclass{article}

\usepackage[utf8]{inputenc}% encodage du fichier source
\usepackage[T1]{fontenc}% gestion des accents (pour les pdf)
\usepackage[english]{babel}% rajouter éventuellement english, greek, etc.
\usepackage{textcomp}% caractères additionnels
\usepackage{amsmath,amssymb,amsthm}% pour les maths
\usepackage{multirow}
\usepackage{lmodern}% remplacer éventuellement par txfonts, fourier, etc.
\usepackage[a4paper]{geometry}% taille correcte du papier
\usepackage{graphicx}% pour inclure des images
\usepackage{xcolor}% pour gérer les couleurs
\usepackage{microtype}% améliorations typographiques
\usepackage{hyperref}% gestion des hyperliens
\usepackage{cite}
\usepackage{bbold}
\usepackage{bbm}
\usepackage{mathtools}
\usepackage{xcolor}
\usepackage{lipsum}
\usepackage{setspace}
\usepackage[normalem]{ ulem }
\usepackage[font={footnotesize,it}]{caption}

\hypersetup{pdfstartview=XYZ}% zoom par défaut

\newcommand{\reels}{\mathbb{R}}

\newcommand{\esp}{\mathbb{E}}
\newcommand{\proba}{\mathbb{P}}
\newcommand{\nat}{\mathbb{N}}
\newtheorem{theoreme}{Theorem}[section]
\newtheorem{proposition}[theoreme]{Proposition}
\newtheorem{corollaire}[theoreme]{Corollary}
\newtheorem{assumption}[theoreme]{Assumption}

\theoremstyle{definition}
\newtheorem{definition}[theoreme]{Definition}
\theoremstyle{remark}
\newtheorem{remarque}{Remark}

\usepackage{todonotes} 
\newboolean{DisplaySOS}
\setboolean{DisplaySOS}{true} %% Mettre false dans les versions finales pour retirer les remarques!
\newcommand{\SOS}[1]{\ifthenelse{\boolean{DisplaySOS}}{{\textcolor{blue}{\textbf{Note:} [#1]}}}{}}
\title{A new stochastic 
	STDP Rule \\ in a neural Network Model}
\author{Pascal Helson\footnote{pascal.helson@inria.fr}}
\date{Draft \today}

\graphicspath{{images/}}

\begin{document}
	\maketitle
	%\makeatletter
	%\begin{titlepage}
	%	\centering
	%	{\LARGE \textbf{A simple spiking neuron model based on stochastic STDP}}\\
	%	\vspace{5em}
	\tableofcontents
	%	\vspace{5em}
\newpage

%\eot[inline]{Quelques remarques générales :\\
%	%\begin{enumerate}
%	%\item 
%	1) Dans le texte, il faut écrire les quantités mathématiques en mode mathématique. Par exemple, page 20, tu écris :
%``whenever a neuron i jump'' qu'il faut remplacer par ``whenever a neuron \(i\) jump''. Cette modification est à faire dans tout ton document.\\
%%\item 
%2) Remplace les \(\overset{\sim}{w}\) par \(\tilde{w}\).\\
%3) Homogénéise autant que possible. Par exemple \(\mathbb{R}_+\) ou \(\mathbb{R}^+\) partout (plutôt le premier).\\
%4) it depends only \(\rightarrow\) it only depends.\\
%%\end{enumerate}
%}

	\section*{Abstract}
	Thought to be responsible for memory, synaptic plasticity has been widely studied in the past few decades. One example of plasticity models is the popular Spike Timing Dependent Plasticity (STDP). There is a huge litterature on STDP models. Their analysis are mainly based on numerical work when only a few has been studied mathematically. Unlike most models, we aim at proposing a new stochastic STDP rule with discrete synaptic weights. It brings a new framework in order to use probabilistic tools for an analytical study of plasticity. A separation of time-scale enables us to derive an equation for the weights dynamics, in the limit plasticity is infinitely slow compare to the neural network dynamic. Such an equation is then analysed in simple cases which show counter intuitive result: divergence of weights even when integral over the learning window is negative. Finally, without adding constraints on our STDP, such as bounds or metaplasticity, we are able to give a simple condition on parameters for which our weights' process remains ergodic. This model attempts to answer the need for understanding
	the interplay between the weights dynamics and the neurons ones.
\section{Introduction}
\vspace{-1em}
%It is commonly accepted that neurons mostly communicate by spikes. A neuron receives spikes from thousands other neurons which affect the neuron'and this affects its morphology, sensitivity. In particular, they caduc and reactivity through the changes oolution in timef its synapses. The dynamical evolution of these synapses is called \textit{synaptic plasticity}. Such plastic behaviour is thought to be at the basis of our memory formas makes it particul                                          c plastcity has been widely studied oarly interesting both at the biologicl level but also at a computional level for its implications in machine learning. Over the plast few decades. There exist plasticity, either biophysical or phenomenological or both. 
%These studies gave rise to new research fields in artificial intelligence and allowed much progress in our understanding of the brain.
%	\vspace{-0.7 em}
%	\begin{figure}[h!]
%		\begin{minipage}{0.77\linewidth}	
A huge amount of studies have focused on neural networks dynamics in order to reproduce biological phenomena observed in experiments. Thereby, there exist many different individual neuron models from the two states neurons to the adaptive exponential integrate-and-fire \cite{izhikevich_dynamical_2007,gerstner_spiking_2002}. Compare to this kind of literature, plasticity in recurrent networks has been well less studied. One reason is because it adds an additional layer of complexity to existing models despite being a candidate for memory formation, learning, etc \cite{brunel_is_2016,benna_computational_2016}. 
\\
In the beginning, plasticity models were based on firing rates~\cite{bienenstock_theory_1981}. Later on, as suggested by Hebb's in 1949~\cite{hebb_organisation_1949}, the crucial role of precise spikes timings was proved experimentally and gave rise to Spike-Timing Dependent Plasticity (STDP)~\cite{markram1997regulation,bi_synaptic_1998,markram_history_2011}. Following such a breakthrough, numerous STDP models emerged. They were associated with neural networks of either Poisson neurons~\cite{kempter_hebbian_1999,kempter_intrinsic_2001,gilson_emergence_2009} or continuous model of neurons~\cite{abbott_synaptic_2000,clopath_connectivity_2009,ocker_self-organization_2015}.
Here, we would like to present a new STDP rule which 
%	we combine a simple two states model of individual neurons with a new stochastic STDP rule.
%(What enabled us this learning rule compare to other models? What can it bring in more?)
is implemented in the well-known stochastic Wilson-Cowan  model of spiking neurons as presented in~\cite{benayoun_avalanches_2010}. More precisely, because of the plasticity rule, our model is a piecewise deterministic Markov process~\cite{davis_piecewise-deterministic_1984,davis_markov_1993} whereas it is a pure point process in~\cite{benayoun_avalanches_2010}.\\\\
Motivations for proposing such a new model are four folds. First, although mechanisms involved in plasticity are mainly stochastic\, such as the activation of ions channels and proteins, the majority of studies on STDP are implemented using a deterministic description or an extrinsic noise source~\cite{morrison_phenomenological_2008,clopath_connectivity_2009,graupner_calcium-based_2012}. One exception is the stochastic STDP model proposed by Appleby and Elliott in~\cite{appleby_synaptic_2005,appleby_stable_2006}. The stochasticity of their model lies in the learning window size. They analyse 
%\eot[]{analyze}
 the dynamic of the weights of one target cell innervated by a few Poisson neurons. A fixed point analysis enabled them to show that their model is not relevant in the pair-based case and that multispike interactions are required to get stable competitive weights dynamics. Second, most studies are based on simulations and their analyses, thus there is still a need to find a good mathematical framework, see \cite{galtier_biological_2013,litwin_formation_2014, ocker_self-organization_2015}. We propose here a mathematical analysis based on probabilistic methods which leads to a control of weights through the study of their dynamics on their slow time scale. Indeed, long term plasticity timescale ranges from minutes to more than one hour. On the other hand, a spike lasts for a few milliseconds~\cite{morrison_phenomenological_2008}. Thus, third, there is a need to understand how to bridge this time scale gap between the synapse level and the network one\cite{fox_integrating_2017,zenke_temporal_2017,turrigiano_dialectic_2017}. Finally, the interplay between the weights dynamics and the neurons ones is not yet fully understood and we think the study of recurrent networks is necessary to bring some basis to fully numerical studies.\\\\
Such motivations impose some constraints on our model. It has to be rich enough to reproduce biological phenomena, simple enough to be mathematically tractable and easily simulated with thousands of neurons. Finally, it has to enable us to observe macroscopic effects out of microscopic events. The Wilson-Cowan model has been widely studied~\cite{bressloff_metastable_2010,benayoun_avalanches_2010,litwin_formation_2014} and reproduces many biological features of a network such as oscillation and bi-stability for example. On the other hand, based on experimental evidence \cite{bi_synaptic_1998,ribrault_stochasticity_2011}, we propose a new STDP rule with intrinsic noise with fixed synaptic weight increment \cite{oconnor_graded_2005}. This allows to control independently the synaptic weight increment and the probability of a plasticity event. Indeed, several pairs protocol are required for the induction of plasticity \cite{bi_synaptic_1998,markram1997regulation}.

Thus, we can produce a mathematical analysis by studying the Markov process composed of the following three components: the synaptic weight matrix, the inter-spiking times and the neuron states. In the context of long term plasticity, synaptic weights dynamics are much slower than the neural network one. A timescale analysis enables us to remove the neurons dynamics from the equations. Then we can derive an equation for the slow weights dynamics alone, in which neurons dynamics are replaced by their stationary distributions. Thus, we don't need to simulate the dynamics of thousands of fast neurons and we obtain a much easier equation to analyse. We then discuss the implications of such derivation for learning and adaptation in neural networks.\\
A similar analysis has been done in a few papers with different mathematical tools and models~\cite{kempter_hebbian_1999,kempter_intrinsic_2001,ocker_self-organization_2015,gilson_emergence_2009,gilson_emergence_2010,lajoie_correlation-based_2017}. When the two first one studied only one postsynaptic neuron, the last ones had a look at recurrent networks. Thanks to a separation of time scale, they derive an equation for weights in which STDP appears in an integral of the STDP curve against cross-correlation matrix. The main problem is the computation of such a matrix, they use Taylor expansion and Fourier analysis to derive estimations of it. We don't need such an estimation for our analysis thanks to probabilistic methods.

\section{Presentation of the model and notations}

 As in all model of neural networks with plastic connections, one can separate the neuron model and the plasticity one. Our neuron model is the well-known stochastic Wilson-Cowan model of spiking neurons presented in~\cite{benayoun_avalanches_2010}. In such a model, neurons are binary, meaning they are either at rest, state 0, or spiking, state 1. 
% {\color{black}This latter can be interpreted as the state of influence on other neurons.}
 This model has been widely studied in the case of fix weights and presents realistic features such as oscillations or bistable phenomenon, see~\cite{bressloff_metastable_2010}. However, there are only few studies with plasticity, see for instance with an Ising model in~\cite{pechersky_stochastic_2017}.\\
 We implement plasticity in this model in a stochastic way. Indeed, our plasticity rule depends on the precise spike times and thus has the same form as STDP, see~\cite{markram_spike-timing-dependent_2012} for an overview, but is not deterministic: in the situation of correlated spikes, weights will change or not according to a certain probability.\\
 First, we are interested in excitatory neurons, as in most models inhibitory neurons are not plastic, so the synaptic weights will be positive. Also, we suppose they are all to all connected so this positivity will be strict. {\color{black}We will discuss about these assumptions at the end.} Therefore, we first give some global notations, then explain the neuron model, the plasticity rule, and finally we gather these dynamics in the generator of the process.\\
\vspace{5em}

 \noindent We are interested in analysing the time continuous Markov process $(W_t,S_t,V_t)_{t\geq 0}$ where:
 \begin{itemize}
 	\item [-] $W_t\in \left\{\Delta w K, K\in E_0\right\}$ synaptic weights matrix, $E_0 = \left\{K, K\in \nat^{N^2},\ K_{ij}>0\ \forall i\neq j\ and\ K_{ii}=0\ \forall i\right\}$, $\Delta w \in \reels_*^+$ and $W_0\in \left\{\Delta w K, K\in E_0\right\}$, $W_t^{ij}$ weight of the connection from neuron $i$ to $j$at t.
 	\item [-] $S_t \in \reels_+^N$ vector of times from last spikes of neurons.
% 	\eot[inline]{Pas vraiment l'interspike, plutôt le temps passé depuis le dernier spike. Ce serait l'interspike si ton neurone déchargeait à l'instant \(t\).}
 	\item [-] $V_t \in I=\{0,1\}^N$ neuron system state.
 \end{itemize}
% We can see in this definition that the weights dynamics consist in jumps of fixed height $\Delta w$ in our model. Also, the presence of the inter-spiking times is crucial to have the Markov property.
 As weights dynamics and the neural network one will be separated, we spare the global state space $E$ in two spaces. Hence, in the following we denote $E_1=\left\{\Delta w K, K\in E_0\right\}$, $E_2=\reels_+^N\times I$ such that $E=E_1\times E_2$.\\
 
 \subsubsection*{Neuron model}
 Let's define the dynamic of the process. It is a recurrent neural plastic network with Poisson neurons in interaction. Each neuron jumps with an inhomogeneous rate between two states: 0 and 1. This rate depends on the network state and the weights matrix:
\begin{align}\label{jump}
  	\mathrm{0} \xrightleftharpoons[\beta]{\alpha_i(W_t,V_t)} \mathrm{1}						
\end{align}
Where $\alpha_i$ is given by $\xi_i:\reels\mapsto\reels_*^+$ 
%\eot[noinline]{Conflit de notation avec \(S_t\)}
 bounded, positive and nondecreasing:
% ($\xi_i$ is in general a sigmo\"{\i}de function), and $\alpha_m>0$ such that:
 \begin{align}\label{neuron-rate}
 	\alpha_i(W_t,V_t)=\xi_i\left( \sum^N_{j=1}W_t^{ji}V_t^j \right)
 \end{align}

As the neuron activity is never null, we will consider that for all $i$, $\inf_{x\in \reels}\xi_i(x)\geq \alpha_m>0$. Hence, $\alpha_i$ is uniformly bounded in $w$ and $v$ for all $i$:
	\[0<\alpha_m= \min_i\left(\inf_{x\in \reels}\xi_i(x)\right)\leq \alpha_i(w,v)\leq \alpha_M=\max_i\left(\sup_{x\in \reels}\xi_i(x)\right)\]

\subsubsection*{Plasticity rule}
The basic idea of STDP is that of the Hebb's law (1949):\\
\textquotedblleft \textit{When an axon of cell A[...] repeatedly or persistently takes part in firing (a cell B), [...]A's efficiency, as one of the cells firing B, is increased}\textquotedblright~\cite{hebb_organisation_1949}.\\
STDP is a bit more complex as it completes this law with the possibility for weights to decrease when they are decorrelated.
%\begin{figure}[h!]
%	\centering
%	\includegraphics[scale=0.25]{LTDP.png}
%	\caption{ Explaining scheme for STDP. When the spike of neuron $i$ induces a spike of neuron $j$, the weight $w^{i->j}$ increases. In the opposite case, it decreases.}
%	\label{STDP}
%\end{figure}

%In STDP models, the way synaptic weights change depends on the inter-spiking times $\Delta t$.
%\begin{figure}[h!]
%	\centering
%	\includegraphics[scale=0.25]{courbe-exp.png}
%	\caption{Classic STDP curve~\cite{izhikevich_relating_2003}}
%	\label{STDP-curve}
%\end{figure}

%This curve approximates datas found by Bi \& Poo in~\cite{bi_synaptic_1998} for dispersed hypocampal cultured neurons. Of course such a curve is not unique and there exist many of different forms depending on synapses but also on other parameters such as for instance the frequency of correlated spikes~\cite{buchanan_activity_2010}. In our applications, we use classic STDP curve but our results hold for any curve.\\
We expose our plasticity model through an example. First, weights can change only when a neuron spikes that we define as the jump from 0 to 1 (we could have chosen from 1 to 0 
%{\color{black}[dire ce que cela change dans discussion!]})
. So suppose the neuron $i$ spikes at time $t$. Then, weights related to this neuron, that is to say $W_t^{ji}$ and $W_t^{ij}$ for all $j\neq i$, have a certain probability to jump. This differs from models we can find in the literature for which weights' jumps are systematic but small~\cite{kempter_hebbian_1999,abbott_synaptic_2000,morrison_phenomenological_2008}. Here, the jump is not small but happens with a small probability: $W_t^{ji}$ has probability $p^+(S_t^j)$ to increase and $W_t^{ij}$ decrease with probability $p^-(S_t^j)$. These probabilities depends on the inter-spiking times given by $S_t^j$:\\
%\romain[inline]{Il n'y a pas de $S_t$ dans les probas de la figure}
\begin{figure}[h!]
	\centering
	\includegraphics[scale=0.4]{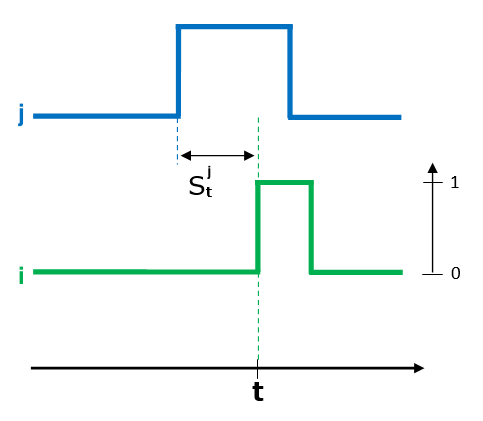}\hspace{3em}
	\includegraphics[scale=0.25]{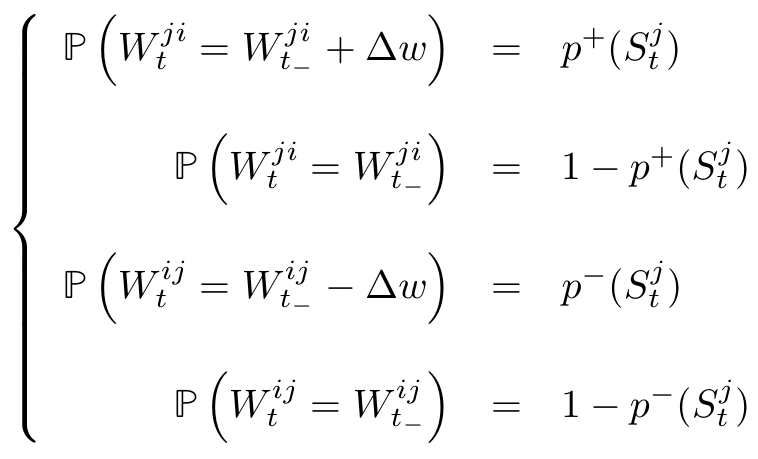}
	\caption{\small Dynamics of neurons $i$ and $j$over time, and the corresponding probability of jump for weights}
\end{figure}
\newpage
%with a probability which depends on inter-spiking times. Thus, weights jump of $+/-\ \Delta w$ with a stochastic STDP rule that is as follows. , then weights related to it, so $W_t^{ji}$ and $W_t^{ij}$, has a certain probability to change. As $W^{ij}$ refers to the influence of the $i^{th}$ neuron on the $j^{th}$, when neuron $j$ spikes    This probability depends on when other neurons emitted their last spike (which correspond to $S_t^i$).  and we make weights related to the neuron $j$ change with  $\forall i\neq j$:
%\\
%\begin{align*}
%\left\{
%\begin{array}{r c l}
%\proba\left(W^{ji}_t=W^{ji}_{t_-}+\Delta w\right)&=&p^+(S^j_t)
%\\
%\\
%\proba\left(W^{ji}_t=W^{ji}_{t_-}\right)&=&1-p^+(S^j_t)
%\\
%\\
%\proba\left(W^{ij}_t=W^{ij}_{t_-}-\Delta w\right)&=&p^-(S^j_t)
%\\
%\\
%\proba\left(W^{ij}_t=W^{ij}_{t_-}\right)&=&1-p^-(S^j_t)
%\end{array}
%\right.
%\end{align*}
As the classic STDP curve, found by Bi\&Poo~\cite{bi_synaptic_1998}, suggests it, we take the following probability functions in our examples, with $0<A_+,\ A_-\leq 1$ and $\tau_+,\ \tau_->0$:
\begin{align}\label{proba-poids}
p^+(s)=A_+e^{-\frac{s}{\tau_+}}\ \ \text{and}\ \ p^-(s)=A_-e^{-\frac{s}{\tau_-}} 
\end{align}
\begin{remarque}
	{\it By definition of $E_1$ and $\alpha_i$, we study excitatory neurons. {\color{black} We see at the end how to extend our results to inhibitory-excitatory neurons.} Also, we remark that $W_t^{ii}$ stays constant and as $W_0^{ii}=0$ for all $i$, $W_t^{ii}=0$ for all $t$. {\color{black} We will discuss this assumption later on. Finally, $(S_t)_{t\geq 0}$ is crucial for our process to be Markovian.}}
\end{remarque}
%Hence, the inter-spiking times $(S_t)$ take part in the plasticity rule as it allows to keep in memory the time duration from the last spike. Indeed, STDP rules depend precisely on these times.

\subsubsection*{Generator of the process}
Now we know how the process works, we can write its infinitesimal generator. To do so, we need the following notations. We denote by $G_i^w$ all reachable weights after a spike of neuron $i$ 
%\eot[]{A la page d'avant, c'était le neurone \(j\) qui spikkait. Ce serait plus confortable à lire si tu gardais la même lettre.}
 while the current weight is $w\in E_1$. Thus:
%\romain[inline]{Tu ne veux pas dire $w+\Delta w(\gamma-\zeta)$? Il y a une erreur de signe je pense. Il vaux mieux utiliser les notations $\zeta_p$ (potentiation) et $\zeta_d$}

%\romain[inline]{Une distribution de sauts en $W$ complique beaucoup les choses?}
\[G_i^w=\left\{w+\Delta w 
\left(\begin{bmatrix}
0 & \ldots & 0 & \ & 0 & \ldots & 0\\
\vdots & \ & \vdots & \ &\vdots & \ & \vdots \\ \\
\vdots & \ & \vdots & \vec{\zeta}_p &\vdots & \ & \vdots \\ \\
\vdots & \ & \vdots & \ &\vdots & \ & \vdots \\
0 & \ldots & 0 & \ & 0 & \ldots & 0
\end{bmatrix}
-	
\left.\underbrace{\begin{bmatrix}
	0 & \ldots & 0\\
	\ & \vdots  \\
	0 & \ldots & 0\\
	\ & \vec{\zeta}_d \\
	0 & \ldots & 0\\
	\ & \vdots \\
	0 & \ldots & 0
	\end{bmatrix}}_{N\times N\  matrix}\right.\right),\ (\vec{\zeta_p},\vec{\zeta_d})\in F^w_i\right\}\]

Where 
\[F^w_i=\left\{(\vec{\zeta_p},\vec{\zeta_d}),\vec{\zeta_d}=\left[\zeta_d^1,...,\zeta_d^N\right], \vec{\zeta_p}=
\begin{bmatrix}
\zeta_p^1 \\
\vdots \\
\zeta_p^N
\end{bmatrix},\zeta_d^j,\zeta_p^j \in \{0,1\},\zeta_d^i=\zeta_p^i=0\text{ and }\zeta_d^j=0\text{ if } w^{ij}=\Delta w\right\}\]
We call $Z_p$\label{def_Z} (respectively $Z_d$) the matrix associated to the vector $\vec{\zeta_p}$ (respectively $\vec{\zeta_d}$). As each weight jumps independently whenever a neuron $i$ spikes, we can decompose the probability of jumping to a certain state as the product of probabilities to jump or not for each weights. We want to compute $\phi^i(s,\tilde{w},w)$, the probability of jumping in a given $\tilde{w}\in G_i^w$ knowing the neuron $i$ spikes. Let $\tilde{w}=w+\Delta w (Z_p+Z_d)$, 
%\eot[]{\(\Delta w(Z_p - Z_d)\) ?}
the probability for $w^{ji}$ to increase ($\zeta_p^j=1$) 
%\eot[]{\(\zeta_p^j=1\)}
is $p_+(s_j)$ when the probability to stay the same ($\zeta_p=0$) is $\left(1-p_+(s_j)\right)$, for all $j\neq i$. This will appear as $\zeta_p^j\ p^+(s_j)+(1-\zeta_p^j)(1-p^+(s_j))$ in $\phi^i(s,\tilde{w},w)$:
\begin{align}\label{def-proba-saut}
\phi^i(s,\tilde{w},w)=\Phi^i(s,\vec{\zeta_p},\vec{\zeta_d})=\prod_{j\neq i}\left[\zeta_p^j\ p^+(s_j)+(1-\zeta_p^j)(1-p^+(s_j))\right]\left[\zeta_d^j\ p^-(s_j)+(1-\zeta_d^j)(1-p^-(s_j))\right]
\end{align}
%\romain[inline]{Pourquoi cette forme pour la proba, explique?}
Therefore, we can write the generator $(\mathcal{C},D(\mathcal{C}))$ of the all process $(W_t,S_t,V_t)_{t\geq 0}$ where $D(\mathcal{C}) \subset C_b(E)$
%$D(\mathcal{C}) \subset C_b^1(E)=\{f\in C_{b}(E)\ and\ \partial_{s_i}f\in C_{b}(E),\ \forall i\}$
and $\mathcal{C}$ given $\forall f\in D(\mathcal{C})$ :
\begin{align*}
\mathcal{C}f(w,s,v) &= \sum_i\delta_1(v^i)\beta[f(w,s,v-e_i)-f(w,s,v)]\\
&+\sum_i\alpha_i(w,v)\delta_0(v^i)\left(\sum_{\tilde{w} \in G_i^w}(f(\tilde{w},\ s-s_ie_i,\ v+e_i)-f(w,\ s,\ v))\phi^i(s,\tilde{w},w)\right)\\
&+\sum_{i=1}^N\partial_{s_i}f(w,s,v)
\end{align*}
Or
\begin{align*}
\mathcal{C}f(w,s,v) &=\underbrace{\sum_i\delta_1(v^i)\beta[f(w,s,v-e_i)-f(w,s,v)]}_{\mathcal{B}_{\downarrow}f(w,s,v)}\\
&+\underbrace{\sum_i\phi^i(s,w,w)\alpha_i(w,v)\delta_0(v^i)\left(f(w,\ s-s_ie_i,\ v+e_i)-f(w,s,v)\right)}_{\mathcal{B}_{\uparrow}f(w,s,v)}\\
&+\underbrace{\sum_{i=1}^N\partial_{s_i}f(w,s,v)}_{\mathcal{B}_{tr}f(w,s,v)}\\
&+\sum_i\alpha_i(w,v)\delta_0(v^i)\left(\sum_{\tilde{w} \in G_i^w, \tilde{w}\neq w}(f(\tilde{w},\ s-s_ie_i,\ v+e_i)-f(w,\ s,\ v))\phi^i(s,\tilde{w},w)\right)
\end{align*}
Written in this form, the generator shows two different dynamics which are related: the weights dynamic and the network, inter-spiking time dynamics. As we know that synaptic weights dynamics are slow compare to the network dynamics ($(S_t,V_t)_{t>0}$ change fast compare to  $(W_t)_{t>0}$), this means that for all $i$:
\[\sum_{\tilde{w} \in G_i^w, \tilde{w}\neq w}\phi^i(s,\tilde{w},w)\ll\phi^i(s,w,w)\]
Typically, \(\sum_{\tilde{w} \in G_i^w, \tilde{w}\neq w}\phi^i(s,\tilde{w},w)=O(\epsilon)\) 
%\eot[]{\(\in O(\epsilon)\) plutôt que \(=\epsilon\), non ?}
 and
\( \phi^i(s,w,w)=1-O(\epsilon)\). This time scale difference is studied in section~\ref{slow-fast-sec} while the study of the fast part of the process is done in section~\ref{inv-mes-sec}. This process is given by the generator $\mathcal{B}: D(\mathcal{B})\subset C_b(E)\rightarrow C_b(E)$:
\begin{align}\label{def-B}
\mathcal{B}=\mathcal{B}_{tr}+\mathcal{B}_{\downarrow}+\mathcal{B}_{\uparrow}
\end{align}

\section{Derivation of the weight equation}
\subsection{Invariant measure of the fast processes}\label{inv-mes-sec}
%\eot[inline]{Je suggère que tu dises dès le début que \(w\) sera fixé dans toute la partie.}

In this section, $W_t=W_0=w\in E_1$ is fixed. We are interested in proving:
\begin{theoreme}\label{theo-inv-mes}
	For all $w\in E_1$, the process $(S_t,V_t)_{t\geq 0}$ with generator $\mathcal{B}_{w}$ mapping $D(\mathcal{B})$ into $ C_b(E_2)$, defined $\forall f\in D(\mathcal{B})$ as:
	\begin{align}\label{def-B_w}
	\mathcal{B}_{w}f(s,v) &=\sum_i\delta_1(v^i)\beta[f(s,v-e_i)-f(s,v)]\\
	&\ \ \ +\sum_i\alpha_i(w,v)\delta_0(v^i)\left(f(s-s_ie_i,\ v+e_i)-f(s,v)\right)\\
	&\ \ \ +\sum_{i=1}^N\partial_{s_i}f(s,v)
	\end{align}
	has a unique invariant measure.
\end{theoreme}  

\noindent This aim enters in a bigger ambition to analyse the total process $(W_t,S_t,V_t)_{t\geq 0}$ on two different time scales. Indeed, in the limit where the plasticity is infinitely slow, it stays constant so $\phi_i(s,w,w)= 1$, and then for all $f\in D(\mathcal{B}_w)$, $\mathcal{B}_{w}f(s,v)=\mathcal{B}f(w,s,v)$. This analysis enables us to show in section~\ref{slow-fast-sec} that, on the slow time scale of plasticity, $(W_t)_{t\geq 0}$ behaves simply against the invariant measure of $(S_t,V_t)^w_{t\geq 0}$. In the following, we omit the dependence on $w$ in the notation of processes only and we use $(S_t,V_t)_{t\geq 0}$ instead of $(S_t,V_t)^w_{t\geq 0}$.
%\eot[]{indexée par \(W_0\)...}

In a first subsection we show existence of an invariant measure of the process $(S_t,V_t)_{t\geq 0}$ and then its uniqueness in the next subsection. We start with some notations.

\subsubsection*{Notations}

Let $X_t=(S_t,V_t)$ with $S_t \in \reels_+^N$ and $V_t \in I=\{0,1\}^N$. The process is then the same as the one defined before with a fixed matrix of weights $w$. Each $X_t^i=(S_t^i,V_t^i)\in \reels_+\times \{0,1\}$, for $i\in [\![1,N]\!]$, follows the same kind of process: the discrete variable $V_t$ jumps with a total rate $\sum_j\left(\alpha_j(w,v)\delta_0(v^j)+\beta\delta_1(v^j)\right)$ when $V_t=v$. Between these jumps, the continuous part $S_t$ will grow linearly with a slope of 1 ($\frac{dS_t}{dt}=1$) except when $V_t^i$ jumps from 0 to 1 at time $t_0$, then the continuous part restarts from 0, i.e. $S_{t_0}^i=0$, see $Figure~\ref{graph-process}$.

\begin{figure}[h!]
	\centering
	\includegraphics[scale=0.5]{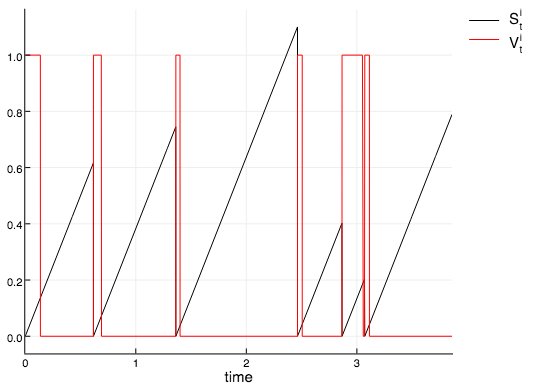}
	\caption{Graph representing the $i_{th}$ coordinates of the processes $S_t$ and $V_t$}\label{graph-process}
\end{figure}

From these notations, one can denote by $(N_t)_{t\geq 0}$ the counting process corresponding to the number of jump of the process $(V_t)_{t\geq 0}$. We can then define the processes $N_t =\sum_{i=1}^{N} N^i_t$ where $(N^i_t)_{t\geq 0}$ are counters of the number of jumps of neuron i. By definition of $\alpha_i$ , one has $N^i_t = Y_i\left(\int_{0}^{t}\alpha_i(w,V_s)ds\right)$ where $Y_i$ are independent Poisson processes of intensity 1, as in~\cite{kang_separation_2013}. Finally, we call $(P_t)_{t\geq 0}$ the transition probability of the process, $P_t$ maps $E_2\times\mathcal{B}(E_2)$ in $\reels_+$. Hence, for all $x\in E_2,\ A\in\mathcal{B}(E_2)$($\sigma$-algebra of Borel sets of $E_2$), $P_t(x,A)$ is the probability that $X_t\in A$ knowing $X_0 = x$, probability also written as $\proba_x(X_t\in A)$.
\newpage
\subsubsection{Existence using a Lyapounov function}

In this subsection, we aim at proving the following theorem:
\begin{proposition}\label{theo-inv-mes-exist}
	The process $(S_t,V_t)_{t\geq 0}$ defined in Theorem~\ref{theo-inv-mes} has at least one invariant measure of probability.
\end{proposition}
To do so, we use the following theorem, classical in theory of discrete Markov chains on any state space:
\begin{theoreme}\label{theo-Lyapounov}
	If a transition probability P is Feller and admits a Lyapunov function, then it also has an invariant probability measure.
\end{theoreme}

\begin{proof}
	A nice proof of this result can be found in the course of Martin Hairer called \textit{Ergodic Properties of Markov Processes}. {\color{black} See theorem 2 of~\cite{tweedie_invariant_1988}. Just need to show condition ($F_1$) is equivalent to our Lyapunov condition}.
%	\eot[inline]{Cite le résultat précis. As-tu une référence publiée ?}
\end{proof} 

After recalling the definitions of a Lyapunov function and a Feller process,
we find such a Lyapunov function for our process.

\begin{definition}\label{def-Lyapounov}
	Let X be a complete separable metric space and let P be a transition probability on X . A Borel measurable function $V:X\mapsto \reels_+\cup\{\infty\}$ is called a Lyapunov function for P if it satisfies the following conditions:
	\begin{itemize}
		\item [-] $V^{-1}(\reels_+)\neq \emptyset$, in other words there are some values of $x$ for which $V(x)$ is finite.
		\item [-] For every $c\in \reels_+$, the set $V^{-1}(\{x \leq c\})$ is compact.
		\item [-] There exists a positive constant $\gamma$ < 1 and a constant $C$ such that for every x such that $V(x)\neq +\infty$:
		\[\int_XV(y)P(x,dy) \leq \gamma V(x) + C\]
	\end{itemize}
\end{definition}

\begin{definition}
	We say that a homogeneous Markov process with transition operator P is Feller if Pf is continuous whenever f is continuous and bounded. It is strong Feller if Pf is continuous whenever f is measurable and bounded.
\end{definition}
We emphasize that previous definitions and theorem are given for Markov chains and not processes. The following proposition links them.
\begin{proposition}\label{prop-Lyapunov}
	Let $(P_t)_{t\geq 0}$ be a Markov semigroup over X and let $P = P_T$ for some fixed $T > 0$. Then, if $\mu$ is invariant for P, the measure $\overset{\wedge}{\mu}$ defined by:
	\[\overset{\wedge}{\mu}(A) = \frac{1}{T} \int_0^T P_t \mu(A)dt, \ \ \ \ \forall\ A\in \mathcal{B}(E_2) \]
	is invariant for $(P_t)_{t\geq 0}$.
\end{proposition}

\begin{proof}
	\begin{align*}
	P_t\overset{\wedge}{\mu} &= P_t\left(\frac{1}{T} \int_0^T P_s \mu\ ds\right) = \frac{1}{T} \int_0^T P_tP_s \mu\ ds = \frac{1}{T} \int_0^T P_{t+s}\ \mu\ ds\\
	&= \frac{1}{T} \int_t^{T+t} P_{s}\ \mu\ ds= \frac{1}{T} \left(\int_t^{T} P_{s}\ \mu\ ds + \int_T^{T+t} P_{s}\ \mu\ ds\right)\\
	&= \frac{1}{T} \left(\int_t^{T} P_{s}\ \mu\ ds + \int_0^{t} P_sP_T\ \mu\ ds\right)=\frac{1}{T} \int_0^T P_s \mu\ ds =\overset{\wedge}{\mu}
	\end{align*}
\end{proof}

Hence, we want to apply theorem~\ref{theo-Lyapounov} to the transition probability $P_T$ extracted from $(P_t)_{t\geq 0}$ for some fixed $T > 0$. 
%\eot[inline]{Reformule le paragraphe suivant. Coupe en plusieurs phrases en mettant les choses dans l'ordre.}
To do so, we show that for $T>0$ any given time, $V$ defined as $V(x) = s_1+s_2+...s_N\ \ \forall\ x = (s,v)\in E_2$ is a Lyapunov function for $P_T$. Then we use theorem 27.6 of the Davis' book~\cite{davis_markov_1993} to prove $P_T$ is Feller. We conclude on the existence of the invariant measure of probability for $P_T$ and thus for $(P_t)_{t\geq 0}$ thanks to proposition~\ref{prop-Lyapunov}.
%:
%\begin{theoreme}\label{feller-theo}
%	\eot[inline]{ \(t_*(x)\) est bien défini ?}
%	If $t_*(x) = \infty$ for all $x\in E$, $\lambda \in C_b(E)$ and the function $x\rightarrow Qf(x)$ is continuous for $f\in C_b(E)$, then $(X_t)$ is a Feller process.
%\end{theoreme}
\\

After these definitions and notations, let's prove the process $(X_t)_{t\geq 0}$ has at least one invariant measure $\pi$, i.e. $X_0\sim \pi\ \Rightarrow\ \forall \ t\geq 0, X_t\sim \pi$ or more formally, $\forall A \in \mathcal{B}(E_2)$:
\begin{align}\label{def_pi}
\int_{E_2}P_t(x,A)\pi(dx)=\pi(A)
\end{align}

\subsection*{Existence}
\begin{assumption}\label{ass-jump-rate}
	$\exists\ \alpha_m,\ \alpha_M\in\reels_+$ such that $ \forall\ v\in I,\ w\in E_1$:\\
	\[ 0<\alpha_m\leq\alpha_i(w,v){\color{black},\beta}\leq \alpha_M<\infty\]
\end{assumption}

\begin{proposition}\label{prop-lyap}
	 With assumption~\ref{ass-jump-rate}, for any $T>0$, $V(x)=s_1+...+s_N$ is Lyapunov for $P_T$ with constants $C=NT$ and $\gamma=
%	 \left(1-\left(\sum_{k=2}^{+\infty}\frac{(\alpha_mT)^k}{k!}\right)^Ne^{-N\alpha_mT}\right)
	\mathbb{P}_x(\exists i:\ N_T^i <2)<1$, $\forall\ x\in E_2$.
\end{proposition} 

\begin{proof}
	The main idea is to use the fact that $S^i_t$ values return to 0 whenever neuron $i$ jumps from 0 to 1. Hence, as neurons have only two states, if $N_T^i\geq 2$, neuron $i$ has jumped at least one time from 0 to 1 between \(0\) and $T$. Therefore, decomposing possible events we get:
	\[
	V(X_T) \leq (V(x) + NT) \mathbbm{1}_{\{\exists i N^i_T < 2  \}} + NT
	\mathbbm{1}_{\{\forall i N^i_T \geq 2  \}}  
	\]
	So
	\[
	\mathbb{E}_xV(X_T) \leq NT + V(x) \underbrace{\mathbb{P}_x(\exists i:\ N_T^i <2)}_{<1}
	\]
\end{proof}
%\eot[inline]{Je trouve la preuve précédente un peu longue.
%Tu écris:
%\[
%V(X_T) \leq (V(x) + NT) \mathbbm{1}_{\{\exists i N^i_T < 2  \}} + NT
% \mathbbm{1}_{\{\forall i N^i_T \geq 2  \}}  
%\]
%ce qui te donne immédiatement 
%\[
%\mathbb{E}V(X_T) \leq NT + V(x) \mathbb{P}(\exists i N_T^i <2)
%\]
%et je crois que c'est fini.
%}

Furthermore, one can show the process $(S_t,V_t)$ is Feller thanks to 
%the same computation as in the appendix remark~\ref{feller} or 
Davis' book~\cite{davis_markov_1993}:
\begin{proposition}\label{prop-feller}
	$(S_t,V_t)$ is Feller.
\end{proposition}
\begin{proof}
	First, we define a distance $\rho$ such that $(E_2,\rho)$ is a metric space, locally compact. Such a distance is proposed in~\cite{davis_markov_1993} page 58:
	\begin{align}\label{rho-def}
	\begin{split}
	\forall x=(s_x,v_x),\ y=(s_y,v_y)\in E_2:\qquad \rho(x,y)=
	\left\{
	\begin{array}{r c l}
	1\quad \mbox{ if } v_x\neq v_y
	\\
	\dfrac{2}{\pi}\max_{\{1 \leq i \leq N\}} \tan^{-1}(\lvert s_x^i-s_y^i\rvert) \quad \mbox{ if } v_x= v_y
	\end{array}
	\right.
	\end{split}
	\end{align}
	\noindent We need this kind of norm because if we take for instance the euclidean distance $\rho(x,y)=\lVert s_x-s_y\rVert_2$, we can have $\rho(x,y)=0$ and $x\neq y$ as soon as $s_x=s_y$ and $v_x\neq v_y$.
	
	Then, we want to apply theorem 27.6 of~\cite{davis_markov_1993}. We define $t_*(x)$ as
	\[t_*(x)=\{\text{time to hit the boundary of $E_2$ leaving from x and following the flow on s}\}\]
	$t_*(x)=+\infty$ as the only boundary is for $x=(0,v)$ which is never reached because $S_t$ increases toward infinity following the flow.\\ Moreover, we define the total jump rate $\lambda(x)=\sum_j\left(\alpha_j(w,v)\delta_0(v^j)+\beta\delta_1(v^j)\right)=\lambda(v)$. 
%	\eot[inline]{Pas favorable à écrire \(\lambda(v)\).}	
	Thus, as $\lambda$ is bounded by assumption~\ref{ass-jump-rate} and it only depends on $v$, as soon as $\rho(x,y)<1$, $v_x=v_y$ so $\lambda(x)=\lambda(y)$, hence $\lambda \in C_b(E)$.\\ 
	Finally, we define $Q$ as  
	\[Q\left(\{((s-\delta_0(v^i)s_ie_i,v+e_i)\},(s,v)\right)=\frac{\alpha_i(w,v)\delta_0(v^i)+\beta\delta_1(v^i)}{\lambda(v)}\]
	and show it is continuous for $f\in D(\mathcal{B}_w)$.
%	\eot[inline]{Nouvelle notation ?}
	 Indeed, let $f\in D(\mathcal{B}_w)$, if $\rho(x,y)\leq \eta<1$:
	\begin{align*}
		\left| Qf(x)-Qf(y) \right|&=\left| \sum_if(s_x,v+e_i)\frac{\alpha_i(w,v)\delta_0(v^i)+\beta\delta_1(v^i)}{\lambda(v)}-\sum_if(s_y,v+e_i)\frac{\alpha_i(w,v)\delta_0(v^i)+\beta\delta_1(v^i)}{\lambda(v)} \right|\\ &\leq N\sup_i\left| f(s_x,v+e_i)-f(s_y,v+e_i)\right|
	\end{align*}
	Then, choosing $\eta$ such that $\sup_{v'\in I}|f(s_x,v')-f(s_y,v')|\leq \frac{\epsilon}{N}$ (possible as $f\in D(\mathcal{B}_w)\subset C_b(E_2)$) we have for all $\epsilon>0$, $\exists \eta>0$ such that:
	\begin{align*}
	\rho(x,y)\leq \eta\  \Rightarrow\ |Qf(x)-Qf(y)|\leq \epsilon 
	\end{align*}
%	\eot[inline]{Davis autorise que l'on n'utilise pas la même distance au départ et à l'arrivée ?}
	Thus, $x\rightarrow Qf(x)$ is continuous for $f\in D(\mathcal{B}_w)$. We can apply theorem 27.6 of Davis' book~\cite{davis_markov_1993} which ends the proof.
\end{proof}
We can now prove theorem~\ref{theo-inv-mes-exist}:
\begin{proof}
	Proposition~\ref{prop-lyap} and Proposition~\ref{prop-feller} allows to apply  Theorem~\ref{theo-Lyapounov} and thus conclude on the existence of an invariant measure of probability for $(S_t,V_t)$.
\end{proof}
 In  the following, we show that such a measure is unique.
 
\subsubsection{Uniqueness through Laplace transform}
We now want to show this process has a unique invariant measure of probability $\pi$. To do so, we find the possible Laplace transforms of the invariant measures of the process. We prove such Laplace transforms satisfy an equation with a unique solution. By uniqueness of the Laplace transform of a measure, we deduce the result we want.

In the following, we use an equivalent definition of invariant measures which makes use of the generator $(\mathcal{B}_{w},D(\mathcal{B}_{w}))$ of the process, see proposition 34.7 in~\cite{davis_markov_1993}.

\begin{proposition}\label{prop-inv-mes-gen}
	Let $(T_t)_{t\geq 0}$ be a semigroup on $F$, a Banach space,  associated to a Markov process $(X_t)_{t\geq 0}$. We note, $(\mathcal{B}_{w},D(\mathcal{B}_{w}))$ its generator and we assume $D(\mathcal{B}_{w})$ is separating. Then, $\pi$ is an invariant measure if and only if $\forall f\in D(\mathcal{B}_{w})$, \begin{align}\label{def_pi}
		\int_{E_2}\mathcal{B}_{w}fd\pi=0
	\end{align}
\end{proposition}
We remind us what is a separating class of functions:
\begin{definition}
	A class of functions $\mathcal{D}\in B(E_2)$ (measurable and bounded function on E) is said to be separating if for probability measures $\mu^w_1$ and $\mu^w_2$ on $E_2$, $\mu^w_1 = \mu^w_2$ whenever $\int_{E_2}fd\mu^w_1 = \int_{E_2}fd\mu^w_2$ for all $f \in \mathcal{D}$.
\end{definition} 

In what follows, domains of generators will always be separating as showed in the proposition 34.11 of~\cite{davis_markov_1993}.

\subsection*{Uniqueness}
We invite you to have a look to the appendix~\ref{uniqueness_dim2} to have a better view on the following computations.

\begin{proposition}
	Assume the process $(X_t)_{t\geq 0}$ in dimension N has at least one invariant measure of probability $\pi^w$. Then it is unique.
\end{proposition}

\begin{proof}
	Let start with some notations:
	\begin{align}\label{not-lap}
	\begin{split}
		& I = \{0,1\}^N\ \text{and}\ E_2 = \reels_+^N \times I\\
		& \forall (s,v)\in E_2,\ s = (s^1, ..., s^N)\in \reels_+^N\ \text{and}\ v = (v^1, ..., v^N)\in I\\
		& \ e_i = (0,...,0,\underbrace{1}_{i},0,...,0)\\
		& \mathcal{B}_I=(\vec{v}_1, ..., \vec{v}_{2^N})\ \text{an enumeration of}\ I\ \text{s.t.}\ k\geq l \Rightarrow \sum_{i=1}^{N}v_k^i\geq \sum_{i=1}^{N}v_l^i\\
		& |\lambda|=\sum_{i=1}^{2^N} \lambda_i
	\end{split}
	\end{align}
	The jump process alone $(V_t)_{t\geq 0}$ has a unique invariant measure $\mu^w=(\mu^w_1, ..., \mu^w_{2^N})\in \reels_+^{2^N}$. Indeed, as each neuron is connected to each other, $(V_t)_{t\geq 0}$ is irreducible. As its state space is finite, the process is also positive recurrent so it has a unique invariant probability measure $\mu^w$ by theorem1.7.7 in~\cite{norris_markov_1998}. Moreover, as each state is positive recurrent, $\mu^w_v>0,\ \forall v\in I$. In particular, this measure satisfies $\sum_{k=1}^{2^N}\mathcal{B}_0g(\vec{v}_k)\mu^w_k = 0$, where $\mathcal{B}_0$ is the generator of $(V_t)_{t\geq 0}$ and for functions $g$ $I$-measurable:
	\begin{align}\label{def_B_0}
	\mathcal{B}_0g(v) = \mathcal{B}_{w}g(s,v)=\sum_{i=1}^{N}\beta\delta_1(v^i)[g(v-e_i)-g(v)]+\alpha_i(w,v)\delta_0(v^i)\left[g(v+e_i)-g(v)\right]
	\end{align}
	Hence, with $g(v) = \mathbb{1}_{\vec{v}_j}(v)$ we get $\forall j\in [\![1,2^N]\!]$:
	\begin{align}\label{rel-jump}
		\sum_{k=1}^{2^N}\mathcal{B}_0g(\vec{v}_k)\mu^w_k = \sum_{k=1}^{2^N}\mu^w_k\sum_{i=1}^{N}(\beta\delta_1(v_k^i)[\mathbb{1}_{\vec{v}_j}(v_k-e_i)-&\mathbb{1}_{\vec{v}_j}(\vec{v}_k)]+\alpha_i(\vec{v}_k)\delta_0(v_k^i)\left[\mathbb{1}_{\vec{v}_j}(\vec{v}_k+e_i)-\mathbb{1}_{\vec{v}_j}(\vec{v}_k)\right])=0 \nonumber\\
		&\Leftrightarrow\nonumber\\
		\sum_{k=1,k\neq j}^{2^N}\mu^w_k\sum_{i=1}^{N}[\beta\delta_1(v_k^i)\mathbb{1}_{\vec{v}_j}(\vec{v}_k-e_i)+\alpha_i(\vec{v}_k)\delta_0(v_k^i)&\mathbb{1}_{\vec{v}_j}(\vec{v}_k+e_i)]=\mu^w_j\sum_{i=1}^{N}\beta\delta_1(v_j^i)+\alpha_i(\vec{v}_j)\delta_0(v_j^i)
	\end{align}
%	Which is equivalent in 2D to the relation~\eqref{murel2}.\\
	
	We can then write the system satisfied by Laplace transforms of invariant probability measures of the process $(S_t,V_t)_{t\geq 0}$. We call $\pi^w$ one of them. First we can decompose $\pi^w$ as: 
	\begin{align}\label{pi-decomp}
		\pi^w(ds,v) = \sum_{k=1}^{2^N}\pi^w_{\vec{v}_k}(ds)\mu^w_k\mathbb{1}_{\vec{v}_k}(v)
	\end{align}
	In what follows, for the sake of simplicity, we note $\pi^w_k$ for $\pi^w_{\vec{v}_k}$.\\
	From proposition~\ref{prop-inv-mes-gen}, $\forall f\in D(\mathcal{B}_{w})$:
	\begin{align}\label{mes-gen-N}
		\sum_{k=1}^{2^N}\int_{s\in \reels_+^N}\mathcal{B}_{w}f(s,\vec{v}_k)\mu^w_k\pi^w_k(ds)=0
	\end{align}
	Where $(\mathcal{B}_{w},D(\mathcal{B}_{w}))$ is the generator of the process $(X_t)_{t\geq 0}$~\eqref{def-B_w} 
%	and $D(\mathcal{B}_{w})=\left\{f\in \mathcal{B}_{w}(E_2),\ \lim_{t\rightarrow 0}\frac{\esp_x[f(X_t)]-f(x)}{t}\ exists\ \forall\ x\in E_2\right\}$
	. As we are interested in finding the Laplace transform of $\pi^w$ we take $f(s,v) = e^{-\vec{\lambda}.\vec{s}}g(v)$. First we compute $\mathcal{B}_{w}f$:
	\begin{align}\label{gen-N}
	\mathcal{B}_{w}f(s,v) =&\sum_{i=1}^N\beta\delta_1(v^i)[e^{-\vec{\lambda}.\vec{s}}g(v-e_i)-e^{-\vec{\lambda}.\vec{s}}g(v)] \nonumber \\
	&+\sum_{i=1}^N\alpha_i(w,v)\delta_0(v^i)\left[e^{-\vec{\lambda}.(\vec{s}-s^i\vec{e_i})}g(v+e_i)-e^{-\vec{\lambda}.\vec{s}}g(v)\right]\nonumber\\
	&-\underbrace{(\sum_{i=1}^N\lambda_i)}_{\lvert\lambda\rvert}\ e^{-\vec{\lambda}.\vec{s}}g(v)
	\end{align}
	So in \eqref{mes-gen-N} we get:
	\begin{align}\label{lap-gen}
		\sum_{k=1}^{2^N}\int_{s\in \reels_+^N}&\mathcal{B}_{w}f(s,\vec{v}_k)\mu^w_k\pi^w_k(ds)\nonumber\\
		=&\sum_{k=1}^{2^N}\left[\left(\sum_{i=1}^N\beta\delta_1(v_k^i)[g(\vec{v}_k-e_i)-g(\vec{v}_k)]-\alpha_i(\vec{v}_k)\delta_0(v_k^i)g(\vec{v}_k)\right)-\lvert\lambda\rvert g(\vec{v}_k)\right]\mu^w_k\underbrace{\int_se^{-\vec{\lambda}.\vec{s}}\pi^w_k(ds)}_{\mathcal{L}(\pi^w_k)(\lambda)}\nonumber\\
		&+\sum_{k=1}^{2^N}\left[\sum_{i=1}^N\alpha_i(\vec{v}_k)\delta_0(v_k^i)g(\vec{v}_k+e_i)\underbrace{\int_se^{-\vec{\lambda}.(\vec{s}-s^i\vec{e_i})}\pi^w_k(ds)}_{\mathcal{L}(\pi^w_k)(\widehat{\lambda}_i)}\right]\mu^w_k=0
	\end{align}
	Where $\widehat{\lambda}_i=(\lambda_{1},...,\lambda_{i-1},0,\lambda_{i+1},...,\lambda_{N})$. We first show recursively that we can express $\mathcal{L}(\pi^w_k)(\lambda)$ in function of linear combinations of  $\mathcal{L}(\pi^w_l)(\check{\lambda}_l)$ where $\check{\lambda}_l = (0, ..., 0, \lambda_l, 0, ..., 0)$: step 1. Second, we show there exists $D(\check{\lambda}_l)$ invertible such that:
	\[D(\check{\lambda}_l)\begin{bmatrix}
	\mathcal{L}(\pi^w_1) (\check{\lambda}_l)\\
	\vdots\\
	\mathcal{L}(\pi^w_{2^N}) (\check{\lambda}_l)
	\end{bmatrix}
	=\Lambda^{(l)},\ \ with\ \Lambda^{(l)}\in \reels^{2^N}\ a\ constant\ vector\]
	Where $\check{\lambda}_l = (0, ..., 0, \lambda_l, 0, ..., 0)$: step 2. Finally, we conclude on the uniqueness of the solution $\mathcal{L}(\pi^w_k)(\lambda)$ as a linear combination of $\mathcal{L}(\pi^w_k)(\check{\lambda}_l)$.\\
	
	\subsubsection*{Step 1}
	First, we express the $\mathcal{L}(\pi^w_k)(\lambda)$ in function of the $\mathcal{L}(\pi^w_l)(\widehat{\lambda}_i)$. In particular, we find $\Gamma(\lambda): \reels_+^N\rightarrow M_{2^N}(\reels)$ and $\Lambda(\lambda): \reels_+^N\rightarrow \reels^{2^N}$, for which $\Lambda_j(\lambda)$ depends only on linear combination of $\mathcal{L}(\pi^w_l)(\widehat{\lambda}_i)$ where $i\in [\![1,N]\!]$ and $l\in [\![1,2^N]\!]$, such that:
	\begin{align}\label{matrix-lap}
	\Gamma(\lambda)\begin{bmatrix}
	\mathcal{L}(\pi^w_1) (\lambda)\\
	\vdots\\
	\mathcal{L}(\pi^w_{2^N}) (\lambda)
	\end{bmatrix}
	=\Lambda(\lambda)
	\end{align}
	To do so, we take $g(v) = \mathbb{1}_{\vec{v}_j}(v)$ in \eqref{lap-gen} and find $\Gamma \text{ and } \Lambda$ :
	\begin{align}\label{def-Gam-Lam}
	\begin{split}
	\sum_{k=1}^{2^N}\mathcal{L}(\pi^w_k)(\lambda)&\underbrace{\left[\left(\sum_{i=1}^N\beta\delta_1(v_k^i)[\mathbb{1}_{\vec{v}_j}(\vec{v}_k)-\mathbb{1}_{\vec{v}_j}(\vec{v}_k-e_i)]+\alpha_i(\vec{v}_k)\delta_0(v_j^i)\mathbb{1}_{\vec{v}_j}(\vec{v}_k)\right)+\lvert\lambda\rvert\mathbb{1}_{\vec{v}_j}(\vec{v}_k)\right]\mu^w_k}_{\Gamma_{jk}(\lambda)}\\
	&=\underbrace{\sum_{k=1}^{2^N}\left[\sum_{i=1}^N\alpha_i(\vec{v}_k)\delta_0(v_k^i)\mathbb{1}_{\vec{v}_j}(\vec{v}_k+e_i)\mathcal{L}(\pi^w_k)(\widehat{\lambda}_i)\right]\mu^w_k}_{\Lambda_j(\lambda)}
	\end{split}
	\end{align}
	We can remark from~\eqref{rel-jump} that:
	\begin{align*}
	\left\{
	\begin{array}{r c l}
	\Gamma_{jk}(\lambda) = 0\ \ \ \ \ \ \ \ \ \ \ \ \ \ \ \ \ \ &\forall k<j \\
	\\
	\Gamma_{jj}(\lambda) = \left[\left(\sum_{i=1}^N\beta\delta_1(v_j^i)+\alpha_i(\vec{v}_j)\delta_0(v_j^i)\right)+\lvert\lambda\rvert\right]\mu^w_j>0&
	\\
	\\
	\Gamma_{jk}(\lambda) = -\sum_{i=1}^N\beta\delta_1(v_k^i)\mathbb{1}_{\vec{v}_j}(\vec{v}_k-e_i)\mu^w_k,\ \ \ \ \ \ \ \ &\forall k>j
	\end{array}
	\right.
	\end{align*}
	So
	\begin{align}
	\begin{split}
	\Gamma_{jj}(\lambda)&=\mu^w_j\sum_{i=1}^{N}\beta\delta_1(v_j^i)+\alpha_i(\vec{v}_j)\delta_0(v_j^i)+|\lambda|\mu^w_j\\
	&=\sum_{k=1,k\neq j}^{2^N}\mu^w_k\sum_{i=1}^{N}[\beta\delta_1(v_k^i)\mathbb{1}_{\vec{v}_j}(\vec{v}_k-e_i)+\alpha_i(\vec{v}_k)\delta_0(v_k^i)\mathbb{1}_{\vec{v}_j}(\vec{v}_k+e_i)]+|\lambda|\mu^w_j\\
	&= \sum_{k=1,k\neq j}^{2^N}|\Gamma_{jk}(\lambda)|+\sum_{k=1,k\neq j}^{2^N}\mu^w_k\sum_{i=1}^{N}\alpha_i(\vec{v}_k)\delta_0(v_k^i)\mathbb{1}_{\vec{v}_j}(\vec{v}_k+e_i)+|\lambda|\mu^w_j
	\end{split}
	\end{align}
	Thus $\Gamma$ is invertible as a strictly dominant diagonal matrix as soon as $|\lambda|\geq 0$. We will use the same idea in what follows to show there is a unique way to express each $\mathcal{L}(\pi^w_m)(\lambda)$, $m\in I$, as a linear combination of terms of the family $\left(\mathcal{L}(\pi^w_k)(\check{\lambda}_l)\right)_{1\leq l\leq N,m\in I}$.\\
	
	Second, take a sequence $k_1,\ k_2,\ ...\ ,k_d \in [\![1,N]\!]$, $d\leq N-1$ and define as before $\widehat{\lambda}_{k_1...k_d}$ which checks the conditions $\widehat{\lambda}_{k_1...k_d}^{k_i}=0$. We have from~\eqref{matrix-lap}:
	\begin{align}\label{eq-hat-d}
	\Gamma(\widehat{\lambda}_{k_1...k_d})\begin{bmatrix}
	\mathcal{L}(\pi^w_1) (\widehat{\lambda}_{k_1...k_d})\\
	\vdots\\
	\mathcal{L}(\pi^w_{2^N}) (\widehat{\lambda}_{k_1...k_d})
	\end{bmatrix}
	=\Lambda(\widehat{\lambda}_{k_1...k_d})
	\end{align}
	Using~\eqref{def-Gam-Lam} we get:
	\begin{align*}
	\Lambda_j(\widehat{\lambda}_{k_1...k_d})=\sum_{k=1}^{2^N}\left[\left(\sum_{i\notin \{k_1, ... , k_d\}}\alpha_i(\vec{v}_k)\delta_0(v_k^i)\mathbb{1}_{\vec{v}_j}(\vec{v}_k+e_i)\mathcal{L}(\pi^w_k)(\widehat{\lambda}_{k_1...k_dm})\right)\right.\\
	\left. + \underbrace{\sum_{i\in\{k_1,...,k_d\}}\alpha_i(\vec{v}_k)\delta_0(v_k^i)\mathbb{1}_{\vec{v}_j}(\vec{v}_k+e_i)}_{\Gamma'_{jk}}\mathcal{L}(\pi^w_k)(\widehat{\lambda}_{k_1...k_d})\right]\mu^w_k
	\end{align*}
	Hence we can decompose $\Lambda_j(\widehat{\lambda}_{k_1...k_d})$ as follows:
	\[\Lambda(\widehat{\lambda}_{k_1...k_d})=\Lambda^{(k_1...k_d)}(\lambda)+\Gamma' \begin{bmatrix}
	\mathcal{L}(\pi^w_1) (\widehat{\lambda}_{k_1...k_d})\\
	\vdots\\
	\mathcal{L}(\pi^w_{2^N}) (\widehat{\lambda}_{k_1...k_d})
	\end{bmatrix}\]
	Where $\Lambda_j^{(k_1...k_d)}(\lambda)$ depends on $\lambda$ only through $\left(\mathcal{L}(\pi^w_{k}) (\widehat{\lambda}_{k_1...k_dm})\right)_{m\notin \{k_1,...,k_d\},k\in I}$. Thus, equation~\eqref{eq-hat-d} can be rewritten as:
	\begin{align}
	\underbrace{\left[\Gamma(\widehat{\lambda}_{k_1...k_d})-\Gamma'\right]}_{\Gamma^{(k_1...k_d)}(\widehat{\lambda}_{k_1...k_d})}\begin{bmatrix}
	\mathcal{L}(\pi^w_1) (\widehat{\lambda}_{k_1...k_d})\\
	\vdots\\
	\mathcal{L}(\pi^w_{2^N}) (\widehat{\lambda}_{k_1...k_d})
	\end{bmatrix}
	=\Lambda^{(k_1...k_d)}(\lambda)
	\end{align} 
	Eventually, we show $\Gamma^{(k_1...k_d)}(\widehat{\lambda}_{k_1...k_d})$ is invertible as soon as $|\lambda|\geq0$, denoting by $K=\{k_1,...,k_d\}$:
	\begin{align*}
	\sum_{k=1,\neq j}^{2^N}\lvert \Gamma^{(k_1...k_d)}_{jk}(\widehat{\lambda}_{k_1...k_d})\rvert &=  \sum_{k< j}\Gamma'_{jk}+\sum_{k> j}\Gamma(\widehat{\lambda}_{k_1...k_d})\\
	&= \sum_{k\neq j}^{2^N}\left(\sum_{i\in K}\alpha_i(\vec{v}_k)\delta_0(v_k^i)\mathbb{1}_{\vec{v}_j}(\vec{v}_k+e_i)+\sum_{i=1}^N\beta\delta_1(v_k^i)\mathbb{1}_{\vec{v}_j}(\vec{v}_k-e_i)\right)\mu^w_k\\
	& =\sum_{k=1,k\neq j}^{2^N}\mu^w_k\sum_{i=1}^N\left(\alpha_i(\vec{v}_k)\delta_0(v_k^i)\mathbb{1}_{\vec{v}_j}(\vec{v}_k+e_i)+\beta\delta_1(v_k^i)\mathbb{1}_{\vec{v}_j}(\vec{v}_k-e_i)\right)
	\\ & \ \ \ -\sum_{k=1,k\neq j}^{2^N}\mu^w_k\sum_{i\in I\backslash K}\alpha_i(\vec{v}_k)\delta_0(v_k^i)\mathbb{1}_{\vec{v}_j}(\vec{v}_k+e_i)\\
	&= \Gamma_{jj}(\widehat{\lambda}_{k_1...k_d})-\lvert\widehat{\lambda}_{k_1...k_d}\rvert\mu^w_j-\sum_{k=1,k\neq j}^{2^N}\mu^w_k\sum_{i\in I\backslash K}\alpha_i(\vec{v}_k)\delta_0(v_k^i)\mathbb{1}_{\vec{v}_j}(\vec{v}_k+e_i)
	\end{align*}
	Hence $\Gamma^{(k_1...k_d)}(\widehat{\lambda}_{k_1...k_d})$ is  invertible as a strictly dominant diagonal matrix as soon as $|\lambda|\geq0$. Finally, there is a unique way to express each $\mathcal{L}(\pi^w_m)(\lambda)$, $m\in I$, as a linear combination of terms of the family $\left(\mathcal{L}(\pi^w_k)(\check{\lambda}_l=\widehat{\lambda}_{1...l-1\ l+1...N})\right)_{1\leq l\leq N,m\in I}$.

	\subsubsection*{Step 2}
	To end with a way to compute $\mathcal{L}(\pi^w)(\lambda)$ we show how to find $\mathcal{L}(\pi^w_m)(\check{\lambda}_l)$ and then we get a new system of the form:
	\begin{align}\label{eq-fin-lap}
		D(\check{\lambda}_l)\begin{bmatrix}
		\mathcal{L}(\pi^w_1) (\check{\lambda}_l)\\
		\vdots\\
		\mathcal{L}(\pi^w_{2^N}) (\check{\lambda}_l)
		\end{bmatrix}
		=\Lambda^{(i)},\ \ with\ \Lambda^{(l)}\in \reels^{2^N}\ a\ constant\ vector		
	\end{align}

	The idea is the same as previously. We evaluate the expression \eqref{matrix-lap} in all $\check{\lambda}_l$ which gives:
	\begin{align}
		\Gamma(\check{\lambda}_l)\begin{bmatrix}
		\mathcal{L}(\pi^w_1) (\check{\lambda}_l)\\
		\vdots\\
		\mathcal{L}(\pi^w_{2^N}) (\check{\lambda}_l)
		\end{bmatrix}
		=\Lambda(\check{\lambda}_l)
	\end{align}
	So at line j:
	\[\Lambda_j(\check{\lambda}_l) =  \sum_{k=1}^{2^N}\left[\sum_{i=1}^N\alpha_i(\vec{v}_k)\delta_0(v_k^i)\mathbb{1}_{\vec{v}_j}(\vec{v}_k+e_i)\mathcal{L}(\pi^w_k)(\underbrace{\widehat{\lambda}_i\cap \check{\lambda}_l}_{=(0, ..., 0)\ if\ i\neq l})\right]\mu^w_k\]
	And as $\mathcal{L}(\pi^w_k)(0, ...,0) = 1$ we have:
	\[\Lambda_j(\check{\lambda}_l) = \underbrace{\sum_{k=1}^{2^N}\mu^w_k\sum_{i=1,i\neq l}^N\alpha_i(\vec{v}_k)\delta_0(v_k^i)\mathbb{1}_{\vec{v}_j}(\vec{v}_k+e_i)}_{\Lambda^{(l)}_j = cst} + \sum_{k=1}^{2^N}\underbrace{\alpha_l(\vec{v}_k)\delta_0(v_k^l)\mathbb{1}_{\vec{v}_j}(\vec{v}_k+e_l)\mu^w_k}_{D_{jk}\text{ for k<j}}\mathcal{L}(\pi^w_k)(\check{\lambda}_l)\]
	We conclude showing D is a diagonally dominant matrix:
	\begin{align*}
	\left\{
	\begin{array}{r c l}
	&D_{jk}(\check{\lambda}_l) = \alpha_l(\vec{v}_k)\delta_0(v_k^l)\mathbb{1}_{\vec{v}_j}(\vec{v}_k+e_l)\mu^w_k&,\ \ \forall k<j\\
	\\
	&D_{jj}(\check{\lambda}_l) = \Gamma_{jj}(\check{\lambda}_l)= \left[\left(\sum_{i=1}^N\beta\delta_1(v_j^i)+\alpha_i(\vec{v}_j)\delta_0(v_j^i)\right)+\lambda_l\right]\mu^w_j&
	\\
	\\
	&D_{jk}(\check{\lambda}_l) = \Gamma_{jk}(\check{\lambda}_l)= -\sum_{i=1}^N\beta\delta_1(v_k^i)\mathbb{1}_{\vec{v}_j}(\vec{v}_k-e_i)\mu^w_k&,\ \ \ \forall k>j
	\end{array}
	\right.
	\end{align*}
	As previously we show thanks to~\eqref{rel-jump} that whenever $\lambda_l\geq0$
	\[\lvert D_{jj}(\check{\lambda}_l)\rvert = \sum_{k=1,k\neq j}^{2^N}\lvert D_{jk}(\check{\lambda}_l)\rvert+ \mu^w_k\sum_{i=1,i\neq l}^{N}\alpha_i(\vec{v}_k)\delta_0(v_k^i)\mathbb{1}_{\vec{v}_j}(\vec{v}_k+e_i)+\lambda_l\mu^w_j > \sum_{k=1,k\neq j}^{2^N}\lvert D_{jk}(\check{\lambda}_l)\rvert\]
	
	Hence, $\mathcal{L}(\pi^w_k)(\check{\lambda}_l)$ are uniquely determined by~\eqref{eq-fin-lap} for all k, and for any l. Moreover, there is a unique way to express each $\mathcal{L}(\pi^w_m)(\lambda)$, $m\in I$, as a linear combination of terms of the family $\left(\mathcal{L}(\pi^w_k)(\check{\lambda}_l)\right)_{1\leq l\leq N,m\in I}$. We conclude that if it exists, ~\eqref{mes-gen-N} has a unique solution $\pi^w$.
\end{proof}
\subsection{Slow Fast analysis}\label{slow-fast-sec}

%We now study the all process $(W_t,S_t,V_t)_{t\geq 0}\in E=E_1\times E_2$ where $E_1=\left\{\Delta w K, K\in E_0\right\}$ and $E_2=\reels^N\times \{0,1\}^N$. As presented in the beginning, our slow process is the weight matrix which jumps of $+/-\ \Delta w$ with a certain rule. A slow fast analysis enables us to show that, on the slow time scale of plasticity, our process is close to the one given by $(S_t,V_t)\sim \pi^{W_t}$ and $(W_t)$ is the solution of the martingale problem associated to the operator $\mathcal{C}_{av}:D(\mathcal{C}_{av})\rightarrow C_b(E_1)$: \[\mathcal{C}_{av}f(w)=\int_{E_2}\mathcal{A}f(w,s,v)\pi^w(ds,dv)\]
%Where $\mathcal{A}$ is an operator defined in what follows and $\pi^w$ is the invariant measure of the fast process.

	As we know that synaptic weights dynamics are slow compare to the network dynamics, $(S_t,V_t)_{t\geq0}$ change fast compare to  $(W_t)_{t>0}$, so:
	\[\sum_{\tilde{w} \in G_i^w, \tilde{w}\neq w}\phi^i(s,\tilde{w},w)\ll\phi^i(s,w,w)\]
	Hence, in order to make a slow fast analysis we introduce the sequence $(\epsilon_n)_{n\geq 0}$, such that $\displaystyle \lim_{n\infty}\epsilon_n=0$, as follows:
	\begin{align}\label{def_phi_n}
	\sum_{\tilde{w} \in G_i^w, \tilde{w}\neq w}\phi_n^i(s,\tilde{w},w)=O(\epsilon_n)=1-\phi_n^i(s,w,w)
	\end{align}
	
%	\romain[inline]{Je mettrais un $\epsilon$ devant les $p^+,p^-$. Il me semble qu'a l'echelle de temps $O(\epsilon^{-1})$, on a une limite un peu plus simple que ce qu'on a en ce moment. Le cas de l'echelle $O(\epsilon^{-2})$ est plus complique. Voir par example Hutzenthaler, Martin, et Peter Pfaffelhuber. Stochastic averaging for multiscale Markov processes with an application to branching random walk in random environment. arXiv preprint arXiv:1504.01508, 2015. http://arxiv.org/abs/1504.01508.}
%	\pascal[inline]{à voir si je réécris tout avec cette hypothèse ou bien je l'inclus plus tard...}
%	
%	We could have taken $\epsilon_n=O(\epsilon_n)$ instead of $\epsilon_n$.
	We denote $\varphi^i$ functions such that 
	\begin{align}\label{def_varphi}
	\phi_n^i(s,\tilde{w},w)=\epsilon_n\varphi^i(s,\tilde{w},w)+o(\epsilon_n),\text{ so that } \sum_{\tilde{w} \in G_i^w, \tilde{w}\neq w}\varphi^i(s,\tilde{w},w)=O(1)
	\end{align}
	\begin{remarque}\label{rem-sf}
		{\it We give an example to illustrate~\eqref{def_varphi}. 
		Indeed, $\varphi^i$ depends on the choice of $\phi_n^i$ which is not unique.
		For instance, if the slow process comes from the fact $p^{+/-}$ are multiplied by $\epsilon_n$. Thus, $\phi_n^i$ can be deduced from~\eqref{def-proba-saut} where we replace $p^{+/-}$ by $\epsilon_np^{+/-}$:
		\[\phi_n^i(s,w',w)=\prod_{j\neq i}\left[\zeta_p^j\ \epsilon_n p^+(s_j)+(1-\zeta_p^j)(1-\epsilon_n p^+(s_j))\right]\left[\zeta_d^j\ \epsilon_n p^-(s_j)+(1-\zeta_d^j)(1-\epsilon_n p^-(s_j))\right]\]
		So, reminding that $W_t^{ij} \geq \Delta w$ for all $t\geq 0$:
		\begin{align*}
		&\phi_n^i(s,w'=w-\Delta w E_{ij},w)=\left\{
		\begin{array}{rcr}
		\hspace{-1em}&\epsilon_n p^-(s_j)\prod_{k\neq i,j}\left[(1-\epsilon_n p^-(s_k))\right]\prod_{k\neq i}\left[(1-\epsilon_n p^+(s_k))\right] & \text{ if } w^{ij}>\Delta w\\ \\
		\hspace{-1em}&0 & \text{ if } w^{ij}=\Delta w\\
		\end{array}
		\right.\\ \\
		&\phi_n^i(s,w'=w+\Delta w E_{ij},w)=\epsilon_n p^+(s_j)\prod_{k\neq i}\left[(1-\epsilon_n p^-(s_k))\right]\prod_{k\neq i,j}\left[(1-\epsilon_n p^+(s_k))\right]
		\end{align*}
		For all other $w'$, $\ \phi_n^i(s,w',w)=o(\epsilon_n)$.\\
		Thus:
		\begin{align*}
		&\phi_n^i(s,w'=w-\Delta w E_{ij},w)=\left\{
		\begin{array}{rcr}
		&\epsilon_n p^-(s_j)+o(\epsilon_n) & \text{ if } w^{ij}>\Delta w\\
		&0 & \text{ if } w^{ij}=\Delta w\\
		\end{array}
		\right.
		\\
		&\phi_n^i(s,w'=w+\Delta w E_{ij},w)=\epsilon_n p^+(s_j)+o(\epsilon_n)
		\end{align*}
		Hence we give the $\varphi^i$ which verifies conditions of~\eqref{def_phi_n} and~\eqref{def_varphi} for this example:
		\begin{align*}
		&\varphi^i(s,w'=w-\Delta w E_{ij},w)=\left\{
		\begin{array}{rcr}
		&p^-(s_j) & \text{ if } w^{ij}>\Delta w\\
		&0 & \text{ if } w^{ij}=\Delta w\\
		\end{array}
		\right.
		\\
		&\varphi^i(s,w'=w+\Delta w E_{ij},w)=p^+(s_j)\\
		&\varphi^i(s,w',w)\ \ \ \ \ \ \ \ \ \ \ \ \ \ \ \ \ \ =0 \ \ \ \ \ \ \ \ \ \ \ \ \ \ \ \ \ \ \text{for all other } w'
		\end{align*}
		We give another example: if we keep the normal $p^{+/-}$, and define $\phi_n^i$ as $\phi_n^i=\epsilon_n\phi^i$, then $\varphi^i=\phi^i$.}
	\end{remarque}
\vspace{2em}
	We now highlight the difference of time scale in the new generator $\mathcal{C}'_n$ which is the same as $\mathcal{C}$ with $\phi_n^i$ instead of $\phi^i$. In the following, test functions we take are all in $D(\mathcal{C})\cap C_b(E)$ : 
	\begin{align*}
	\mathcal{C}'_nf(w,s,v) &=\sum_{i=1}^N\partial_{s_i}f(w,s,v)+\sum_i\delta_1(v^i)\beta[f(w,s,v-e_i)-f(w,s,v)]\\
	&+\sum_i\alpha_i(w,v)\delta_0(v^i)\left(f(w,\ s-s_ie_i,\ v+e_i)-f(w,s,v)\right)\phi^i_n(s,w,w)\\
	&+\sum_i\alpha_i(w,v)\delta_0(v^i)\left(\sum_{\tilde{w} \in G_i^w, \tilde{w}\neq w}(f(\tilde{w},\ s-s_ie_i,\ v+e_i)-f(w,\ s,\ v))\phi^i_n(s,\tilde{w},w)\right)\\
	&=\underbrace{\sum_{i=1}^N\partial_{s_i}f(w,s,v)}_{\mathcal{B}_{tr}f(w,s,v)}+\underbrace{\sum_i\delta_1(v^i)\beta[f(w,s,v-e_i)-f(w,s,v)]}_{\mathcal{B}_{\downarrow}f(w,s,v)}\\
	&+(1-O(\epsilon_n))\underbrace{\sum_i\alpha_i(w,v)\delta_0(v^i)\left(f(w,\ s-s_ie_i,\ v+e_i)-f(w,s,v)\right)}_{\mathcal{B}_{\uparrow}f(w,s,v)}\\
	&+\sum_i\alpha_i(w,v)\delta_0(v^i)\left(\sum_{\tilde{w} \in G_i^w, \tilde{w}\neq w}(f(\tilde{w},\ s-s_ie_i,\ v+e_i)-f(w,\ s,\ v))(\epsilon_n\varphi^i(s,\tilde{w},w)+o(\epsilon_n))\right)
	\end{align*}
	Denoting the operator $\mathcal{A}: D(\mathcal{A})\subset C_b(E)\rightarrow C_b(E)$ by:
	\begin{align}\label{def-A}
	\mathcal{A}f(w,s,v) = \sum_i\alpha_i(w,v)\delta_0(v^i)\left(\sum_{\tilde{w} \in G_i^w, \tilde{w}\neq w}(f(\tilde{w},\ s-s_ie_i,\ v+e_i)-f(w,\ s,\ v))\varphi^i(s,\tilde{w},w)\right)
	\end{align}
	And $\mathcal{A}_r: D(\mathcal{A}_r)\subset C_b(E)\rightarrow C_b(E)$ by:
	\begin{align}\label{def-Ar}
	\mathcal{A}_rf(w,s,v) = \sum_i\alpha_i(w,v)\delta_0(v^i)\left(\sum_{\tilde{w} \in G_i^w, \tilde{w}\neq w}(f(\tilde{w},\ s-s_ie_i,\ v+e_i)-f(w,\ s,\ v))\right)
	\end{align}
	And $\mathcal{B}: D(\mathcal{B})\subset C_b(E)\rightarrow C_b(E)$:
	\begin{align}\label{def-B}
	\mathcal{B}=\mathcal{B}_{tr}+\mathcal{B}_{\downarrow}+\mathcal{B}_{\uparrow}
	\end{align}
%	\sout{We are interested in the process $(X^n_t)_{t\geq 0}=(\widetilde{W}_t^n,\widetilde{S}_t^n,\widetilde{V}_t^n)_{t\geq 0}$ generated by}
	
	With the previous assumptions on time scales we get the following process $(\widetilde{X}^n_t)_{t\geq 0}=(\widetilde{W}_t^n,\widetilde{S}_t^n,\widetilde{V}_t^n)_{t\geq 0}$ generated by:
%	\romain[inline]{C'est sur?}
%	\pascal[inline]{Tu veux dire qu'on est pas intéressé par celui la mais par le suivant? Est ce mieux écrit comme ca?}
	\[\mathcal{C}'_n = \epsilon_n\mathcal{A}+\mathcal{B}+O(\epsilon_n)\mathcal{B}_{\uparrow}+o(\epsilon_n)\mathcal{A}_r\]
	On this time scale, the network evolve at speed 1 and the plasticity at $\epsilon_n$. In order to apply results of~\cite{kurtz_averaging_1992}, we will study the system $\frac{1}{\epsilon_n}$ times faster and then denote by $(Y^n_t)_{t\geq 0}=(W_t^n,S_t^n,V_t^n)_{t\geq 0}=\left(\widetilde{X}^n_{\frac{t}{\epsilon_n}}\right)_{t\geq 0}$. Thus, $(Y^n_t)_{t\geq 0}$ is generated by:
	\begin{align}\label{def-C_n}
		\mathcal{C}_n = \frac{1}{\epsilon_n} \mathcal{C}'_n = \mathcal{A}+\frac{1}{\epsilon_n}\mathcal{B}+O(1)\mathcal{B}_{\uparrow}+o(1)\mathcal{A}_r
	\end{align}
	We remark that $\forall w\in E_1,\ h\in D(\mathcal{B})\subset C_b(E_2)$ the operator $\mathcal{B}_w$ defined by $\mathcal{B}_wh(s,v)=\mathcal{B}h(w,s,v)$ is the one studied previously. In the above, we showed it has a unique invariant measure $\pi^w$. Thereby, the process $(W_t^n,S_t^n,V_t^n)_{t\geq 0}$ with generator $\mathcal{C}_n$ is composed of a fast part which gives the dynamics of the network, $(S_t^n,V_t^n)_{t\geq 0}$, and a slow one which gives the weights' dynamics, $(W_t^n)_{t\geq 0}$. Hence, we can expect that as n tends to infinity, the fast part will quickly reach its stationary distribution depending on the current weights whereas the weights will jump from time to time. As soon as weights jump, the network will reach a new stationary distribution instantaneously. Weights jumps will depend on the network distribution.
	We apply Theorem 2.1 of~\cite{kurtz_averaging_1992} in the special case of example 2.3 given in the same article which gives in our case the following proposition:
	\begin{proposition}
		$(W_t^n,S_t^n,V_t^n)_{t\geq 0}$ converges, when $n\rightarrow+\infty$, in law to $(W_t,S_t,V_t)_{t\geq 0}$ where $(S_t,V_t)\sim \pi_{W_t}$ and $(W_t)$ is the solution of the martingale problem associated to the operator $\mathcal{C}_{av}:D(\mathcal{C}_{av})\rightarrow C_b(E_1)$: 
		\begin{align}\label{limit_eq}
		\mathcal{C}_{av}f(w)=\int_{E_2}\mathcal{A}f(w,s,v)\pi^w(ds,dv)
		\end{align}
	\end{proposition}
	\begin{proof}
%		We use the theorem~\ref{kurtz_theo} twice.
		We use the Theorem 2.1 of~\cite{kurtz_averaging_1992} twice. Once to link the occupation measure of the fast process to its invariant measure and then again to show~\eqref{limit_eq}.\\
%		This article deals with processes of two variables, one of which is much faster than the other. In the example, the fast variable possesses a unique invariant measure.\\
		We denote by $\mathcal{F}_t^n$ the natural filtration of $(W_t^n,S_t^n,V_t^n)_{t\geq 0}$. I will enumerate and show the properties we need in order to apply \cite{kurtz_averaging_1992}.\\
		
		1. $(W_t^n)_{t\geq 0}$ satisfies the compact containment condition that is for each $\epsilon>0$ and $T>0$ there exists a compact $K\subset E_1$ such that:
		\[\inf_n\proba(W^n_t\in K, t\leq T)\geq 1-\epsilon\]
		\begin{proof}
			We denote $K_i=\left\{\tilde{w}\in E_1\ s.t.\ \forall k,l\in [\![1,N]\!],\ \ \lvert\tilde{w}^{kl}-w^{kl}_0\rvert \leq i\Delta w \right\}$. Therefore, we want to show that for each $\epsilon>0$ and $T>0$, $\exists i$ large enough to have $\forall n\in \nat$:
			\begin{align}\label{cond-1}
			\proba(W^n_t\in K_{i-1}, t\leq T)\geq 1-\epsilon
			\end{align}
			But:
			\begin{align*}
			\proba(W_t^n\in K_{i-1}, t\leq T)=\proba\left(\widetilde{W}^n_t\in K_{i-1},\ \ t\leq \frac{T}{\epsilon_n}\right)=1-\proba\left(\exists t\leq \frac{T}{\epsilon_n},\ \ \widetilde{W}^n_t\notin K_{i-1}\right)
			\end{align*}
			So we major $\proba\left(\exists t\leq \frac{T}{\epsilon_n},\ \ \widetilde{W}^n_t\notin K_{i-1}\right)$ in what follows. As $\lim_{n\infty}\frac{T}{\epsilon_n} = +\infty$, the time on which we are looking at our process is becoming larger and larger with $n$ so we need the probability of jumping to become smaller and smaller as it is the case for $(\widetilde{W}^n_t)_{t\geq 0}$. Indeed, when neuron $i$ jumps from 0 to 1, $w^{ij}$ and $w^{ji}$ for $j\neq i$ have probability to jump of order $\epsilon_n$.\\
			First, from~\eqref{def_phi_n} there exists $c>0$ such that the probability to have a change of weight knowing neuron $i$ jumped from 0 to 1 is less than $c\ \epsilon_n$, so for all $i, s \text{ and } w$:
			\[\sum_{\tilde{w} \in G_i^w, \tilde{w}\neq w}\phi^i_n(s,\widetilde{w},w)=\proba\left(\widetilde{W}^n_t\neq \widetilde{W}^n_{t^-}| \widetilde{V}^{n,i}_t-\widetilde{V}^{n,i}_{t^-}=1\right)\leq c\ \epsilon_n<1\]
			From this we define the process $\overline{X}_t^n$ as the particular case of the process $\widetilde{X}^n_t$ for which neurons are independent and fire at rate $\gamma=\max(\beta,\alpha_M)$ and whenever a neuron $i$ jumps (from 0 to 1 or 1 to 0), $\overline{W}^n_t$ change with probability $c\ \epsilon_n$. We just impose that the size of weights jumps are as before: $+/-\ \Delta w$. Hence, in such a process weights jump more frequently. So denoting by $N_t^w$ and $\overline{N}_t^w$ processes respectively counting the number of jumps of $\widetilde{W}^n_t$ and $\overline{W}_t^w$ between 0 and t, and as previously, $\overline{N}_t$ the counting process corresponding to the number of jump of the process $(\overline{V}_t)_{t\geq 0}$. Thus:
			\begin{align*}
			\proba\left(\exists t\leq \frac{T}{\epsilon_n},\ \ \widetilde{W}^n_t\notin K_{i-1}\right)&=\proba\left(\exists k,l\in [\![1,N]\!],\ \ \exists t\leq \frac{T}{\epsilon_n},\ \ \lvert (\widetilde{W}^n_t)^{kl}-w^{kl}_0\rvert \geq i\Delta w\right)\\
			&\leq\proba\left(N_{\frac{T}{\epsilon_n}}^w \geq i\right)\leq \proba\left(\overline{N}_{\frac{T}{\epsilon_n}}^w \geq i\right)=\sum_{k=i}^{+\infty}\proba(\overline{N}_{\frac{T}{\epsilon_n}}^{w} = k)
			\end{align*}
			But 
			\begin{align*}
			\proba(\overline{N}_{\frac{T}{\epsilon_n}}^{w} = k)&=\sum_{m=k}^{+\infty}\left(\proba(\overline{N}_{\frac{T}{\epsilon_n}} =m) (c\epsilon_n)^k\left(1-c\epsilon_n\right)^{m-k} \binom{m}{k}\right)\\
			&=\sum_{m=k}^{+\infty}e^{-N\alpha_M\frac{T}{\epsilon_n}}\frac{(N\alpha_M\frac{T}{\epsilon_n})^m}{m!} \underbrace{(c\epsilon_n)^k\left(1-c\epsilon_n\right)^{m-k} \binom{m}{k}}_{probability\ (\overline{W}^n_t)_{t\geq 0}\ changed\ k\ times\ knowing\ (\overline{V}^n_t)_{t\geq 0}\ jumped\ m\ times}
			\end{align*}
			So for $\epsilon_n$ small enough:
			\begin{align*}
			\proba\left(\exists t\leq \frac{T}{\epsilon_n},\ \ \widetilde{W}^n_t\notin K_{i-1}\right)&\leq\sum_{k=i}^{+\infty}\sum_{m=k}^{+\infty}\left(e^{-N\alpha_M\frac{T}{\epsilon_n}}\frac{(N\alpha_M\frac{T}{\epsilon_n})^m}{m!} (c\epsilon_n)^k\left(1-c\epsilon_n\right)^{m-k} \binom{m}{k}\right)\\
			&\leq\sum_{k=i}^{+\infty}e^{-N\alpha_M\frac{T}{\epsilon_n}}\sum_{m=k}^{+\infty}\left(\frac{(N\alpha_MT)^m}{k!(m-k)!}\frac{c^k}{(\epsilon_n)^{m-k}}\left(1-c\epsilon_n\right)^{m-k}\right)\\
			&\leq\sum_{k=i}^{+\infty}\frac{(N\alpha_MTc)^k}{k!}e^{-N\alpha_M\frac{T}{\epsilon_n}}\underbrace{\sum_{m=k}^{+\infty}\left(\frac{(N\alpha_MT)^{m-k}(\frac{1}{\epsilon_n}-c)^{m-k}}{(m-k)!}\right)}_{e^{N\alpha_M(\frac{1}{\epsilon_n}-c)T}}\\
			&\leq \sum_{k=i}^{+\infty}\frac{(N\alpha_MTc)^k}{k!}e^{-N\alpha_MT} \leq \sum_{k=i}^{+\infty}\frac{(N\alpha_MTc)^k}{k!}\underset{i\rightarrow +\infty}{\longrightarrow }0
			\end{align*}
			Hence, $\exists\ i$ such that~\eqref{cond-1} is satisfied.	
		\end{proof}
		
		\vspace{1.5em}
		2. Moreover, define $\forall w\in E_1,\ h\in D(\mathcal{B})\subset C_b(E_2)$ the operator $\mathcal{B}_w$ by $\mathcal{B}_wh(s,v)=\mathcal{B}h(w,s,s)$. There exists a unique probability measure on $E_2$ $\pi^w$ such that:
		\[\int_{E_2} \mathcal{B}_w(s,v) \pi^w(ds,dv)=0\]
		\begin{proof}
			See theorem~\ref{theo-inv-mes}.
		\end{proof}
		\vspace{1.5em}
		3. $\forall g\in D(\mathcal{C})\cap C_b(E_1)$ :
		\begin{align}\label{eq-mart-A}
		g(W_t^n)-\int_{0}^{t}\mathcal{A}g(W_u^n,S_u^n,V_u^n)du+o_{\epsilon_n}(1)\int_{0}^{t}\mathcal{A}_rg(W_u^n,S_u^n,V_u^n)du
		\end{align}
		is a $\mathcal{F}_t^n$ martingale and $\forall (w,s,v)\in E$
		\begin{align}\label{eq-lim-Ar}
			\lim_{n\rightarrow +\infty}\esp_{(w,s,v)}\left[\sup_{t\leq T} \left| o_{\epsilon_n}(1)\int_{0}^{t}\mathcal{A}_rg(W_u^n,S_u^n,V_u^n)du\right|\right]=0
		\end{align}
		
		\begin{proof}
			$\forall f\in D(\mathcal{C})=D(\mathcal{C}_n)$:
			\begin{align}\label{mart-gen}
			f(W_t^n,S_t^n,V_t^n)-\int_{0}^{t}\mathcal{C}_nf(W_u^n,S_u^n,V_u^n)du
			\end{align}
			is a $\mathcal{F}_t^n$ martingale and $\forall g\in D(\mathcal{C})\cap C_b(E_1)$
			\begin{align*}
			\epsilon_n\mathcal{C}_ng(w,s,v)&=\sum_i\alpha_i(w,v)\delta_0(v^i)\left(\sum_{\tilde{w} \in G_i^w, \tilde{w}\neq w}(g(\tilde{w})-g(w))\phi^i_n(s,\tilde{w},w)\right)\\
			&=\epsilon_n\mathcal{A}g(w,s,v)+o(\epsilon_n)\mathcal{A}_rg(w,s,v)
			\end{align*}
			So~\eqref{eq-mart-A} is a $\mathcal{F}_t^n$ martingale.\\
			Moreover, as $g\in D(\mathcal{A})\cap C_b(E_1)$ and $\max_{i\in I,w\in E_1}\left(\#G_i^w\right)\cap\leq 2^{2(N-1)}$, $\exists M>0$ such that:
			\begin{align*}
				|\mathcal{A}_rg(w,s,v)| &=\left| \sum_i\alpha_i(w,v)\delta_0(v^i)\sum_{\tilde{w} \in G_i^w, \tilde{w}\neq w}(g(\tilde{w})-g(w))\right|\\
				&\leq 2^N\ \alpha_M\ \left(\max_{i\in I,w\in E_1}\#G_i^w\right) 2\sup_{x\in E_1}|g(x)|\leq M
			\end{align*}
			Hence, $\forall (w,s,v)\in E$, $\esp_{(w,s,v)}\left[\sup_{t\leq T} \left| o_{\epsilon_n}(1)\int_{0}^{t}\mathcal{A}_rg(W_u^n,S_u^n,V_u^n)du\right|\right]=o_{\epsilon_n}(1)$, thus condition~\eqref{eq-lim-Ar} is satisfied.
		\end{proof}
		\vspace{1.5em}
		
		4. Similarly, $\forall h\in D(\mathcal{C})\cap C_b(E_2)$
%		, $\exists \delta_n^h(t)$ such that:
		\begin{align}\label{mart-B}
		h(S_t^n,V_t^n)-\int_{0}^{t}\frac{1}{\epsilon_n}\mathcal{B}h(W_u^n,S_u^n,V_u^n)du
%		 + \delta_n^h(t)
		\end{align}
		is a $\mathcal{F}_t^n$ martingale
%		 and \[\lim_{n\rightarrow +\infty}\esp\left[\sup_{t\leq T}\epsilon_n\lvert\delta_n^h(t)\rvert\right]=0\]. 
		
%		We define $\delta_n^h(t)$ from the generator $\mathcal{C}_n$ computed for $h\in C_b(E_2)$:
		\begin{proof}
			\begin{align*}
			\epsilon_n\mathcal{C}_nh(w,s,v)&=\sum_{i=1}^N\partial_{s_i}h(s,v)+\sum_i\delta_1(v^i)\beta[h(s,v-e_i)-h(s,v)]\\
			&+\sum_i\alpha_i(w,v)\delta_0(v^i)\left(h(\ s-s_ie_i,\ v+e_i)-h(s,v)\right)\phi^i_n(s,w,w)\\
			&+\sum_i\alpha_i(w,v)\delta_0(v^i)\left(\sum_{\tilde{w} \in G_i^w, \tilde{w}\neq w}(h(\ s-s_ie_i,\ v+e_i)-h(\ s,\ v))\phi^i_n(s,\tilde{w},w)\right)\\
			&=\sum_{i=1}^N\partial_{s_i}h(s,v)+\sum_i\delta_1(v^i)\beta[h(s,v-e_i)-h(s,v)]\\
			&+\sum_i\alpha_i(w,v)\delta_0(v^i)\left(h(\ s-s_ie_i,\ v+e_i)-h(s,v)\right)\underbrace{\left(\phi^i_n(s,w,w)+\sum_{\tilde{w} \in G_i^w, \tilde{w}\neq w}\phi^i_n(s,\tilde{w},w)\right)}_{=1}\\
			&=\mathcal{B}h(w,s,v)
			\end{align*}
			So
			\[\mathcal{C}_nh=\frac{1}{\epsilon_n}\mathcal{B}h\]
			As~\eqref{mart-gen}is a $\mathcal{F}_t^n$ martingale, \eqref{mart-B} is a $\mathcal{F}_t^n$ martingale too.
		\end{proof}
			
%		In order to show this, we will use the fact that jump rates are bounded by assumption \eqref{ass-jump-rate}.
%%		, we can show that $\forall T>0$,  $\esp\left[\sup_{t\leq T}\lvert\delta_n^h(t)\rvert\right]=O(T\lVert h\rVert)$.
%		 Indeed, $\delta_n^h(t)$ is the sum of $\Delta h = h(after\ jump)-h(before\ jump)$ when $(S_t^n,V_t^n)$ jumps. Hence it can be bounded by the sum of all possible jumps of the neurons between 0 and t times twice the norm of h, hence $2^{N+1}\lVert h\rVert N\alpha_Mt$ in our case:
%		\begin{align*}
%			\esp\left[\lvert\delta_n^h(t)\rvert\right]&=\esp\left[\lvert\int_{0}^{t}(\mathcal{B}_{\uparrow}-\mathcal{A})h(W_u^n,S_u^n,V_u^n)du\rvert\right]=\esp\left[\epsilon_n\lvert\int_{0}^{nt}(\mathcal{B}_{\uparrow}-\mathcal{A})h(\widetilde{W}_u^n,\widetilde{S}_u^n,\widetilde{V}_u^n)du\rvert\right]\\
%			&\leq \epsilon_n\int_{0}^{nt}\int_{E}\sum_i\alpha_i(w,v)\delta_0(v^i)\sum_{\tilde{w} \in G_i^w}\lvert h(s-s_ie_i,\ v+e_i)-h(s,v)\rvert d\proba_{\left(\widetilde{W}_u^n,\widetilde{S}_u^n,\widetilde{V}_u^n\right)}(w,s,v)du\\
%			&\leq \epsilon_n\int_{0}^{nt}\int_{E}N\alpha_M2^{N+1}\lVert h\rVert d\proba_{\left(\widetilde{W}_u^n,\widetilde{S}_u^n,\widetilde{V}_u^n\right)}(w,s,v)du\\
%%			&\leq \epsilon_n\sum_{m=0}^{+\infty}m\proba\left(N_{nt}=m)2\lVert h\rVert = \epsilon_n\esp(N_{nt})2\lVert h\rVert\\
%%2\lVert h\rVert \epsilon_n\esp\left[Y(N\alpha_Mnt)\right]=
%			&\leq 2^{N+1}\lVert h\rVert N\alpha_Mt
%		\end{align*}
%		Thereby:
%		\[\lim_{n\rightarrow +\infty}\esp\left[\sup_{t\leq T}\epsilon_n\lvert\delta_n^h(t)\rvert\right]=0\]		
		Thus, conditions of example 2.3 of~\cite{kurtz_averaging_1992} are satisfied and $(W_t^n,S_t^n,V_t^n)_{t\geq 0}$ converges, when $n\rightarrow+\infty$, in law to $(W_t,S_t,V_t)_{t\geq 0}$ where $(S_t,V_t)\sim \pi_{W_t}$ and $(W_t)$ is the solution of the martingale problem associated to the operator $\mathcal{C}_{av}:D(\mathcal{C}_{av})\subset C_b(E_1)\rightarrow C_b(E_1)$: \[\mathcal{C}_{av}f(w)=\int_{E_2}\mathcal{A}f(w,s,v)\pi^w(ds,dv)\]
		
		Indeed, we use theorem 2.1 of~\cite{kurtz_averaging_1992} twice. First the point 1., 2. and 4. enable us to use the theorem to obtain that when $n \rightarrow +\infty$, $(W^n,\Gamma^n)\rightarrow (W,\Gamma)$ such that there exists a filtration $\{\mathcal{G}^1_t\}$ such that
		\[M_t=\int_{0}^{t} \int_{E_2}\mathcal{B}f(W(s),y)\Gamma(ds\times dy)\]
		is a $\{\mathcal{G}^1_t\}$-martingale for each $f \in \mathcal{D}(\mathcal{C}_{\lvert C_b(E_2)})$. But $M_t$ is continuous and of bounded variation, so it must be constant (see for instance Theorem 27 of~\cite{protter_stochastic_2010}) and finally $M_t=0$ for all $t>0$. We then write $\Gamma(ds\times dy)=\gamma_s(dy)ds$ and get 
		\[\int_{0}^{t} \int_{E_2}\mathcal{B}f(W(s),y)\gamma_s(dy)ds=0\]
		And then 
%		{\color{black}[Il faut alors vérifier qu'il existe D denombrable dans $D(\mathcal{B})$ tq $\overline{\{(g,\mathcal{B}g), g\in D\}}=\{(g,\mathcal{B}g), g\in \mathcal{B}\}$ pour conclure l'égalité suivante pour tout $f\in\mathcal{B}$ d'après l'exemple 2.3 de~\cite{kurtz_averaging_1992} (p 195-196) que j'ai mis dans Seafile sachant que le générateur est absolument continu suivant le flot d'après le theoreme 26.14 de~\cite{davis_markov_1993} dans Seafile également. Est-ce suffisant de prendre $D(\mathcal{B})\cap \{\text{fonctions mesurables de }E_2\}$ pour satisfaire cette condition?]}
		
		\[\int_{E_2}\mathcal{B}f(W(s),y)\gamma_s(dy)=0\]
		So we can take $\gamma_s(dy)=\pi_{W_s}(dy)$ is the unique invariant measure for $\mathcal{B}_x$ such that $\mathcal{B}_xf(y)=\mathcal{B}f(x,y)$. We conclude using 1.,2. and 3. and the Theorem 2.1 of~\cite{kurtz_averaging_1992} which gives that
%		theorem~\ref{kurtz_theo} which give that 
		\[\int_{0}^{t} \int_{E_2}\mathcal{A}f(W(s),y)\Gamma(ds\times dy)\]
		a martingale and thus $(W_t)$ is the solution of the martingale problem associated to the operator $\mathcal{C}_{av}:D(\mathcal{C}_{av})\subset C_b(E_1)\rightarrow C_b(E_1)$: \[\mathcal{C}_{av}f(w)=\int_{E_2}\mathcal{A}f(w,s,v)\pi^w(ds,dv)\]
	\end{proof}
	This time scale separation gives the infinitesimal generator of the weight process on the slow time scale. However, we don't know explicitly $\pi^w$ but its Laplace transform. Under some simple assumptions, we can get explicitly the dynamic of the weights which is a Markov process on $E_1$ with non-homogeneous jump rates depending on the Laplace transform of $\pi^w$. 
	\begin{proposition}\label{prop-LT}
		Suppose that for all $i$ $\exists\ \Phi^i_{(\tilde{w},w)}$ such that  $\varphi^i(s,\tilde{w},w)=\mathcal{L}\left(\Phi^i_{(\tilde{w},w)}\right)(s)$.\\
		Then,
		\[\mathcal{C}_{av}f(w)=\sum_{\tilde{w} \in G_w, \tilde{w}\neq w}(f(\tilde{w})-f(w))\left(\sum_{v\in I}\mu^w_v\sum_{i\ s.t.\ \tilde{w} \in G_i^w}\alpha_i(w,v)\delta_0(v^i)\int_{\reels_+^N}\Phi^i_{(\tilde{w},w)}(s)\mathcal{L}(\pi^w_v)(s)(ds)\right)\]
		Where $G_w=\{w' \in E_1, \proba(W_1=w'|W_0=w)>0\}$ and $\mu_v^w$ is the invariant measure of the process generated by $\mathcal{B}_0$ defined in~\eqref{def_B_0}.
	\end{proposition}
	%	\romain[inline]{$\mu^w_v$ depend de $w$. D'ailleurs, tu devrais dire que c'est la distribution stationnaire d'une certain processus.}
	%	\pascal[inline]{OK}
	\vspace{1em}
	
	\begin{proof}
		If we develop the infinitesimal generator of the process $(W_t)_{t\geq 0}$. Thanks to \eqref{limit_eq} and \eqref{def-A} we get:
		\begin{align*}
		\mathcal{C}_{av}f(w)&=\int_{E_2}\mathcal{A}f(w,s,v)\pi^w(ds,dv)=\sum_{v\in I}\int_{E_2}\mathcal{A}f(w,s,v)\mu^w_v\pi^w_v(ds)\\
		&= \sum_{v\in I}\int_{E_2}\sum_i\alpha_i(w,v)\delta_0(v^i)\left(\sum_{\tilde{w} \in G_i^w, \tilde{w}\neq w}(f(\tilde{w})-f(w))\varphi^i(s,\tilde{w},w)\right)\mu^w_v\pi^w_v(ds)\\
		&= \sum_{v\in I}\mu^w_v\sum_i\alpha_i(w,v)\delta_0(v^i)\left(\sum_{\tilde{w} \in G_i^w, \tilde{w}\neq w}(f(\tilde{w})-f(w))\int_{E_2}\varphi^i(s,\tilde{w},w)\pi^w_v(ds)\right)\\
		%		&{\color{black}= \sum_{v\in I}\mu^w_v\sum_i\alpha_i(w,v)\delta_0(v^i)\sum_{j\neq i}\left[(f(w+\Delta w E_{ji})-f(w))\int_{E_2}\frac{p^+(s_j)}{\sum_{j\neq i}p^+(s_j)+p^-(s_j)}\pi^w_v(ds)\right.}\\
		%		&{\color{black}\ \ \ \ \ \ \ \ \ \ \ \ \left. -(f(w-\Delta w E_{ij})-f(w))\int_{E_2}\frac{p^-(s_j)}{\sum_{j\neq i}p^+(s_j)+p^-(s_j)}\pi^w_v(ds)\right]+o(\epsilon)}
		\end{align*}
		With the assumption that for all $i$ $\exists\ \Phi^i_{(\tilde{w},w)}$ such that  $\varphi^i(s,\tilde{w},w)=\mathcal{L}(\Phi_{(\tilde{w},w)})(s)$ we get:
		\begin{align*}
		\mathcal{C}_{av}f(w)&= \sum_{v\in I}\mu^w_v\sum_i\alpha_i(w,v)\delta_0(v^i)\left(\sum_{\tilde{w} \in G_i^w, \tilde{w}\neq w}(f(\tilde{w})-f(w))\int_{\reels_+^N}\mathcal{L}(\Phi^i_{(\tilde{w},w)})(s)\pi^w_v(ds)\right)\\
		&=\sum_{v\in I}\mu^w_v\sum_i\alpha_i(w,v)\delta_0(v^i)\left(\sum_{\tilde{w} \in G_i^w, \tilde{w}\neq w}(f(\tilde{w})-f(w))\int_{\reels_+^N}\Phi^i_{(\tilde{w},w)}(s)\mathcal{L}(\pi^w_v)(s)(ds)\right)\\
		&=\sum_{\tilde{w} \in G, \tilde{w}\neq w}(f(\tilde{w})-f(w))\left(\sum_{v\in I}\mu^w_v\sum_{i\ s.t.\ \tilde{w} \in G_i^w}\alpha_i(w,v)\delta_0(v^i)\int_{\reels_+^N}\Phi^i_{(\tilde{w},w)}(s)\mathcal{L}(\pi^w_v)(s)(ds)\right)
		\end{align*}
	\end{proof}
%	\begin{remarque}
%		To check the assumption of proposition~\ref{prop-LT} we only need that $\exists\ P_{+/-}$ such that  $p_{+/-}(s)=\mathcal{L}(P_{+/-})(s)$. This is clear from equation~\eqref{def-proba-saut}:\romain[inline]{what is clear?}
%		\begin{align}\label{varphi-prod}
%		{\color{black}\varphi^i(s,w',w)=\frac{\prod_{j\neq i}\left[\zeta_p^jp^+(s_j)+(1-\zeta_p^j)(1-p^+(s_j))\right]\left[(\zeta_d^jp^-(s_j)+(1-\zeta_d^j)(1-p^-(s_j))\right]}
%		{\sum_{w'\neq w}\phi^i(s,w',w)}}
%		\end{align}
%	\end{remarque}
%	Thus, we don't need to simulate the all network any more, only the limit model. Moreover, we can analyse the weights dynamics. An example is given in simulations.
	\section{Sufficient conditions for recurrence and transience}
	
%	\romain[inline]{Plus simplement, on ne peut pas trouver des conditions pour que $V(w) = \sum_{ij} w^{ij}$ soit Lyapunov?}
%	\pascal[inline]{La fonction $V(w) = \sum_{ij} w^{ij}$ est Lyapunov si: $\exists r_0\ s.t.\ \forall ||w||\geq r_0\ ,\  \sum_{Z_p} \proba(W_{n+1}-W_n=Z_p)-\sum_{Z_d} \proba(W_{n+1}-W_n=-Z_d)\leq 0$ où $Z_p$ et $Z_d$ sont définis en~\ref{def_Z}. 
%	ce qui donne une condition de récurrence pas forcément positive. Il faut faire attention au fait que dans le cas où $W^{ij}=1$, ce poids ne peut qu'augmenter}
	Plasticity models evolved interacting with neurologists' discoveries. For instance, models based on STDP confirmed the need of homeostasis in order to regulate evolution of weights: prevent from their divergence or extinction, need of competition. Indeed, Hebbian learning suffers from a positive feedback instability and lead to all neurons wiring together~\cite{zenke_temporal_2017}. Synaptic scaling and metaplasticity are the main homeostatic mechanisms used in models through different ways~\cite{yger_models_2015}. In our model we don't have such mechanisms, like hard or soft bounds, but we can show that weights still stabilize under some conditions. We propose some general conditions which we manage to express in a simple condition on parameters of our model.\\

	In our case, we are faced with a non-homogeneous in space and homogeneous in time Markov process which is in a space equivalent to $\nat^{N^2}$. A few results exists for such processes. As underlines authors of the book~\cite{menshikov2016non}, Lyapunov techniques seem to be the most adapted to analyse such processes. 

	For the sake of simplicity and as it doesn't change anything in what follows, we consider now $\Delta w=1$. Then $E_1=\nat_*^{N^2}$. Also, we are interested in the case presented in the first example given in remark~\ref{rem-sf}. Therefore, the slow process comes from the fact $p^{+/-}$ are multiplied by $\epsilon_n$, so:
	\begin{align*}
	&\varphi^i(s,w'=w- E_{ij},w)=\left\{
	\begin{array}{rcr}
	&p^-(s_j) & \text{ if } w^{ij}>1\\
	&0 & \text{ if } w^{ij}=1\\
	\end{array}
	\right.
	\\
	&\varphi^i(s,w'=w+ E_{ij},w)=p^+(s_j)\\
	&\varphi^i(s,w',w)\ \ \ \ \ \ \ \ \ \ \ \hspace{0.1em}\ \ \ =0 \ \ \ \ \ \ \ \ \ \ \ \ \ \ \ \ \ \ \text{for all other } w'
	\end{align*}
	If we develop the infinitesimal generator of the process $(W_t)_{t\geq 0}$. Thanks to \eqref{limit_eq} and \eqref{def-A} we get:
	\begin{align}
		\begin{split}
		\mathcal{C}_{av}f(w)&=\int_{E_2}\mathcal{A}f(w,s,v)\pi^w(ds,dv)=\sum_{v\in I}\int_{E_2}\mathcal{A}f(w,s,v)\mu^w_v\pi^w_v(ds)\\
		&= \sum_{v\in I}\int_{E_2}\sum_i\alpha_i(w,v)\delta_0(v^i)\left(\sum_{\tilde{w} \in G_i^w, \tilde{w}\neq w}(f(\tilde{w})-f(w))\varphi^i(s,\tilde{w},w)\right)\mu^w_v\pi^w_v(ds)\\
		&= \sum_{i,j:\ i\neq j}(f(w+E_{ij})-f(w))\left(\sum_{v,v^j=0}\mu^w_v\alpha_j(w,v)\int_{E_2}p^+(s_i)\pi^w_v(ds)\right)\\
		& +\sum_{i,j:\ i\neq j}\mathbb{1}_{]1, +\infty[}(w^{ij})(f(w-E_{ij})-f(w))\left(\sum_{v,v^i=0}\mu^w_v\alpha_i(w,v)\int_{E_2}p^-(s_j)\pi^w_v(ds)\right)
		\end{split}
	\end{align}
%	Finally, we get a model for weights in the limit their dynamic is much slower than the neural network one. However, we face two problems. First, we know explicitly jump rates of the random walk followed by $W_t$ on condition of computing the $\mathcal{L}(\pi^w_v)$ and $\mu^w_v$ for all $v,\ w$. Which is quite {\color{black}heavy} to compute: need to inverse $2^N$ by $2^N$ matrices. Second, even if we managed to compute them, the study of non homogeneous (in space) random walks is a field with only a few results.
	Denoting rate of jump by $r^{+/-}_{ij}(w)$ we get:
	\begin{align}\label{gen-nat}
	\mathcal{C}_{av}f(w)=\sum_{i, j}(f(w+ E_{ij})-f(w))r^+_{ij}(w) + \mathbb{1}_{]1, +\infty[}(w^{ij})(f(w-E_{ij})-f(w))r^-_{ij}(w)
	\end{align}
%	With:
%	\[r^+_{ij}(w)=\sum_{v,v^j=0}\mu^w_v\alpha_j(w,v)\int_{E_2}p^+(s_i)\pi^w_v(ds)\] 
%	And
%	\[r^-_{ij}(w)=\sum_{v,v^i=0}\mu^w_v\alpha_i(w,v)\int_{E_2}p^-(s_j)\pi^w_v(ds)\]
%	We wonder about recurrence or transience of such a chain. It is not easy in the general case to show such results, see for instance~\cite{menshikov2016non} for current results in that direction. First, thanks to the course of M. Hairer~\cite{hairer_convergence_2010}, we show positive recurrence of $(W_t)_{t\geq 0}$ under some strong assumptions which lead to simple conditions on parameters $\alpha_m,\ \alpha_M,\ \beta,\ p^{+/-}$. Then, we use results from~\cite{menshikov2016non} to give weaker conditions for positive recurrence.
	\newpage
	\subsection{General conditions for positive recurrence and transience}
	\begin{proposition}\label{prop-pos-rec-strong}
		Assume the following conditions:
		\begin{itemize}
			\item $\exists I_m^{+/-},\ I_M^{+/-}\in \reels_+^* \text{ such that } I_m^{+/-} \leq r^{+/-}_{ij}(w) \leq I_M^{+/-}$ for all $w$,
			\item $I_m^->I_M^+$ which leads to $r^+_{ij}(w)-r^-_{ij}(w)\leq I_M^+-I_m^-<0$ for all $w$
		\end{itemize}
		Then, the process $(W_t)_{t\geq 0}$ associated to the generator $\mathcal{C}_{av}$ given in~\eqref{gen-nat} is positive recurrent.
	\end{proposition}
	\begin{proof}
		We use proposition 1.3 from Hairer's course~\cite{hairer_convergence_2010}. In order to check assumptions of this proposition, we need to find a function $f:E_1\rightarrow \reels_+$ such that $\lim_{x\rightarrow +\infty}f(x) = +\infty$ and $\exists A\subset E_1$ finite such that for all  $w\in E_1\backslash A$:
		\begin{align}\label{cond_lyap}
		\mathcal{C}_{av}f(w)\leq -1
		\end{align}
		We define $f:E_1\rightarrow \reels+$ as:
		\begin{align*}
		\forall\ w\in E_1,\ \ f(w)=\sum_{i,j:\ i\neq j}(w^{ij})^2=||w||^2
		\end{align*}
		So
		\begin{align*}
		\mathcal{C}_{av}f(w)&=\sum_{i,j:\ i\neq j}(||w+E_{ij}||^2-||w||^2)r^+_{ij}(w)+\sum_{i,j:\ i\neq j,\ w^{ij}\neq 1}(||w-E_{ij}||^2-||w||^2)r^-_{ij}(w)\\
		&= \sum_{i,j:\ i\neq j}(2w^{ij}+1)r^+_{ij}(w)+\sum_{i,j:\ i\neq j,\ w^{ij}\neq 1}(-2w^{ij}+1)r^-_{ij}(w)\\
		&= \sum_{i,j:\ i\neq j,\ w^{ij}\neq 1}(2w^{ij}+1)r^+_{ij}(w)+\sum_{i,j:\ i\neq j,\ w^{ij}\neq 1}(-2w^{ij}+1)r^-_{ij}(w)+\sum_{i,j:\ i\neq j,\ w^{ij}= 1}(2w^{ij}+1)r^+_{ij}(w)\\
		&= \sum_{i,j:\ i\neq j,\ w^{ij}\neq 1}2w^{ij}(r^+_{ij}(w)-r^-_{ij}(w))+\sum_{i,j:\ i\neq j,\ w^{ij}\neq 1}\left(r^-_{ij}(w)+r^+_{ij}(w)\right)+\sum_{i,j:\ i\neq j,\ w^{ij}= 1}(2w^{ij}+1)r^+_{ij}(w)\\
		&\leq \sum_{i,j:\ i\neq j,\ w^{ij}\neq 1}2w^{ij}(r^+_{ij}(w)-r^-_{ij}(w))+\underbrace{(N^2-\# \{w^{ij}=1\})r^-_{ij}(w)+N^2r^+_{ij}(w)+2\# \{w^{ij}=1\}r^+_{ij}(w)}_{\leq 3N^2(r^+_{ij}(w)+r^-_{ij}(w))\leq 3N^2(I_M^++I_M^-)}
		\end{align*}
		As $r^+_{ij}(w)-r^-_{ij}(w)\leq I_M^+-I_m^-<0$ for all $w$ such that $||w||>N$ (to enforce that at least one $w^{kl}>1$):
		\begin{align*}
		\mathcal{C}_{av}f(w)\leq 2\max_{i,j:\ i\neq j}(w^{ij})(I_M^+-I_m^-)+3N^2((I_M^++I_M^-))\underset{||w||\rightarrow +\infty}{\longrightarrow} -\infty
		\end{align*}
%		Thus, denoting by $A_1=\{w,\ 2\max_{i,j}(w^{ij})(I_M^+-I_m^-)+3N^2(I_M^++I_M^-)> -1\}$, we show $A_1$ is finite and then conclude the proof. 
		Let $w_{sep}\in \nat_+^*$ be such that
		\[w'\geq w_{sep}\ \  \Rightarrow\ \  2w'(I_M^+-I_m^-)+3N^2(I_M^++I_M^-)\leq -1\]
		As $\max_{i,j}(w^{ij})\geq \frac{||w||}{N}$, we define $A_0=\{w,\ ||w||> Nw_{sep}\}$ so:
		\[w\in A_0\ \  \Rightarrow\ \  \max_{i,j}(w^{ij})\geq w_{sep}\ \ \Rightarrow\ \ 2\max_{i,j}(w^{ij})(I_M^+-I_m^-)+3N^2((I_M^++I_M^-))\leq -1\]
		Let $A= A_0^c=\{w,\ ||w||\leq Nw_{sep}\}$. $A$ is finite and for all $w\in E_1\backslash A$:
		
		\begin{align*}
		\mathcal{C}_{av}f(w)\leq -1
		\end{align*}
		Which proves, by proposition 1.3 from Hairer's course~\cite{hairer_convergence_2010}, positive recurrence of $(W_t)_{t\geq 0}$.
	\end{proof}
	\begin{corollaire}\label{cor-sup}
%		{\color{black}[Inverser avec la proposition précédente]}
		If $\lim_{r\rightarrow +\infty}\sup_{w\in \Sigma,\lVert w\rVert\geq r}(r^+_{ij}(w)-r^-_{ij}(w)) < 0$\\
		Then, the process $(W_t)_{t\geq 0}$ associated to the generator $\mathcal{C}_{av}$ given in~\eqref{gen-nat} is positive recurrent.
	\end{corollaire}
	\begin{proof}
		Exactly the same as the proof of proposition~\ref{prop-pos-rec-strong}.
	\end{proof}
	
	\begin{proposition}
		$p^+(s)-p^-(s)>\gamma>0$ for all $s\in \reels_+^*$ imply transience of $(W_t)_{t\geq 0}$.
	\end{proposition}	
	\begin{proof}
		Let define 
		\[A=\{w\in E_1 \text{ s.t. }\min w^{ij}\leq w'>1\}\]
		And $f: E_1\rightarrow \reels_+$ such that:
		\begin{align*}
		f(x) = \left\{
		\begin{array}{rcr}
		\frac{1}{N^2w'} & \text{ if } x\in A,\\\\
		\frac{1}{\sum x^{ij}} & \text{ if } x\in A^c \\
		\end{array}
		\right.
		\end{align*}
		Thus, $\inf_A f=\frac{1}{N^2w'}$ so for all $w\in A^c$, $f(w)<\inf_A f$.\\
		Moreover,
		\begin{align*}
		\begin{split}
		\mathcal{C}_{av}f(w)
		&=\sum_{i, j}(f(w+ E_{ij})-f(w))r^+_{ij}(w) + \mathbb{1}_{]1, +\infty[}(w^{ij})(f(w-E_{ij})-f(w))r^-_{ij}(w)\\
		&= \sum_{i, j} \frac{-1}{w(w+1)} r^+_{ij}(w) + \mathbb{1}_{]1, +\infty[}(w^{ij})\frac{1}{w(w-1)}r^-_{ij}(w)\\
		&\leq \sum_{i, j} \frac{-1}{w(w+1)} r^+_{ij}(w) + \frac{1}{w(w-1)}r^-_{ji}(w)\leq\sum_i\sum_{v,v^i=0}\mu^w_v\alpha_i(w,v)\sum_{j\neq i}\frac {1}{w(w+1)}\int_{E_2}(p^-(s_j)-p^+(s_j))\pi^w_v(ds)\\
		&<-\gamma\frac {1}{w(w+1)}\sum_i\sum_{v,v^i=0}\mu^w_v\alpha_i(w,v)\leq 0
		\end{split}
		\end{align*}
		We can apply theorem 2.5.8 of~\cite{menshikov2016non} to prove transience of the process.
	\end{proof}
	Surprisingly, it is not true that $p^+(s)-p^-(s)<-\gamma<0$ for all $s\in \reels_+^*$ imply positive recurrence of $(W_t)_{t\geq 0}$ as we showed in simulations.
%	\eot[inline]{Cette remarque est un eu surprenante. Quand on intègre \(p^+\)
%	ou \(p^-\) par rapport à la mesure invariante, on retombe sur l'hypothèse 
%\(r^+ - r^- < 0\) ce qui devrait aboutir à la transience si la prop 4.1 
%est applicable. Comme l'autre condition ne parait pas chère....
%Comment est-ce possible ?}
%\pascal[inline]{Il y a une différence entre $r^+_{ij}$ et $r^-_{ij}$ que je n'ai pas mise assez en valeur: 
%	\[r^+_{ij}(w)=\sum_{v,v^j=0}\mu^w_v\alpha_j(w,v)\int_{E_2}p^+(s_i)\pi^w_v(ds)\]
%	\[r^-_{ij}(w)=\sum_{v,v^i=0}\mu^w_v\alpha_i(w,v)\int_{E_2}p^-(s_j)\pi^w_v(ds)\]
%Les $i$ et $j$ sont inversés entre l'un et l'autre. C'est pourquoi intégrer contre la mesure invariante n'est pas suffisant. Par exemple, si tous les $w^{kl}$ sont petits sauf $w^{ij}$ qui est très grand, on aura $\alpha_j(w,v)$ grand (dès que $v^i=1$) et $\alpha_i(w,v)$ petits dans tous les cas, dans ce cas on peut avoir $p^+(s)-p^-(s)<-\gamma<0$ for all $s\in \reels_+^*$ et $r^+_{ij}(w)>r^+_{ij}(w)$. Il faut peut être que j'inverse les indices dans ma définition de $r^+_{ij}(w)$. Je serais alors amener à comparer $r^+_{ij}$ et $r^-_{ji}$.
%}
	\begin{remarque}\label{rem-eta}
		Denoting by $\eta(w)$ the expectation of jumps of $(W_t)_{t\geq 0}$, we easily get that $\eta^{ij}(w)=(r^+_{ij}(w)-r^-_{ij}(w))\Delta w$. Thus, conditions on $(r^+_{ij}(w)-r^-_{ij}(w))$ are equivalent to conditions on $\eta^{ij}(w)$.
	\end{remarque}
	We now compute the constants $I_m^{+/-},\ I_M^{+/-}\in \reels_+^*$ in order to derive a simple condition of transience or recurrence depending on parameters.
	\subsection{A simple condition on parameters for positive recurrence}
	We want to bound the following quantities:
	\[r_{ij}^{+/-}(w)=\sum_{v,v^j=0}\mu^w_v\alpha_j(w,v)\int_{E_2}p^{+/-}(s_i)\pi^w_v(ds)\]
	The main idea is to use that $0<\alpha_m \leq \alpha_j(w,v)\leq \alpha_M$ so
	\[\alpha_m\int_{E_2}p^+(s_i)\left(
	\sum_{v,v^j=0}\mu^w_v\pi^w_v(ds)\right) \leq \sum_{v,v^j=0}\mu^w_v\alpha_j(w,v)\int_{E_2}p^+(s_i)\pi^w_v(ds) \leq \alpha_M\int_{E_2}p^+(s_i)\left(\sum_{v,v^j=0}\mu^w_v\pi^w_v(ds)\right)\]
	The quantity to bound is now
	\[\int_{E_2}p^+(s_i)\left(\sum_{v,v^j=0}\mu^w_v\pi^w_v(ds)\right)\]
	But for all differentiable $p^+$, by Fubini:
	\begin{align*}
	\int_{E_2}p^+(s_i)\pi^w_v(ds)&=\int_{E_2}\left(\int_{0}^{s_i}(p^+)'(u)du+p^+(0)\right)\pi^w_v(ds)\\
	&= p^+(0) + \int_{E_2}\left(\int_{0}^{+\infty}(p^+)'(u)\mathbb{1}_{\{u<s_i\}}du\right)\pi^w_v(ds)\\
	&= p^+(0) + \int_{0}^{+\infty}\left(\int_{E_2}\mathbb{1}_{\{s_i<u\}}\pi^w_v(ds)\right)(p^+)'(u)du\\
	&= p^+(0) + \int_{0}^{+\infty}\proba_{\pi^w}\left(S_t^i>u|V_t=v\right)(p^+)'(u)du
	\end{align*}
	We are finally interested in bounding 
	\[
	\sum_{v,v^j=0}\mu^w_v\proba_{\pi^w}\left(S_t^i>u|V_t=v\right)=\sum_{v,v^j=0}\proba_{\pi^w}\left(V_t=v\right)\proba_{\pi^w}\left(S_t^i>u|V_t=v\right)=\proba_{\pi^w}\left(S_t^i>u,V^j_t=0\right)
	\]
	\begin{proposition}\label{prop-born-taux}
		For all $w\in E_2$:
		\begin{align}\label{born-pi}
		\frac{\alpha_M^2 e^{-\beta u}-\beta^2 e^{-\alpha_M u}}{\alpha_M^2 - \beta^2}\frac{\beta}{\alpha_M+\beta}\leq \proba_{\pi^w}\left(S_t^i>u,V^j_t=0\right) \leq \frac{\alpha_m^2 e^{-\beta u}-\beta^2 e^{-\alpha_m u}}{\alpha_m^2 - \beta^2}\frac{\beta}{\alpha_m+\beta}
		\end{align}
	\end{proposition}
	Let $(\overline{V_t},\overline{S_t})$ and $(\underline{V_t},\underline{S_t})$ be the processes for which $\left((\overline{V_t^i},\overline{S_t^i})_{t\geq 0}\right)_i$ (the same for $\left((\underline{V_t^i},\underline{S_t^i})_{t\geq 0}\right)_i$) are independent each other and neurons jump from 0 to 1 respectively with a rate $\alpha_M$ and $\alpha_m$ and from 1 to 0 with the rate $\beta$. We thus get for similar trajectories, for all $t\geq 0$ and all $i$: 
	\[\overline{S_t^i} \leq S_t^i \leq \underline{S_t^i}\]
	
	Thus, we can bound $\proba_{\pi^w}\left(S_t^i>u|V_t=v\right)$ as follows:
	\[\proba_{\pi^w}\left(\overline{S_t^i}>u,V^j_t=0\right) \leq \proba_{\pi^w}\left(S_t^i>u,V^j_t=0\right) \leq \proba_{\pi^w}\left(\underline{S_t^i}>u,V^j_t=0\right)\]
	So
	\begin{align}\label{ineg-mes-inv}
	\proba_{\pi^w}\left(\overline{S_t^i}>u\right)\proba_{\pi^w}\left(V^j_t=0\right)\leq \proba_{\pi^w}\left(S_t^i>u,V^j_t=0\right) \leq \proba_{\pi^w}\left(\underline{S_t^i}>u\right)\proba_{\pi^w}\left(V^j_t=0\right)
	\end{align}
	First, let bound $\proba_{\pi^w}\left(V^j_t=0\right)=\sum_{v,v^j=0}$.
	\begin{proposition}\label{ine-mu}
		For all $i$, $w$:
		\begin{align}
		\frac{\beta}{\alpha_M+\beta}\leq \sum_{v,v^i=0}\mu^w_v\leq\frac{\beta}{\alpha_m+\beta}
		\end{align}
	\end{proposition}
	\begin{proof}
		Let recall from~\eqref{def_B_0} the generator of the process of neurons ($V_t$) only when $w$ is fixed jump:
		\begin{align*}
		\mathcal{B}_0g(v) =\sum_{i=1}^{N}\beta\delta_1(v^i)[g(v-e_i)-g(v)]+\alpha_i(w,v)\delta_0(v^i)\left[g(v+e_i)-g(v)\right]
		\end{align*}
		Which gives for the invariant measure $\mu^w=(\mu_v^w)_{v\in I}$:
		\begin{align*}
		\sum_{v\in I}\mathcal{B}_0g(v)\mu_v^w = 0
		\end{align*}
		Thus, let $i\in [\![1,N]\!]$ with $g_i(v)=\delta_0(v^i)$, we get:
		\begin{align*}
		0=\sum_{v\in I}\mathcal{B}_0g_i(v)\mu_v^w &= \sum_{v\in I}\mu_v^w\sum_{j=1}^{N}\beta\delta_1(v^j)[g_i(v-e_j)-g_i(v)]+\alpha_j(w,v)\delta_0(v^j)\left[g_i(v+e_j)-g_i(v)\right]\\
		&= \sum_{v\in I,v^i=0}\mu_v^w\sum_{j=1}^{N}\beta\delta_1(v^j)[g_i(v-e_j)-g_i(v)]+\alpha_j(w,v)\delta_0(v^j)\left[g_i(v+e_j)-g_i(v)\right]\\
		&\sum_{v\in I,v^i=1}\mu_v^w\sum_{j=1}^{N}\beta\delta_1(v^j)[g_i(v-e_j)-g_i(v)]+\alpha_j(w,v)\delta_0(v^j)\left[g_i(v+e_j)-g_i(v)\right]\\
		&= \sum_{v\in I,v^i=0}\mu_v^w(-\alpha_i(w,v))+\sum_{v\in I,v^i=1}\mu_v^w\beta
		\end{align*}
		Indeed, when $v^i=0$, for all $j\neq i$ one has $g_i(v-e_j)-g_i(v)=g_i(v+e_j)-g_i(v)=1-1=0$ and $g_i(v+e_i)-g_i(v)=0-1=-1$, $\delta_1(v^i)=0$. Doing the same reasoning with $v^i=1$ we get the last line. Then, we also know that  $\sum_{v\in I,v^i=0}\mu_v^w+\sum_{v\in I,v^i=1}\mu_v^w=1$ so:
		\begin{align*}
		0&=\sum_{v\in I,v^i=0}\mu_v^w(-\alpha_i(w,v))+\sum_{v\in I,v^i=1}\mu_v^w\beta\\
		&=\sum_{v\in I,v^i=0}\mu_v^w(-\alpha_i(w,v))+(1-\sum_{v\in I,v^i=0}\mu_v^w)\beta\\
		&=\beta-(\beta\sum_{v\in I,v^i=0}\mu_v^w+\sum_{v\in I,v^i=0}\mu_v^w\alpha_i(w,v))
		\end{align*}
		Finally,
		\begin{align*}
		(\beta+\alpha_m)\sum_{v\in I,v^i=0}\mu_v^w\leq(\beta\sum_{v\in I,v^i=0}\mu_v^w+\sum_{v\in I,v^i=0}\mu_v^w\alpha_i(w,v))\leq (\beta+\alpha_M)\sum_{v\in I,v^i=0}\mu_v^w
		\end{align*}
		We conclude that
		\begin{align}\label{born-mu}
		\frac{\beta}{\alpha_M+\beta}\leq \sum_{v,v^i=0}\mu^w_v\leq\frac{\beta}{\alpha_m+\beta}
		\end{align}
	\end{proof}

	We now focus our interest on computations of $\proba_{\pi^w}\left(\underline{S_t^i}>u\right)$ and $\proba_{\pi^w}\left(\overline{S_t^i}>u\right)$.\\
	
	It is interesting to note that the previous inequality holds for all $t \geq 0$. We already showed in theorem~\ref{theo-inv-mes} that each of $(\underline{S_t^i},\underline{V_t^i})$ and $(\overline{S_t^i},\overline{V_t^i})$ possesses a unique invariant measure $(\underline{S_{\infty}},\underline{V_{\infty}})$ and $(\overline{S_{\infty}},\overline{V_{\infty}})$. Therefore, as~\eqref{ineg-mes-inv} is true for all $t\geq 0$, we get:
	\begin{align}\label{born-infty}
	\proba\left(\overline{S_{\infty}}>u\right)\proba_{\pi^w}\left(V^j_t=0\right)\leq \proba_{\pi^w}\left(S_t^i>u,V^j_t=0\right) \leq \proba\left(\underline{S_{\infty}}>u\right)\proba_{\pi^w}\left(V^j_t=0\right)
	\end{align}
	
	We turn on the computing of measures of $(\underline{S_{\infty}},\underline{V_{\infty}})$ and $(\overline{S_{\infty}},\overline{V_{\infty}})$ from their Laplace transforms. To do so, we study the process $(S_t,V_t)\in \reels+ \times \{0,1\}$ with the following generator $(\mathcal{A},D(\mathcal{A}))$:
	\begin{align}
	\mathcal{A}f(s,v) = \beta \delta_1(v)\left(f(s,0)-f(s,1)\right)+\alpha \delta_0(v)\left(f(0,1)-f(s,0)\right)+\partial _sf(s,v)
	\end{align}	
%	Its domain of definition $D(\mathcal{A})$ is the set of $f$ measurable on $\reels+ \times \{0,1\}$ such that $t\mapsto f(t,v)$ is absolutely continuous ($\forall \epsilon>0\ \exists \delta>0$ such that for all sequence disjoint intervals $([a_n,b_n])_{n\in \nat}\subset \reels+\ $:  $\sum_{n\geq 0}(b_n-a_n)<\delta \Rightarrow \sum_{n\geq 0}|f(b_n,v)-f(a_n,v)|<\epsilon$).
%	\romain[inline]{Tu as deja demontre qu'il y a une distrib invariante (cas N=1). Tu as juste besoin de la calculer.}
	%	On peut calculer $\pi_0$ et $\pi_1$ explicitement, montrant ainsi leur existence. Leur unicité s'explique par le fait que la transformée de Laplace caractérise la loi d'une variable aléatoire.
	\begin{proposition}\label{laplace1}
		The invariant probability measure $\pi(ds,v)$ of $(S_t,V_t)$ is:
		\begin{align}
		\pi (ds,v)= \frac{\alpha\beta}{\alpha-\beta} (e^{-\beta s}-e^{-\alpha s})ds\mu_0\mathbb{1}_0(v)+\beta e^{-\beta s}ds \mu_1\mathbb{1}_1(v)
		\end{align}
	\end{proposition}
	\begin{proof}
		As in~\eqref{pi-decomp}, $\pi$ can be written as: $\pi(ds,v)= \pi_0(ds) \mathbb{1}_0(v)\mu_0+\pi_1(ds) \mathbb{1}_1(v)\mu_1$. In this case, $\mu_0=\frac{\beta}{\alpha+\beta}$ and $\mu_1=\frac{\alpha}{\alpha+\beta}$. Moreover, it is an invariant measure if and only if $\esp_{\pi} [\mathcal{A}f]=0, \forall f\in D(\mathcal{A})$. Thanks to functions $f$ well-chosen we get equations on Laplace transforms of $\pi_0$ and $\pi_1$.\\
		Denoting $e_{\lambda}^i(s,v)=e^{-\lambda s}\delta_v(i)$ we get:
		\[\left \{
		\begin{array}{r c l}
		\frac{\beta}{\alpha+\beta}\int_{\reels_+}Ae_{\lambda}^0(s,0)\pi_0(ds)+\frac{\alpha}{\alpha+\beta}\int_{\reels_+}Ae_{\lambda}^0(s,1)\pi_1(ds)=0 \\
		\\
		\frac{\beta}{\alpha+\beta}\int_{\reels_+}Ae_{\lambda}^1(s,0)\pi_0(ds)+\frac{\alpha}{\alpha+\beta}\int_{\reels_+}Ae_{\lambda}^1(s,1)\pi_1(ds)=0
		\end{array}
		\right.\]	
		We remind us that:
		\[\mathcal{A}f(s,v) = \beta \delta_1(v)\left(f(s,0)-f(s,1)\right)+\alpha \delta_0(v)\left(f(0,1)-f(s,0)\right)+\partial _sf(s,v)\]
		Thus:
		\[\left \{
		\begin{array}{r c l}
		\frac{\beta}{\alpha+\beta}\int_{\reels_+}(\alpha e^{-\lambda s}+\lambda e^{-\lambda s})\pi_0(ds)=\frac{\alpha}{\alpha+\beta}\int_{\reels_+}\beta e^{-\lambda s}\pi_1(ds) \\
		\\
		\beta \int_{\reels_+}\alpha\pi_0(ds) = \alpha\int_{\reels_+}(\beta + \lambda)e^{-\lambda s}\pi_1(ds)
		\end{array}
		\right.\]
		\\	
		But $\int_{\reels_+}\pi_0(ds)=\int_{\reels_+}\pi_1(ds)=1$.
		\\
		Therefore:
		\[\left \{
		\begin{array}{r c l}
		\int_{\reels_+}(\alpha+\lambda)e^{-\lambda s}\pi_0(ds)=\alpha\int_{\reels_+}e^{-\lambda s}\pi_1(ds) \\
		\\
		\int_{\reels_+}e^{-\lambda s}\pi_1(ds)=\frac{\beta}{(\beta + \lambda)} \Rightarrow\pi_1(s)=\beta e^{-\beta s}
		\end{array}
		\right.\]
		So
		\[\left \{
		\begin{array}{r c l}
		\int_{\reels_+}e^{-\lambda s}\pi_0(ds)=\frac{\alpha \beta}{(\alpha+\lambda)(\beta + \lambda)}\Rightarrow\pi_0(s)=\frac{\alpha\beta}{\alpha-\beta}&(e^{-\beta s}-e^{-\alpha s}) \\
		\\
		\pi_1(s)=\beta e^{-\beta s}&
		\end{array}
		\right.\]	
		We finally check the measure is invariant, that is to say:
		\begin{align*}
		&\esp_{\pi} [\mathcal{A}f]=\frac{\beta}{\alpha+\beta}\int_{\reels_+}\mathcal{A}f(s,0)\pi_0(ds)+\frac{\alpha}{\alpha+\beta}\int_{\reels_+}f(s,1)\pi_1(ds)\\
		&=\frac{\alpha\beta^2}{(\alpha+\beta)(\alpha-\beta)}\int_{\reels_+}(-\alpha f(s,0)+ \partial _sf(s,0))(e^{-\beta s}-e^{-\alpha s})ds
		+\frac{\alpha\beta}{\alpha+\beta}\int_{\reels_+}(\beta(f(s,0)-f(s,1))+\partial _sf(s,1))e^{-\beta s}ds\\
		&=\frac{\alpha\beta^2}{(\alpha+\beta)(\alpha-\beta)}\left[\int_{\reels_+}(-\alpha f(s,0)+\partial _sf(s,0))e^{-\beta s}ds+\int_{\reels_+}(\alpha f(s,0)- \partial _sf(s,0))e^{-\alpha s}ds\right]\\
		&+\frac{\alpha\beta}{\alpha+\beta}\left[\int_{\reels_+}\beta f(s,0)e^{-\beta s}ds+\int_{\reels_+}(-\beta f(s,1)+ \partial _sf(s,1))e^{-\beta s}ds\right]\\
		&=\left[\frac{\alpha\beta^2(\beta-\alpha)}{(\alpha+\beta)(\alpha-\beta)}+\frac{\alpha\beta^2}{(\alpha+\beta)}\right]\int_{\reels_+} f(s,0)e^{-\beta s}ds\\
		&=0
		\end{align*}
		Moreover, $\sum\limits_{v\in\{0,1\}}\int_{\reels_+}\pi(ds,v)=\frac{\beta}{\alpha+\beta}\int_{\reels_+}\pi_0(ds)+\frac{\alpha}{\alpha+\beta}\int_{\reels_+}\pi_1(ds)=1$ completes the proof.
	\end{proof}
	We can now go on the proof of proposition~\ref{prop-born-taux}.
	\begin{proof}
		We replace $\alpha$ by $\alpha_M$ for $\overline{S_{\infty}}$ and by $\alpha_m$ for $\underline{S_{\infty}}$:
		\begin{align*}
		\proba\left(\overline{S_{\infty}}>u\right)&=\int_u^{\infty}(\pi_{\alpha_M}(ds,0)+\pi_{\alpha_M}(ds,1))\\
		&= \int_u^{\infty}\left(\frac{\alpha_M\beta}{\alpha_M-\beta} (e^{-\beta s}-e^{-\alpha_M s}) \frac{\beta}{\alpha_M+\beta}+\beta e^{-\beta s} \frac{\alpha_M }{\alpha_M+\beta}\right)ds\\
		&=\frac{\alpha_M^2 e^{-\beta u}-\beta^2 e^{-\alpha_M u}}{\alpha_M^2 - \beta^2}
		\end{align*}
		And
		\begin{align*}
		\proba\left(\underline{S_{\infty}}>u\right)&=\int_u^{\infty}(\pi_{\alpha_m}(ds,0)+\pi_{\alpha_m}(ds,1))=\frac{\alpha_m^2 e^{-\beta u}-\beta^2 e^{-\alpha_m u}}{\alpha_m^2 - \beta^2}
		\end{align*}
		So from~\eqref{born-infty}:
		\[
		\frac{\alpha_M^2 e^{-\beta u}-\beta^2 e^{-\alpha_M u}}{\alpha_M^2 - \beta^2}\proba_{\pi^w}\left(V^j_t=0\right)\leq \proba_{\pi^w}\left(S_t^i>u,V^j_t=0\right) \leq \frac{\alpha_m^2 e^{-\beta u}-\beta^2 e^{-\alpha_m u}}{\alpha_m^2 - \beta^2}\proba_{\pi^w}\left(V^j_t=0\right)
		\]
		Hence with~\eqref{born-mu}:
		\begin{align}
		\frac{\alpha_M^2 e^{-\beta u}-\beta^2 e^{-\alpha_M u}}{\alpha_M^2 - \beta^2}\frac{\beta}{\alpha_M+\beta}\leq \proba_{\pi^w}\left(S_t^i>u,V^j_t=0\right) \leq \frac{\alpha_m^2 e^{-\beta u}-\beta^2 e^{-\alpha_m u}}{\alpha_m^2 - \beta^2}\frac{\beta}{\alpha_m+\beta}
		\end{align}
	\end{proof}

	From this proposition we deduce bounds on the rates $r^{+/-}_{ij}(w)=\sum_{v,v^i=0}\mu^w_v\alpha_i(w,v)\int_{E_2}p^{+/-}(s_j)\pi^w_v(ds)$ for all $p^+$ and $p^-$ differentiable monotone. If functions $p^+$ and $p^-$ are decreasing:
	\begin{align*}
	\left(p^{+/-}(0) + \int_{0}^{+\infty}\right.&\left.\left(\frac{\alpha_m^2 e^{-\beta u}-\beta^2 e^{-\alpha_m u}}{\alpha_m^2 - \beta^2}\right)(p^{+/-})'(u)du\right)\alpha_m\sum_{v,v^j=0}\mu^w_v\\ \\
	&\leq r^{+/-}_{ij}(w) \leq\\ \\
	\left(p^{+/-}(0) + \int_{0}^{+\infty}\right.&\left.\left(\frac{\alpha_M^2 e^{-\beta u}-\beta^2 e^{-\alpha_M u}}{\alpha_M^2 - \beta^2}\right)(p^{+/-})'(u)du\right)\alpha_M\sum_{v,v^j=0}\mu^w_v
	\end{align*}
	
	We finally conclude with $p^+(s)=A_+e^{\frac{-s}{\tau_+}}$ and $p^-(s)=A_-e^{\frac{-s}{\tau_-}}$:
	\begin{align*}
	r^{+/-}_{ij}(w) &\geq\left(A_{+/-} - \int_{0}^{+\infty}\left(\frac{\alpha_m^2 e^{-\beta u}-\beta^2 e^{-\alpha_m u}}{\alpha_m^2 - \beta^2}\right)\frac{A_{+/-}}{\tau_{+/-}}e^{-\frac{u}{\tau_{+/-}}}du\right)\alpha_m\sum_{v,v^j=0}\mu^w_v\\
	&\geq A_{+/-} \alpha_m\left(1- \frac{ \frac{\alpha_m^2}{\tau_{+/-}\beta+1}- \frac{\beta^2}{\tau_{+/-}\alpha_m+1}}{\alpha_m^2 - \beta^2}\right)\sum_{v,v^j=0}\mu^w_v\\
	&\geq \frac{A_{+/-}\alpha_m\beta}{\alpha_M+\beta} \left(1- \frac{ \frac{\alpha_m^2}{\tau_{+/-}\beta+1}- \frac{\beta^2}{\tau_{+/-}\alpha_m+1}}{(\alpha_m^2 - \beta^2)}\right)
	\end{align*}
	To get the last inequality, we used the fact that $1- \frac{ \frac{\alpha_m^2}{\tau_{+/-}\beta+1}- \frac{\beta^2}{\tau_{+/-}\alpha_m+1}}{(\alpha_m^2 - \beta^2)}\geq 0$ and proposition~\ref{ine-mu}. We can do the same to major $r^{+/-}_{ij}(w)$:
	\begin{align*}
	r^{+/-}_{ij}(w)\leq \frac{A_{+/-}\alpha_M\beta}{\alpha_m+\beta} \left(1- \frac{ \frac{\alpha_M^2}{\tau_{+/-}\beta+1}- \frac{\beta^2}{\tau_{+/-}\alpha_M+1}}{(\alpha_M^2 - \beta^2)}\right)
	\end{align*}
	But we showed that if $r^+_{ij}(w)<r^-_{ij}(w)$ for all $w$ we get that the limit process $W_t$ is recurrent positive so it is the case if:
	\[\frac{A_+\alpha_M}{\alpha_m+\beta} \left(1- \frac{ \frac{\alpha_M^2}{\tau_+\beta+1}- \frac{\beta^2}{\tau_+\alpha_M+1}}{(\alpha_M^2 - \beta^2)}\right) < \frac{A_-\alpha_m}{\alpha_M+\beta} \left(1- \frac{ \frac{\alpha_m^2}{\tau_-\beta+1}- \frac{\beta^2}{\tau_-\alpha_m+1}}{(\alpha_m^2 - \beta^2)}\right)\]
	Finally we get the following simple condition:
	\[\frac{ \alpha_M^2 A_+\tau_+(\alpha_M\tau_++\beta\tau_++1)(\tau_-\alpha_m+1)(\tau_-\beta+1)}{\alpha_m^2 A_-\tau_-(\alpha_m\tau_-+\beta\tau_-+1)(\tau_+\alpha_M+1)(\tau_+\beta+1)} < 1 \]
	If $p^+$ and $p^-$ are not monotone, we can get a similar condition separating intervals where they are increasing or decreasing.\\\\
	
	Finally, previous results show that in our model weights can diverge although rates are bounded and we can give simple explicit condition on parameters for which they don't diverge. This is the first time, to our knowledge, that such a condition can be given without any homeostatic mechanisms added. Some analytical studied previously needed to add some constraints in order to bound weights and obtained results depending on the spike correlation matrix they were not able to control~\cite{kempter_hebbian_1999,gilson_emergence_2009,ocker_self-organization_2015}. With such a condition, our model becomes ready to use {\color{black}being aware of criticizes we present in the sixth section}.

	\section{Simulations}
	As shown in the appendix~\ref{uniqueness_dim2}, we can find the Laplace transform of $\pi$, the invariant measure of the fast process. However, inverting it analytically for a network of N neurons, N too large, needs too heavy computations. Hence, we apply our results in a network of 2 neurons and then simulate a bigger network. 
	%	We will compute a numerical inversion in 2 dimensions and compare it to simulation results.
%	Nevertheless, we can 
	%	find some of characteristics of $\pi$ and even
%	study the process thanks to this Laplace transform. 
%	Thus, we can use it in our slow fast result to compute the jump rates of the slow variable.
	But first let remind us the parameters present in our model.
	\subsection{Biologically coherent parameters:}
	Even if simple, our model depends on many parameters. First, let's recall the probability to jump:
	\begin{align*}
	p^+(s)=A_+e^{-\frac{s}{\tau_+}}\ \ and\ \ p^-(s)=A_-e^{-\frac{s}{\tau_-}} 
	\end{align*}
	Then let's detail the function $\xi_i$ we used in our simulations. We used the same $\xi_i=\xi$ for all neurons, $\sigma >0$ and $\theta>0$:
	\[\xi_i(x)=\xi(x)=\frac{S_0}{1+e^{-\sigma (x-\theta)}}+\alpha_m\]
	Our parameters are then: $\epsilon, A_+, A_-, \tau_-, \tau_+, \sigma, \theta, \beta, \alpha_m\ \text{and}\ \alpha_M$. Time of influence of a spike 10ms so $\beta{\sim} 0.1$. Firing rates of neurons are bounded by $\alpha_m{\sim} 0.01$ and $\alpha_M{\sim} 1$. STDP parameters are in the following range: $\tau_{+/-}{\in}[1,50],\ A_{+/-}{\in}[0,1]$. Finally, $S_0= \alpha_M, \sigma = 0.3,\ \theta = \frac{ln(\alpha_M/\alpha_m - 1)}{\sigma}\ \text{and}\ \epsilon{\leq}0.01$.\\
	
	Functions $p^+$ and $p^-$ enable to be close to biological experiments~\cite{bi_synaptic_1998}:
	\begin{figure}[h!]
		\includegraphics[scale=0.38]{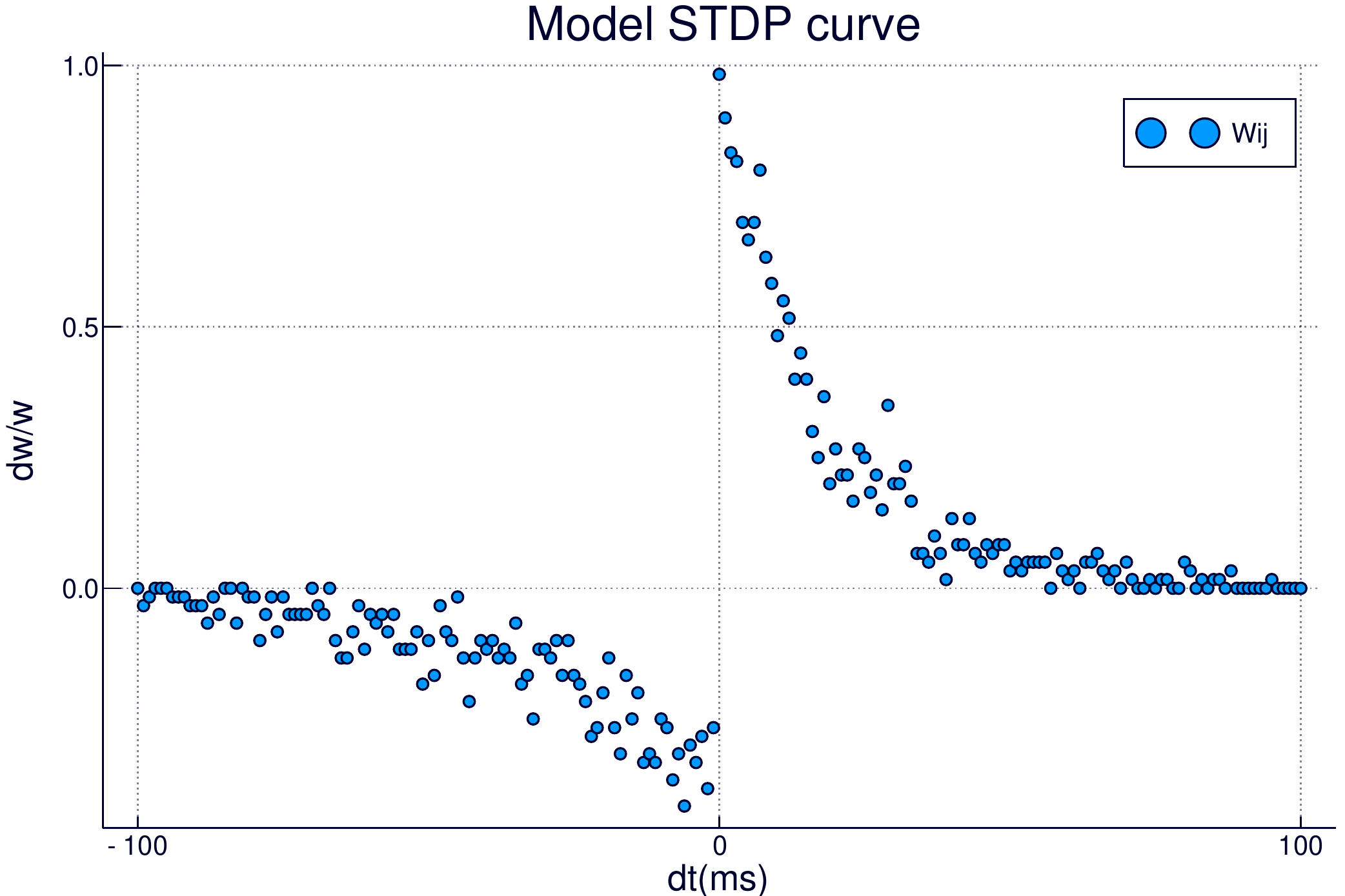}\hspace{2em}
		\includegraphics[scale=0.25]{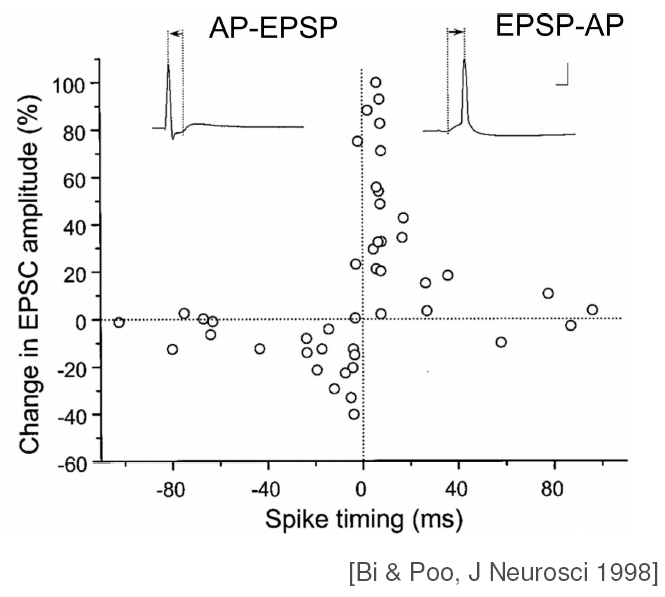}
		\caption{Bi-Poo experiment on our model compare to the real one. Parameters used here are: $ A_+{=}1,\ A_-{=}0{.}4,\ \tau_-{=}2\tau_+{=}34ms$ as in~\cite{gilson_spectral_2012}.}
	\end{figure}

	\newpage
	\subsection{First applications of our results}
	In the simple case of \eqref{proba-poids} we get:
	\begin{align*}
	r^+_{ij}(w)&=\sum_{v,v^j=0}\mu^w_v\alpha_j(w,v)\int_{E_2}p^+(s_i)\pi^w_v(ds)\\
	&=\sum_{v\in I,v^j=0}\mu^w_v\alpha_j(w,v)\int_{E_2}A_+e^{-\frac{s_i}{\tau_+}}\pi^w_v(ds)\\
	&=\sum_{v\in I,v^j=0}\mu^w_v\alpha_j(w,v)A_+\mathcal{L}\{\pi^w_v\}(0,...,0,\underbrace{\frac{1}{\tau_+}}_i,0,...,0)
	\end{align*}
	And
	\[r^-_{ij}(w) = \sum_{v\in I,v^i=0}\mu^w_v\alpha_i(w,v)A_-\mathcal{L}\{\pi^w_v\}(0,...,0,\underbrace{\frac{1}{\tau_-}}_j,0,...,0)\]
	\noindent\textbf{\small One weight free and 2 neurons:}\\
		In this example of one weight free and 2 neurons, we get a birth and death process with $w^{21}$ fixed, $w{=}(w^{12},w^{21})$. We can find the explicit stationnary distribution of the weights in that case. From previous computations we have:
		\begin{align*}
		w^{12}&\rightarrow w^{12}+\Delta w:\\
		&r_+(w^{12}) = A_+ \left[\mu^w_{00}\alpha_2(w,00)\mathcal{L}(\pi^w_{00})\left(\frac{1}{\tau_+},0\right)+\mu^w_{10}\alpha_2(w,10)\mathcal{L}(\pi^w_{10})\left(\frac{1}{\tau_+},0\right)\right]\\
		w^{12}&\rightarrow w^{12}-\Delta w:\\
		& r_-(w^{12}) = \mathbb{1}_{]\Delta w, +\infty[}(w^{12})A_-\left[\mu^w_{00}\alpha_1(w,00)\mathcal{L}(\pi^w_{00})\left(0,\frac{1}{\tau_-}\right)+\mu^w_{01}\alpha_1(w,01)\mathcal{L}(\pi^w_{01})\left(0,\frac{1}{\tau_-}\right)\right]
		\end{align*}
		Hence, it is similar to a birth process on $\nat$ with 0 reflecting. In order to study the conditions for transience and recurrence, we use the following theorem which gather some results of the four first sections of~\cite{karlin1957classification} with its notations.
		\begin{theoreme}\label{theo_birth_death}
			Suppose $X_t$ is a birth and death process on $\nat$ with birth rates $\lambda_k>0$ for all $k\in \nat$ and death rates $\mu_k>0$ for all $k\in \nat^*$ and $\mu_0=0$.
%			Let define for all $n\in \nat^*$ $\pi_n=\prod_{j=1}^{i}\frac{r_+(j-1)}{r_-(j)}$ and $\pi_0=0$.
			Then~\cite{karlin1957classification} gives the following classification:
			\begin{itemize}
				\item [(a)] The process is ergodic if and only if  $\sum_{i=1}^{+\infty}\prod_{j=1}^{i}\frac{\mu_j}{\lambda_j}=+\infty$ and $\sum_{i=1}^{+\infty}\prod_{j=1}^{i}\frac{\lambda_j-1}{\mu_j}<+\infty$. In this case, there exists a unique $\theta$ invariant measure given by:
				\[\theta(i) =\theta(0) \prod_{j=1}^{i}\frac{\lambda_{j-1}}{\mu_j}\]
				With 
				\[\theta(0) = \frac{1}{1+\sum_{i=1}^{+\infty}\prod_{j=1}^{i}\frac{\lambda_{j-1}}{\mu_j}}\]
				\item [(b)] The process is null recurent if and only if  $\sum_{i=1}^{+\infty}\prod_{j=1}^{i}\frac{\mu_j}{\lambda_j}=+\infty$ and $\sum_{i=1}^{+\infty}\prod_{j=1}^{i}\frac{\lambda_j-1}{\mu_j}=+\infty$
				\item [(b)] The process is transient if and only if  $\sum_{i=1}^{+\infty}\prod_{j=1}^{i}\frac{\mu_j}{\lambda_j}<+\infty$ and $\sum_{i=1}^{+\infty}\prod_{j=1}^{i}\frac{\lambda_j-1}{\mu_j}=+\infty$
			\end{itemize}
		\end{theoreme}
	
		In order to apply this theorem to our example, we prove the following corollary.
		\begin{corollaire}\label{cor_birth_death}
			Suppose assumptions of theorem~\ref{theo_birth_death} hold. Suppose in more that $\lambda_k$ and $\mu_k$ converge respectively towards $\lambda$ and $\mu$ when $k\rightarrow +\infty$. Then $X_t$ is ergodic iff $0<\lambda<\mu$ and transient if $\lambda>\mu>0$.
		\end{corollaire}
		\begin{proof}
			Let prove only the ergodic case as the proof for the transient one is similar. Suppose that $0<\lambda<\mu$. Thus, for all $\epsilon >0$, $\exists\ k_0\in \nat$ such that for all $j>k_0$, $\frac{\lambda_{j-1}}{\mu_j}\leq \frac{\lambda}{\mu}-\epsilon=l_{\epsilon}$ and $\frac{\lambda_j}{\mu_j}\leq l_{\epsilon}$. Taking $l_\epsilon<1$ gives the result according to the d'Alembert's ratio test.
		\end{proof}
		\begin{remarque}
{	\it		The case $\lambda=\mu>0$ is more complex as it will depend on the way $(\lambda_k)$ and $(\mu_k)$ converge.}
		\end{remarque}
		
		We come back to our example. 
		\begin{proposition}
			$r_+(w^{12},w^{21})$ and $r_-(w^{12},w^{21})$ are strictly positive and converge respectively to $R_+(\alpha_M,w^{21})>0$ and $R_-(\alpha_M,w^{21})>0$ when $w^{12}\rightarrow \infty$.
		\end{proposition}
		\begin{proof}
			First, $\alpha_1(w,00)=\alpha_2(w,00)=\xi(0)=\alpha_m$ and $\alpha_1(w,01)=\xi(w^{21})$ don't depend on $w^{12}$.
			Second, $x\mapsto \mathcal{L}\pi^w_v(0,x)$, $x\mapsto \mathcal{L}\pi^w_v(x,0)$ and $\mu_v^w$ depend on $w^{12}$ only through $\alpha_2(w,10)=\xi(w^{12})$. But $\lim_{w^{12}\rightarrow +\infty}\xi(w^{12})=\alpha_M$ so $\vec{\mu}^w$ converges to $\vec{\mu}$ solution of~\eqref{murel2} with $\alpha_{01}^{11}=\alpha_1(w,01)=\xi(w^{21})$ and $\alpha_{10}^{11}=\lim_{w^{12}\rightarrow +\infty}\alpha_2(w,10)=\lim_{w^{12}\rightarrow +\infty}\xi(w^{12})=\alpha_M$. Concerning $x\mapsto \mathcal{L}\pi^w_v(0,x)$, we can fix $x=x_0$ and call $f_v(\xi(w^{12}))=\mathcal{L}\pi^w_v(0,x_0)$. Computations of~\ref{eq_lap_0} show that for all $v \in \{0,1\}^2$, $0<f_v(y)<\infty$ for all $y\in [\alpha_m,\alpha_M]$ and is continuous as a positive bounded rational fraction. Hence, $w^{12}\mapsto f_v(\xi(w^{12}))$ is continuous by composition. We conclude that $\lim_{w^{12}\rightarrow +\infty}\mathcal{L}\pi^w_v(0,x_0)=f_v(\alpha_M)$ and:
			\begin{align*}
			\lim_{w^{12}\rightarrow \infty}r_+(w) &=\lim_{w^{12}\rightarrow \infty} A_+ \left[\mu^w_{00}\alpha_m\underbrace{\mathcal{L}(\pi^w_{00})\left(\frac{1}{\tau_+},0\right)}_{\underset{w\rightarrow\infty}{\longrightarrow}f_{00}(\alpha_M)}
			+\mu^w_{10}\alpha_M\underbrace{\mathcal{L}(\pi^w_{10})\left(\frac{1}{\tau_+},0\right)}_{\underset{w\rightarrow\infty}{\longrightarrow}f_{10}(\alpha_M)}\right]
			\\
			&=A_+(\mu_{00}\alpha_mf_{00}(\alpha_M)+\mu_{10}\alpha_Mf_{10}(\alpha_M))
			\\
			&=R_+
			\end{align*}
			It is similar for $x\mapsto \mathcal{L}\pi^w_v(x,0)$.
			\begin{align*}
			\lim_{w^{12}\rightarrow \infty}r_-(w) =R_-
			\end{align*}
		\end{proof}
		Hence, by corollary~\ref{cor_birth_death}, $R_+<R_-$ ensures the process $w^{12}_t$ admits a unique invariant measure $\theta$: \[\theta(i\Delta w) =\theta(\Delta w) \prod_{j=2}^{i}\frac{r_+((j-1)\Delta w)}{r_-(j\Delta w)}\]
		With 
		\[\theta(\Delta w) = \frac{1}{1+\sum_{i=1}^{+\infty}\prod_{j=2}^{i}\frac{r_+((j-1)\Delta w)}{r_-(j\Delta w)}}\]\\
		
		We then wonder when this condition holds and we did simulations with parameters in the range of biological ones. Practically, explosion of the weight reflects the fact that LTP wins over LTD. Some studies has tried to tackle question of the relationship between STDP curve parameters, $\tau_{+/-}$ and $A_{+/-}$, and the balance of LTP and LTD. They showed that when the integral of the
		STDP window is enough biased toward depression the system is intrinsically stable~\cite{kempter_hebbian_1999,kempter_intrinsic_2001,izhikevich_relating_2003}. {\color{black}In our case, we can find examples for which the "enough" is important}. For instance with the following parameters, we get an explosion of $w^{12}$ when depression wins against potentiation:
		\[\beta = 0.1,\ \alpha_m = 0.01,\ \alpha_M = 1,\ \tau_+ = 17ms,\ \tau_- = 34ms,\ A_- = 0.7,\ A_+ = 0.8,\ \epsilon = 10^{-4}\] 
		We took $p_{\epsilon}^{+/-}=\epsilon p^{+/-}$. When $\epsilon$ is small enough ($\leq 10^{-4}$) simulations agrees with analytical results. That is to say $w^{12}$ diverges when $w^{21}<25$ and doesn't diverge when $w^{21}>25$:
		\begin{figure}[h!]
			\includegraphics[height=5cm,width=8cm]{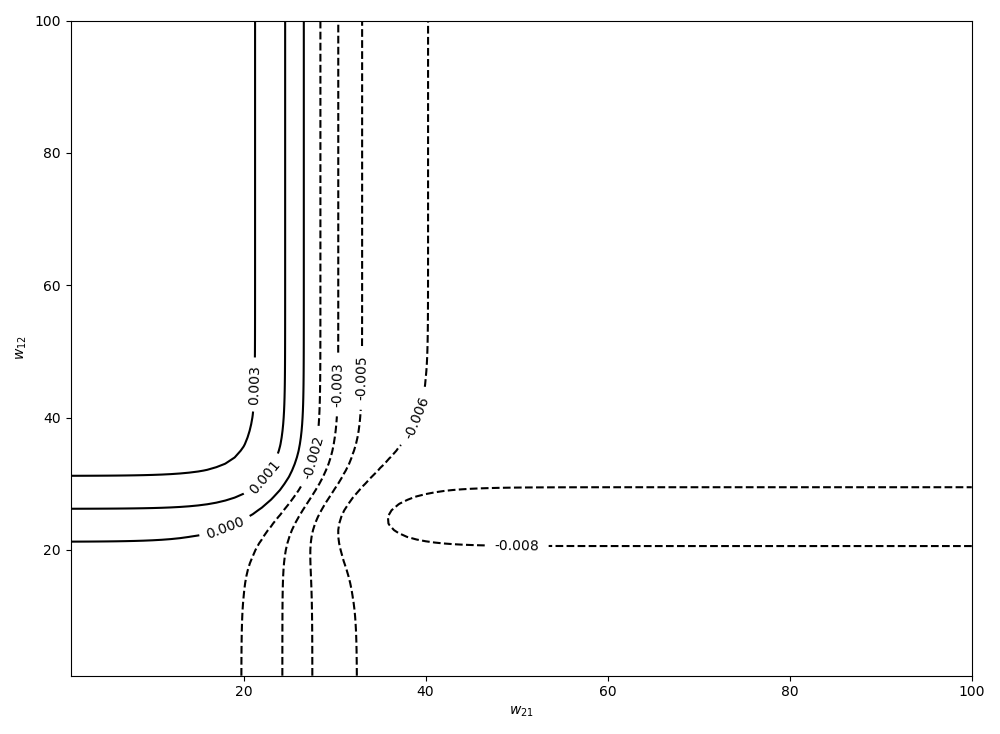}
			\hfill
			\includegraphics[height=5cm,width=8cm]{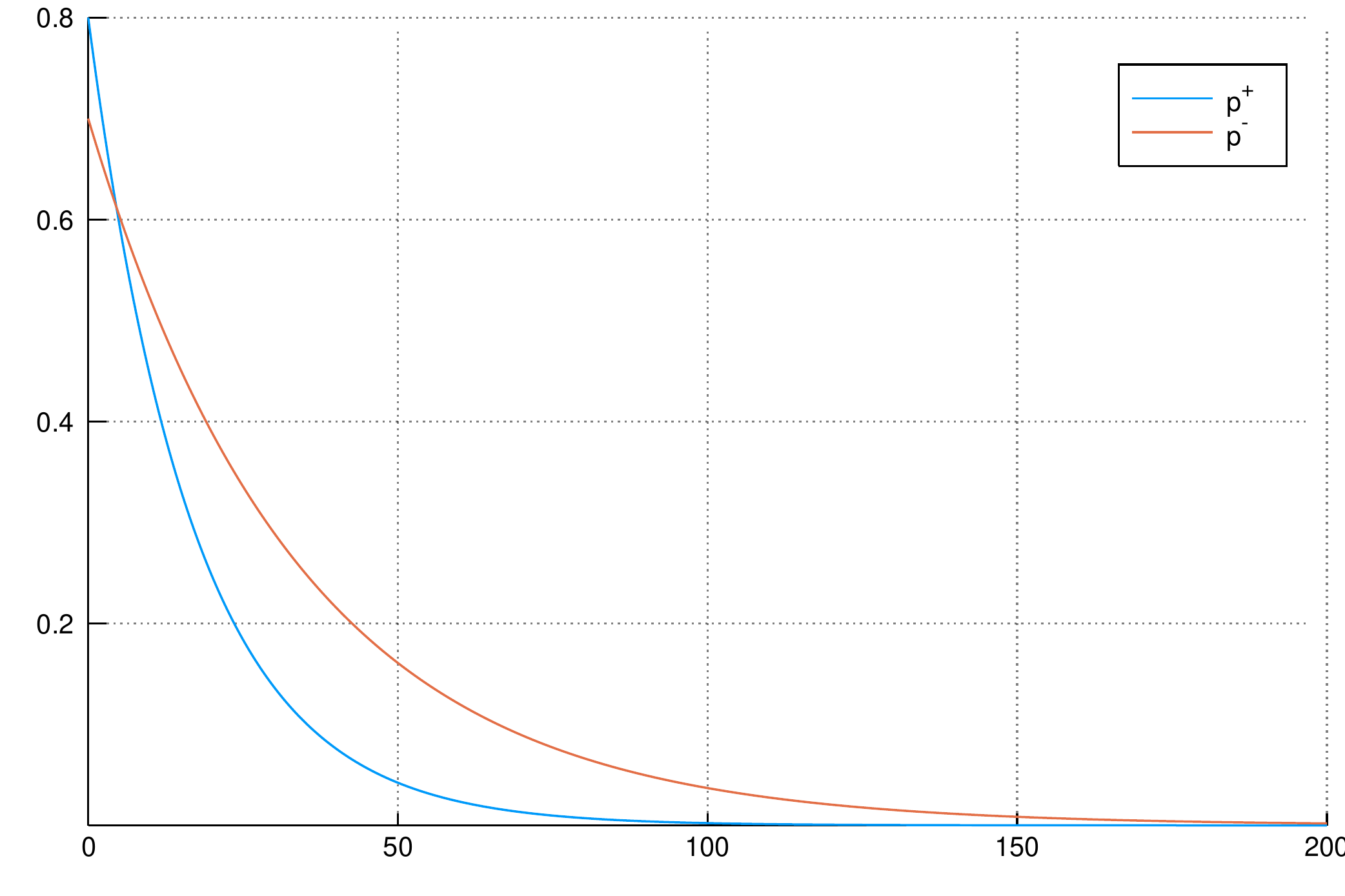}
			\caption{Plot of $r_{+}(w){-}r_{-}(w)$(left) and plot of $p^{+},\ p^{-}$ on the same graph(right)}
		\end{figure}
	
		\begin{figure}[h!]
			\includegraphics[height=5cm,width=8cm]{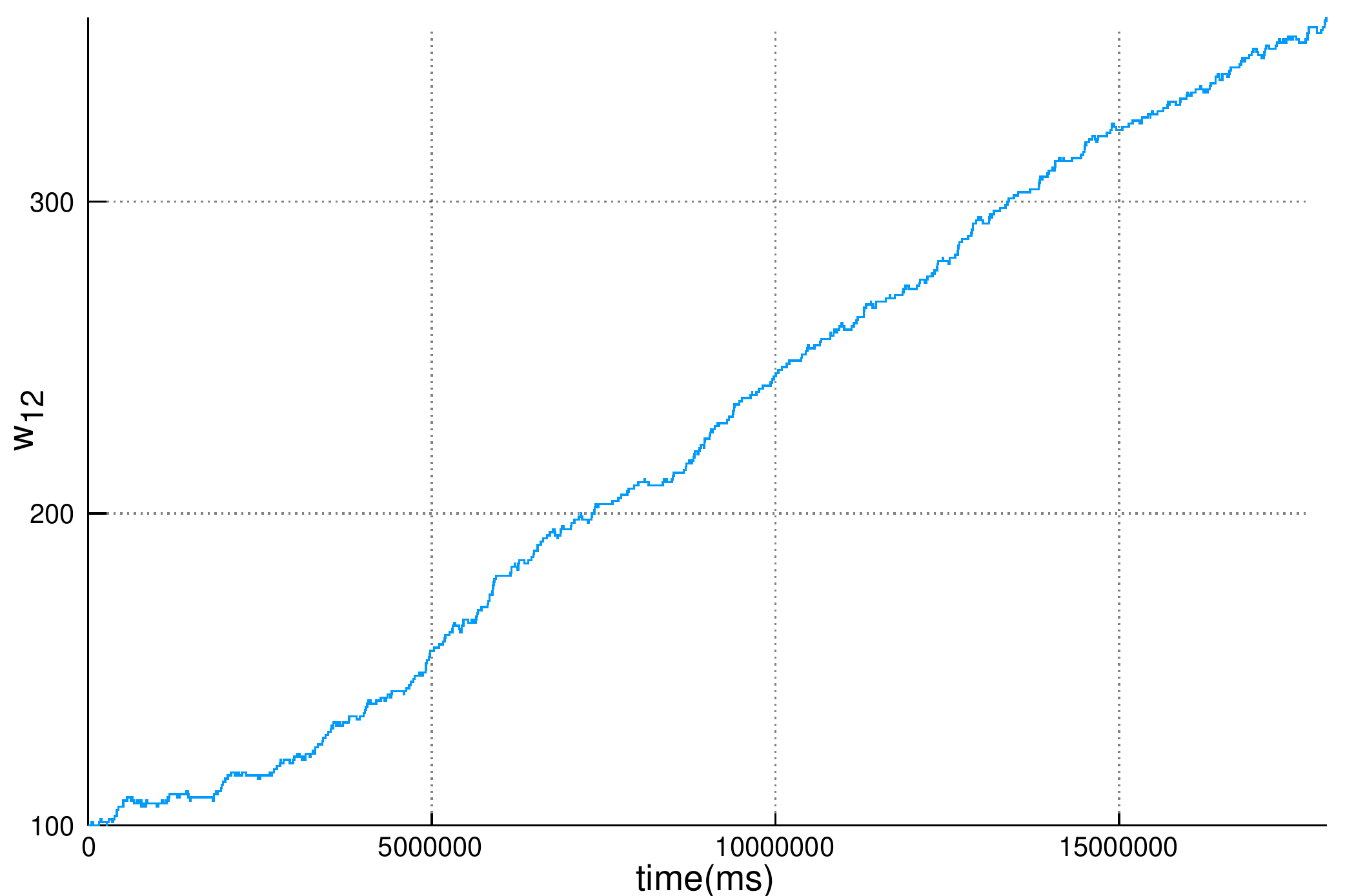}
			\hfill
			\includegraphics[height=5cm,width=8cm]{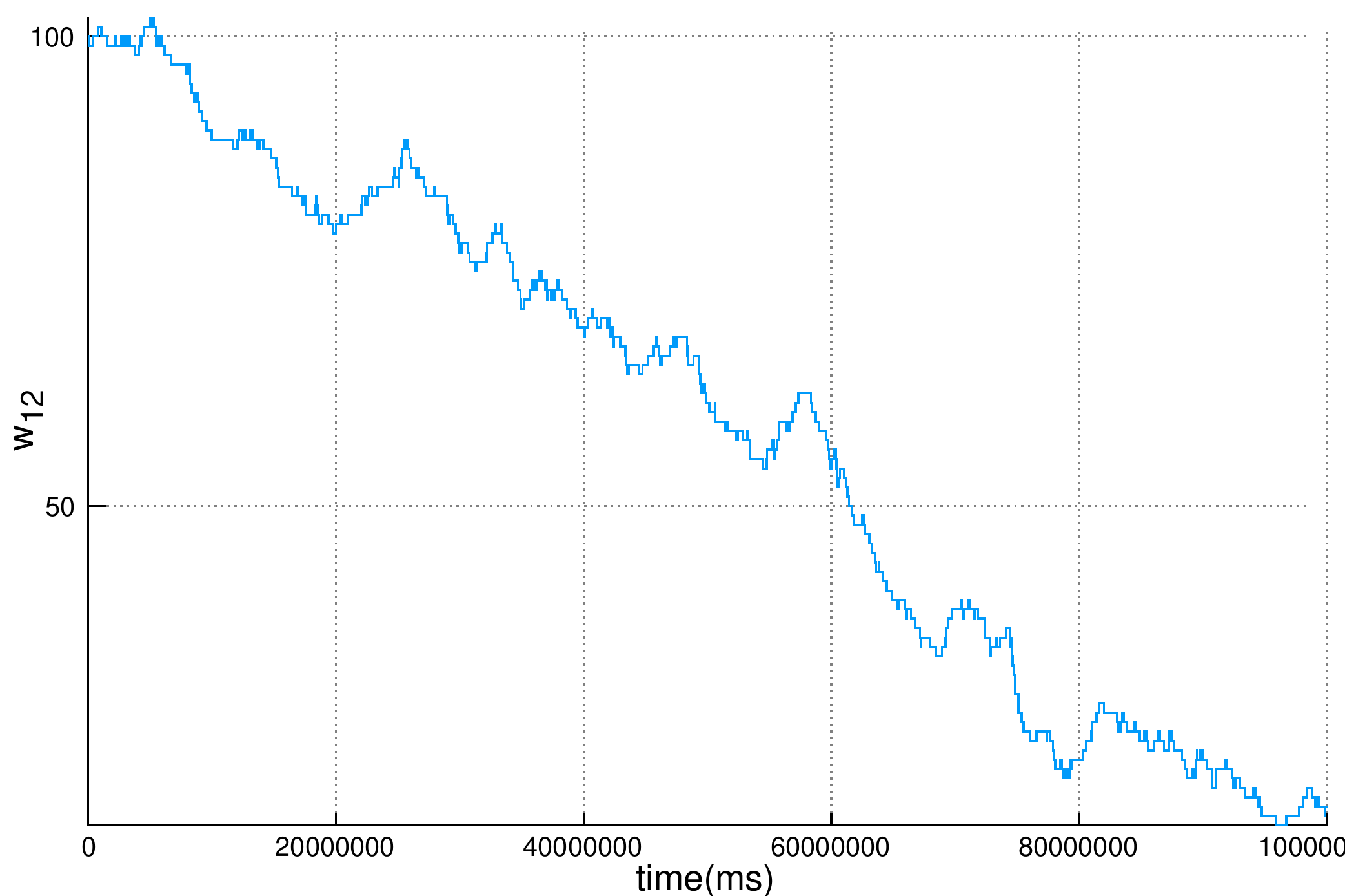}
			\caption{Evolution of the weight $w^{12}$ when $w^{21}$ is fixed at $15$ (left) and 30 (right) and $\epsilon=10^{-4}$}
		\end{figure}
		%				\noindent Our limit model enabled us to tune parameters and get divergence in a case where depression wins over potentiation in average.
		
		\begin{remarque}
			{\it We can even get divergence when $p^+(s)<p^-(s)$ for all $s\in \reels_+$}
		\end{remarque}
		
%		We can compare analytic and numerical results of the density in the case of non divergence:
%		\\
%		\begin{figure}[h!]
%			\centering
%			\includegraphics[height=5cm,width=8cm]{density-60.pdf}
%			\caption{Density of the discrete Markov chain $(W_n)_{n\geq 0}$ associated to the Markov process $(W_t)_{t\geq 0}$}
%		\end{figure}
%	
		\newpage	
%	\end{example}
%	\begin{example}
		\subsubsection*{Example with 2 excitatory neurons}
		\vspace{0.5em}
		
		Let's apply this result in a network of 2 excitatory neurons. First, we denote $w = (w^{12},w^{21})$ since the diagonal elements are null. We are interested in the sign of the limit of $\sup_{\lVert w \rVert\geq r} (r^{ij}_+(w)-r^{ij}_-(w))$ which is equivalent to $\sup_{\lVert w \rVert\geq r} \left(\eta(w)\right)^{ij}$ (see~\ref{rem-eta}), when $r \rightarrow \infty$, in order to use corollary~\ref{cor-sup} to study stability of weights. We first show this limit exists and then compute it to determine parameters for which we don't have weights divergence.\\
		In order to show the existence of the limit, we first recall that $w$ is only present in neurons' rates. Thus, thanks to the sigmoid, these rates are bounded and when one of the components of $w$ goes to $\infty$, rates in which it plays a role tends to the upper bound of the sigmoid, $\alpha_M$, since all neurons are excitatory ones. For instance:
		\[\alpha_{1}(w,01)=\xi(w^{21})\underset{w^{21}\rightarrow\infty}{\longrightarrow}\alpha_M\]
		Therefore, we can separate the space $\reels_+\times\reels_+$ as following the intuition given by the graph of $\left(\eta(w)\right)^{12}$ for instance:\\
		
		\begin{figure}[h!]
			\centering
			\includegraphics[scale=0.4]{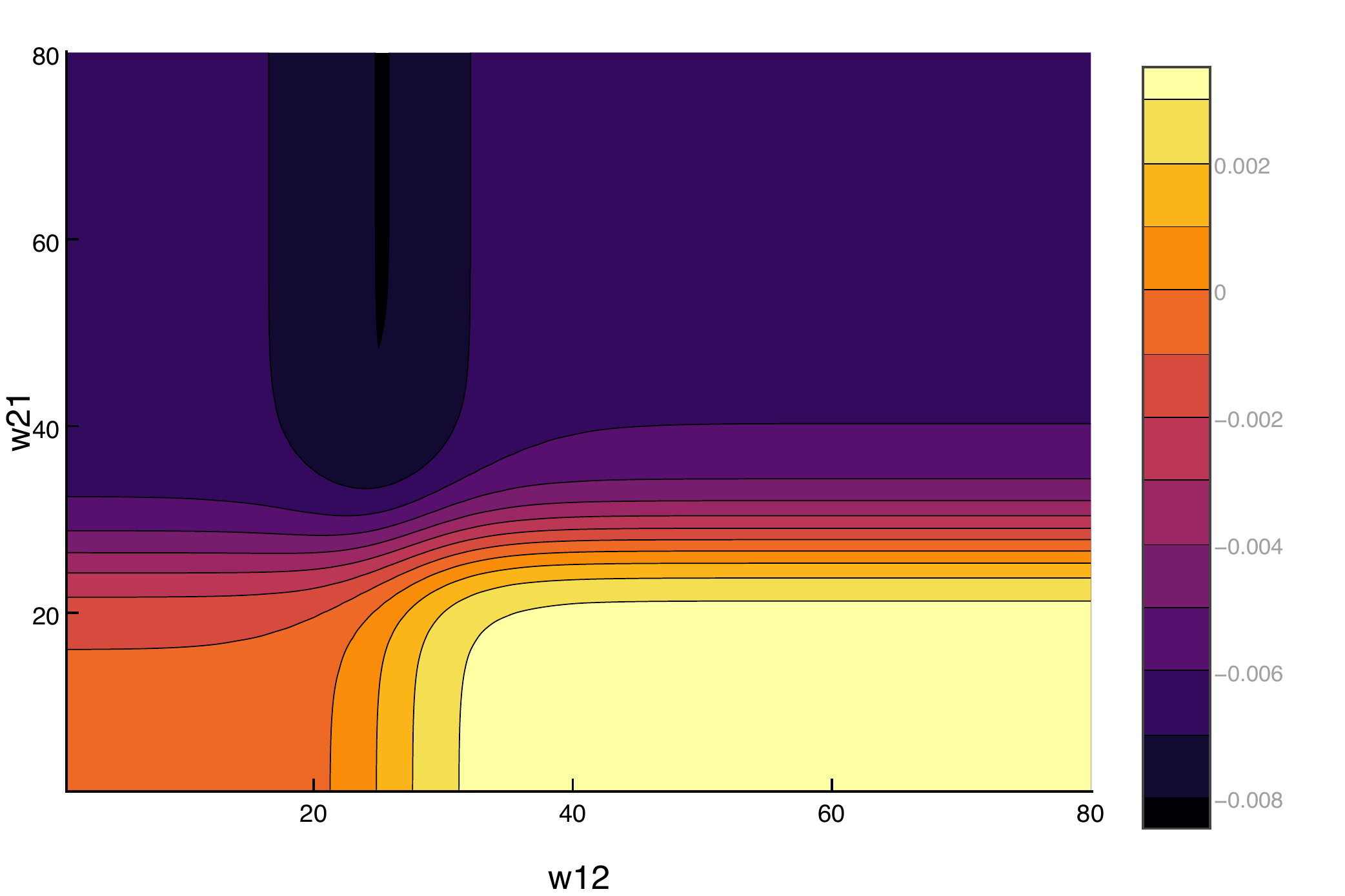}
			\caption{$\eta^{12}(w)$ when $A_+=A_-=0,8$ and $\tau_-=2\tau_+=34ms$}
		\end{figure}
%		\newpage
		So the separation looks like this:
		\begin{figure}[h!]
			\centering
			\includegraphics[scale=0.3]{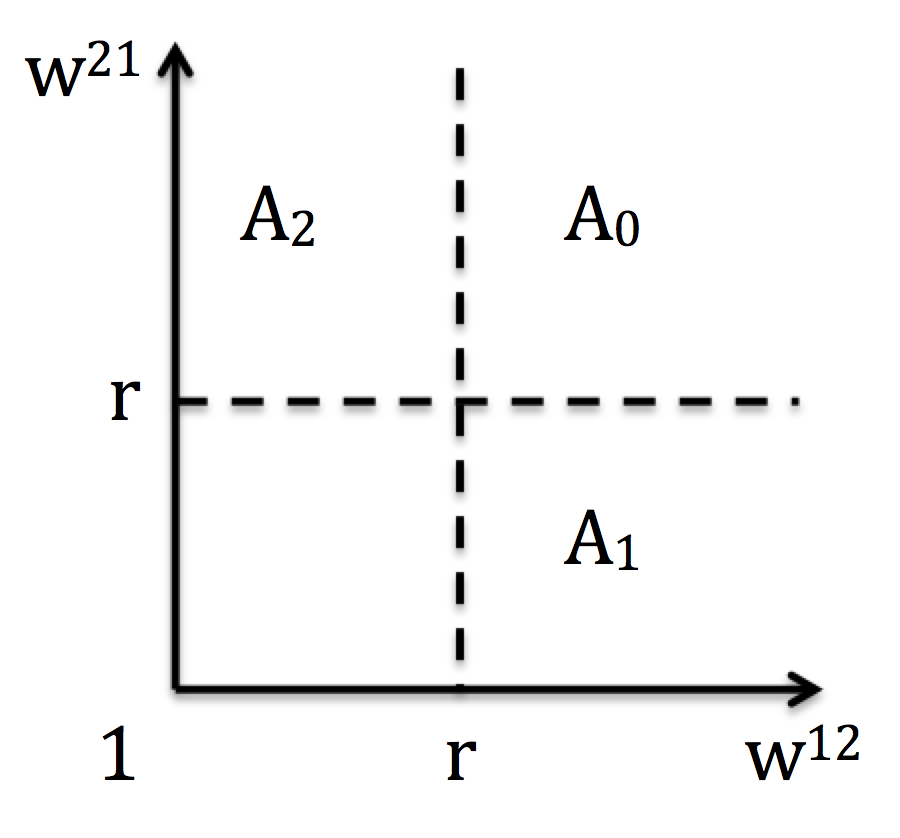}
		\end{figure}
		
		As showed in the appendix~\ref{uniqueness_dim2}, we can compute the Laplace transforms $\mathcal{L}\{\pi^w_v\}(\lambda_1,\lambda_2)$ for fixed $w$. If we introduce the dependence on $w$, it will be in rate terms such as $\alpha_{01}^{11}=\alpha_{1}(w,01)$ for example. As they are not numerous, we finish this kind of translation: $\alpha_{01}^{11}=\alpha_{1}(w,01)=\xi(w^{21})$, $\alpha_{10}^{11}=\alpha_{2}(w,10)=\xi(w^{12})$, $\alpha_{00}^{10}=\alpha_{00}^{01}=\xi(0)=\alpha_m$ and $\alpha_{01}^{00}=\alpha_{01}^{00}=\alpha_{11}^{10}=\alpha_{11}^{01}=\beta$. So we can rewrite $\eta$ as a function of $\alpha_{01}^{11}(w^{21})$ and $\alpha_{10}^{11}(w^{12})$. Therefore, when $r \rightarrow \infty$, $\eta(\alpha_{10}^{11}(w^{12}),\alpha_{01}^{11}(w^{21}))\rightarrow \eta(\alpha_M,\alpha_M)$ on $\mathcal{B}_0$. The sup of $\eta$ becomes $\sup_{\alpha_m \leq \alpha \leq \alpha_M}\eta(\alpha_M,\alpha)$ on $A_1$ and $\sup_{\alpha_m \leq \alpha \leq \alpha_M}\eta(\alpha,\alpha_M)$ on $A_2$. We conclude with 
		\[\lim_{r\ \infty}\sup_{\|w\|\geq r}\eta=\max\left(\sup_{\alpha_m \leq \alpha \leq \alpha_M}\eta(\alpha,\alpha_M),\sup_{\alpha_m \leq \alpha \leq \alpha_M}\eta(\alpha_M,\alpha)\right)\]
		We can compute numerically this limit in function of $A_-$ and $\tau_-$:
		\begin{figure}[h!]
			\centering
			\includegraphics[scale=0.6]{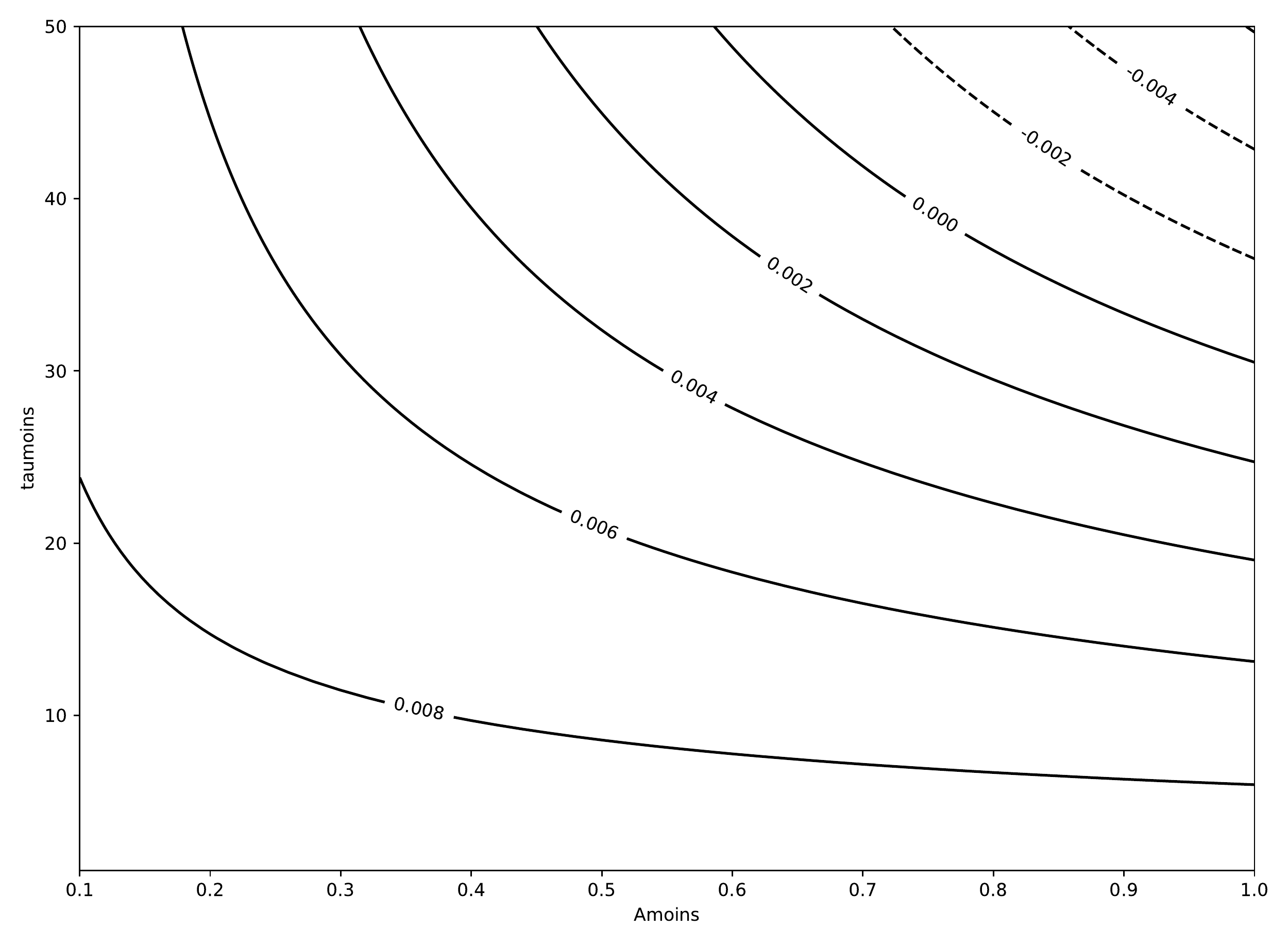}
			\caption{$\sup \eta^{12}$ when $\|w\|\rightarrow \infty$ for $A_+=0,2$ and $\tau_+=17ms$}
		\end{figure}
		\newpage
		We note that we need a really small value of $A_+$ compared to the one of $A_-$  to satisfy the condition of positive recurrence. However, such a difference doesn't seem to be needed in simulations. {\color{black}Indeed, we can have numerically positive recurrence for any parameters $A_+$ between 0 and 1. }
		\begin{remarque}
			{\it The condition for null recurrence given in~\cite{menshikov2016non} result in $\eta_{ij}=0$ for all $i,\ j$ in our case. Condition for transience leads to the exact opposite of the one of corollary~\ref{cor-sup}:
			\begin{align}
			\begin{split}
			\lim_{r\rightarrow +\infty}\sup_{w\in \Sigma,\lVert w\rVert\geq r}(\eta(w))^{ij} \geq 0,\ \forall i,j\\
			\text{And}\ \exists (k,l),\ j\neq i\ s.t.\ \lim_{r\rightarrow +\infty}\sup_{w\in \Sigma,\lVert w\rVert\geq r}(\eta(w))^{kl} > 0	
			\end{split}
			\end{align}
			It would be interesting to try to have a larger range of values of parameters for which we are in the null recurrence case, and we need another plasticity rule to do so (with the condition of~\cite{menshikov2016non}).}
		\end{remarque}
	\subsection*{10 neurons:}
	\noindent When depression is really higher than potentiation, weights seem to converge to a stationary distribution and have such trajectories:\\
	\begin{figure}[h!]
		\centering
		\includegraphics[scale=0.7]{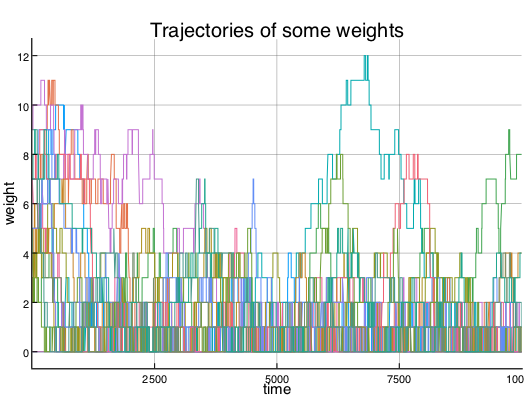}
	\end{figure}
	
	\begin{figure}[h!]
		\centering
		\includegraphics[scale=0.4]{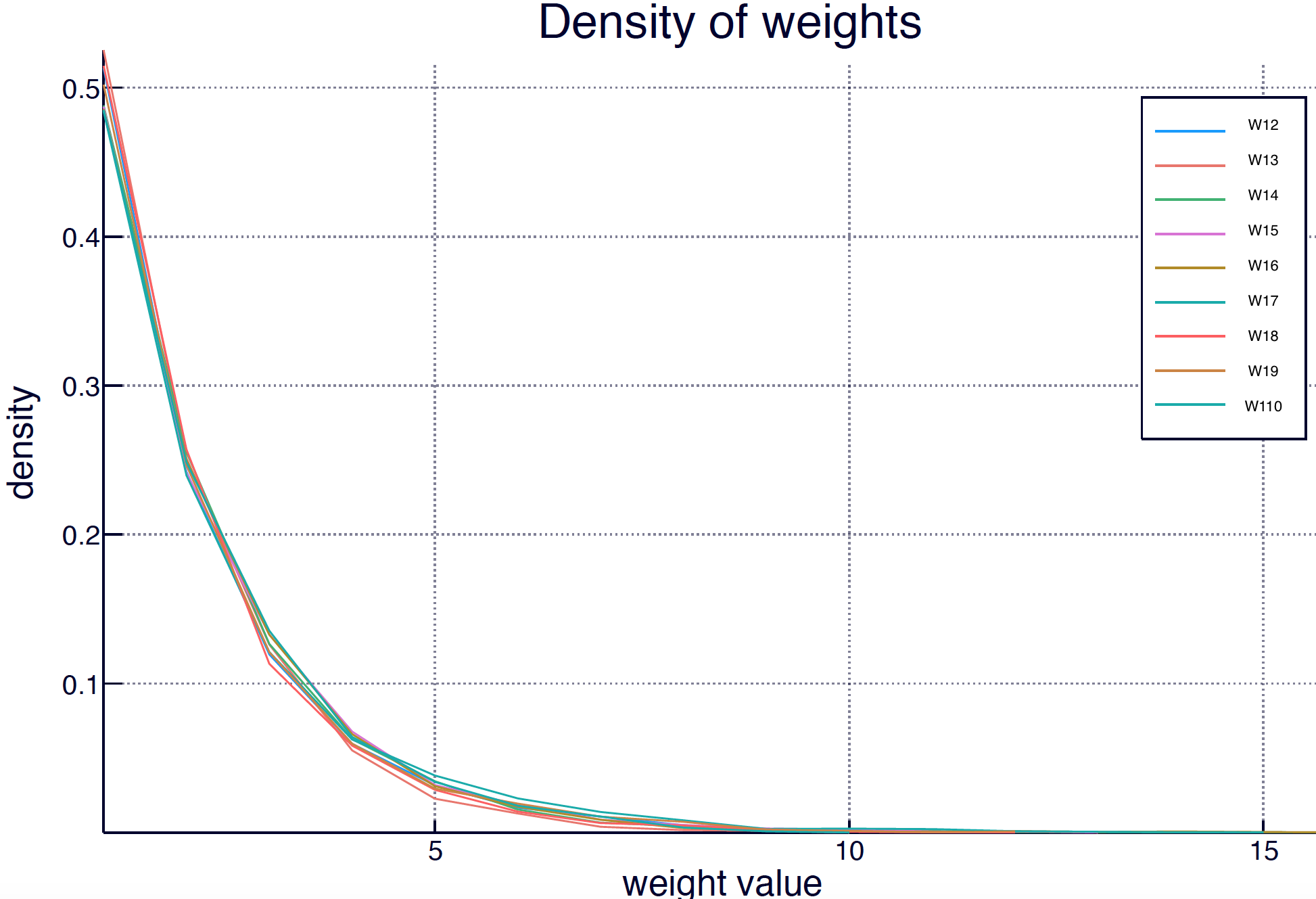}
	\end{figure}

	%				\begin{center}
	%					\includegraphics[height=2.5cm,width=4cm]{weight_traj.png}
	%				\end{center}
	\newpage

	\noindent However, initial weights can play an important role. With parameters $A_+{=}0.8,\ A_-{=}0{.}9,\ \beta{=}1, \alpha_m{=}0.01, \alpha_M{=0.5}$ and $\epsilon{=}0.1$, we have no divergence in short time with low initial weights and selection of one weight from big initial ones, $\bf{W_0^{i1}=50}$:\\
	\begin{figure}[h!]
		\centering
		\includegraphics[height=6cm,width=10cm]{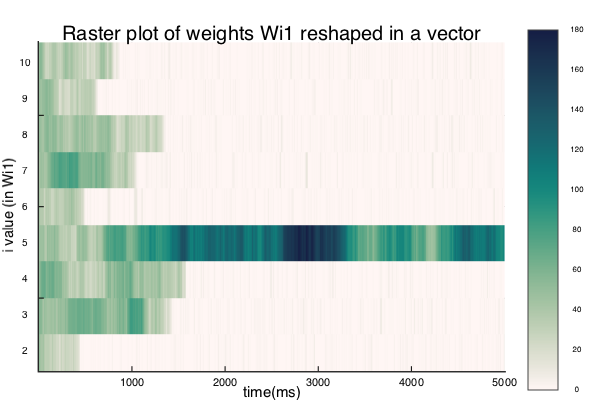}
	\end{figure}

	\noindent The selected weight is different from one trajectory to another.
	
	\begin{remarque}
		{\it We have chosen 10 neurons for plotting constraints. Thousands of them are easily simulated.}
	\end{remarque}
	This kind of phenomenon is called winner take all dynamics in~\cite{litwin_formation_2014} where they prevent them using iSTDP. The reason to avoid them is that it prevents new assemblies to be formed.
	\section{Discussion}
	\subsubsection*{Mathematical results}
	Based on a well known neural network model, we added plasticity in order to get insight on the combined neurons - weights dynamics. 
%	This model was derived because it satisfies biological constraints while being simple enough to be analysed mathematically and numerically. This enabled us to
	We could analyse plasticity on the slow time scale of weights dynamics compared to the neurons ones, thus producing a simplified model. This latter gives the weights dynamics under the stationary distribution of the fast process and is a continuous time Markov jump process on the state space of weights with non homogeneous in space jump rates. Such processes are hard to deal with and current results are given in~\cite{menshikov2016non}. Moreover, even if we could prove existence and uniqueness of the invariant measure of the fast process, we were not able to express it explicitly. Thus, it is even harder to analyse the limit model. However, we can compute its Laplace transform in small networks, we didn't try more than 2 but it should not be too hard for more. The problem will nevertheless become quickly harder as it consists in inverting a $2^N$ square matrix for a given $w$ and as soon as $w$ change, this computation need to be done again. Here, making use of bounds on jump rates of neurons, we are able to give conditions of stability, but {\color{black}we emphasize it is only sufficient ones. To know if we need additive terms, depending on weights for instance or just hard bounds, in order to avoid divergence in the context of biological parameters is still under study.} 
	
	\subsubsection*{Simulation results}
	For small networks (2 neurons) and in the case of a STDP rule following the classical STDP curve \cite{bi_synaptic_1998}, we computed Laplace transform of the stationary distribution. We then gave explicit expression of jump rates for the limit process which enabled us to study the weight dynamics more precisely. We even show that the divergence of weights is possible even when integral of the learning window is biased towards synaptic depression, even when depression curve is always stronger than depression ($p^+(s)<p^-(s)$ for all $s$). Such a result is not intuitive and led us to find conditions on parameters for which such a divergence doesn't occur. Simulations with more than two neurons showed the winner take all phenomenon takes place. A calibration of parameters is needed to test more characteristics of the model: how does it respond to high frequence, low frequence? Does it enable bidirectional connections?... 
	
	\subsubsection*{Limitations of our model and future work}
	We are aware our neuron model is far from the reality of neurons. It is really simple in order to make the study of plasticity easier. Some questions raise when we try to match it with biology. For instance, what does $\beta$ represents? Many things at the same time: the time one neuron will influence others, the time of a spike as it will not be able to spike again until the moment it comes back to the state 0. Neurons are generally described through their membrane potential which has no link to our model. Then, observations such as potential depolarisation is needed to lead to potentiation cannot be checked or modelled. Moreover, the way their rate of jump from 0 to 1 depends on weights is not really clear and needs to be clarify, maybe there is a need to add delay as it is done in other papers~\cite{lajoie_correlation-based_2017}.
	
	While STDP seems good to keep in memory stimuli, even spontaneously after such inputs~\cite{litwin_formation_2014}, it needs to forget somehow. This seems not be the case in our model. Such a phenomenon is possible for instance under homeostatic mechanisms~\cite{turrigiano_dialectic_2017,zenke_temporal_2017,zenke_synaptic_2013,litwin_formation_2014}. STDP plays the role of additive synaptic scaling as when a weight increases, let say $w^{12}$, then $w^{21}$ decreases. It is not a good thing according to~\cite{turrigiano_dialectic_2017}, as they observed multiplicative synaptic scaling in their experiments. This is understandable as it is too specific and seems not sufficient. It is not useless if you think as information supported by $w^{21}$ is the exact opposite of the one supported by $w^{12}$, it enables neurons " to win time ". So there is a need to add homeostasis to our model. Metaplasticity or plastic inhibitory (iSTDP) neurons are the most used. Indeed, we studied only a network of excitatory neurons. Adding non plastic inhibitory neurons will just decrease the minimum of firing rates of neurons. However, plastic inhibitory neurons could prevent from divergence of weights. Finally, $w^{ii}=0$ is imposed but it could be interesting to use it as an homeostatic factor, decreasing the firing rate when it is to high and increasing it when it is weak.
	
	\subsubsection*{Relation to previous work}
	Analysis using the separation of time scale between weights dynamics and the network one has been done in many other articles~\cite{kempter_hebbian_1999,kempter_intrinsic_2001,burkitt_spike-timing-dependent_2004,gilson_emergence_2009,galtier_biological_2013,ocker_self-organization_2015,lajoie_correlation-based_2017}. They modelled neurons as Poisson, except for~\cite{ocker_self-organization_2015}, and derived a similar equation for weights on their slow time scale. This equation mainly depends on the cross correlation matrix which is not easy to handle with. They use Taylor expansion and Fourier transform to approximate it for their simulations. In our model, such a matrix is hidden in the invariant measure of the fast process. Concerning the stability of weights, a similar result was found in~\cite{kempter_intrinsic_2001} where "a stable fixed point of the output rate is possible if the integral over the learning window is sufficiently negative." As, in their model, rates are linear in weights, stability of rates is equivalent to weights stability. Even if it is not a necessary condition, we could give an idea of how much negative the integral over the learning window needs to be in order to have stability.
	
%	Hard and soft bounds suffer from many problems even if it is natural to think of bounding them from a biological point of view. Indeed, they lead to bistable weights in $w_{max}$ or $w_{min}$ {\color{black}in all studies [à vérifier]}. Then the quantity of interest becomes the mean of weights, firing rates ... It is hard to think this is biologically true. Some information should be learnt forever, other are not really learnt and other are susceptible to be forgotten. In our model, memory is stocked through invariant measure which will depend from the stimuli of course. Weights can diverge if stimuli tend to strengthen always the same weights but will not under reasonable stimuli. 
	
	\subsubsection*{Conclusion}
	We propose a new view on STDP models. In contrast with tiny deterministic jumps of weights, weights have some weak probability to make a "big" jump. Thus, instead of continuous, weights are discrete~\cite{amit_learning_1994,ribrault_stochasticity_2011}. Associated to the inter arrival time of spikes and the network state, we get a Markov process. We simplified it thanks to a separation of time scale and found simple conditions of positive recurrence. This work opens a new framework of study for plasticity which we hope it will give rise to more mathematical results on plasticity in the following.

%	Our first mathematical results are encouraging for deeper study and our model showed more interesting behavior than  those already presented: bidirectional as unidirectional connections can be strong.
%	
%	\subsection*{To do list}
%	\subsubsection*{Simulations}
%	Test different models:\\
%	- excitator-inhibitor\\
%	- hard bounds - soft bounds\\
%	
%	Compare with simulations used by:\\
%	- Clopath~\cite{clopath_connectivity_2009} (temporal - rate coding)\\
%	- Kempter~\cite{kempter_intrinsic_2001,kempter_hebbian_1999} (study of sum of firing rates, sum of weights, various groups of neurons, static pattern scenario, translation invariant scenario)\\
%	- Song (compute of CV, firing rates in function of inputs firing rates, selection of strong weights)\\
%	
%	Compute Laplace transform of $\pi$ in N dimensions\\
%	
%	Interesting variables:\\
%	- moments of order 1, 2\\
%	- sum of weights, firing rates\\
%	- number of neurons in state one\\
%	...
%	
%	\subsubsection*{Research}
%	Understand and do computations of Kempter and Doiron in our model\\
%	Add delays in weights\\
%	Study with $\Delta w=\delta/N$\\
%	Improve condition of non divergence for our particular case (better Lyapunov function?)\\
%	Is it possible to have two stable fixed points (mean of weights as in Doiron) in our case ?
%	
%	
	
\newpage
\appendix
\section*{Annexes}

\section{Dimension 2 for uniqueness}\label{uniqueness_dim2}

After giving the generator $(\mathcal{B},D(\mathcal{B}))$ in 2 dimensions, we then compute the equation satisfies by the Laplace transform of a given stationary distribution for $(S_t,V_t)$.

\subsubsection*{Generator}

\begin{proposition}
	$D(\mathcal{B})=\{f\in C_{ub}(E_2)\ and\ (\partial_{s_1}+\partial_{s_2})f \in C_{ub}(E_2)\}$ and $\forall f\in D(\mathcal{B})$:
	\begin{align*}
	\left \{
	\begin{array}{r c l}
	\mathcal{B}f(s,(0,0))=\alpha_{00}^{01}(f((s_1,0), (0,1))-f(s,(0,0)))+\alpha_{00}^{10}(f((0,s_2), (1,0))-f(s,(0,0)))+\sum_{1}^2\partial_{s_i}f(s,(0,0))\\ \\
	\mathcal{B}f(s,(0,1))=\alpha_{01}^{11}(f((0,s_2), (1,1))-f(s,(0,1)))+\beta (f(s,(0,0))-f(s,(0,1)))+\sum_{1}^2\partial_{s_i}f(s,(0,1))\ \ \ \ \ \ \ \ \\ \\
	\mathcal{B}f(s,(1,0))=\alpha_{10}^{11}(f((s_1,0), (1,1))-f(s,(1,0)))+\beta(f(s, (0,0))-f(s,(1,0)))+\sum_{1}^2\partial_{s_i}f(s,(1,0))\ \ \ \ \ \ \ \ \\ \\
	\mathcal{B}f(s,(1,1))=\beta(f(s, (0,1))-f(s,(1,1)))+\beta (f(s,(1,0))-f(s,(1,1)))+\sum_{1}^2\partial_{s_i}f(s,(1,1))\ \ \ \ \ \ \ \ \ \ \ \ \ \ \ \ 
	\end{array}
	\right.
	\end{align*}	
	Or in a shorter version: \[\mathcal{B}f(s,v)=((\partial_{s_1}+\partial_{s_2})f)(x)+\alpha_v^{(1-v_1,v_2)}[f((s_1v_1,s_2),(1-v_1,v_2))-f(x)]+\alpha_v^{(v_1,1-v_2)}[f
	((s_1,s_2v_2),(v_1,1-v_2))-f(x)]\]
\end{proposition}
\vspace{1em}
\begin{proof}
	Let $f\in D(\mathcal{B})$, then by definition $\lim_{t\rightarrow 0} \frac{\esp_x(f(X_t)) -f(x)}{t}$ exists. Let's compute it. We know that each element $v\in I$ has only two neighbors (in the sens it can only reach two different states). We note $\alpha_v^{v'}$ the rates to reach the neighbor v'. We do the computations for $v=(0,1)$:
	\begin{align*}
	\esp_{(s,(0,1))}(f(S_t,V_t))&=\proba_{(s,(0,1))}(V_t=v)f((s_1+t,s_2+t),(0,1))+\proba_{(s,(0,1))}(V_t=(1,1))f((0,s_2+t),(1,1))\\&+\proba_{(s,(0,1))}(V_t=(0,0))f((s_1+t,s_2+t),(0,0))+o(t)\\
	&=\left(1-\left(\alpha_{01}^{11}+\alpha_{01}^{00}\right)t\ e^{-\left(\alpha_{01}^{11}+\alpha_{01}^{00}\right)t}\right)f((s_1+t,s_2+t),(0,1))
	\\&+\alpha_{01}^{11}t\ e^{-\alpha_{01}^{11}t}f((0,s_2+t),(1,1))+\alpha_{01}^{00}t\ e^{-\alpha_{01}^{00}t}f((s_1+t,s_2+t),(0,0))+o(t)
	\\& =f((s+t),(0,1))+\alpha_{01}^{11}(f((0,s_2+t), (1,1))-f(s+t,(0,1)))
	\\& +\beta (f(s+t,(0,0))-f(s+t,(0,1)))+o(t)
	\end{align*}
	Then we obtain:
	\begin{align*}
	\mathcal{B}f(x)=&\lim_{t\rightarrow 0} \frac{\esp_{(s,(0,1))}(f(X_t)) -f(x)}{t}
	\\=&\alpha_{01}^{11}(f((0,s_2), (1,1))-f(s,(0,1)))
	+\beta (f(s,(0,0))-f(s,(0,1)))+(1,1).\nabla_sf(s,(0,1))
	\end{align*}
	The same kind of computations gives us the same $\mathcal{B}f(x)$ as in the proposition $\forall x\in E_2$, and $D(\mathcal{B})\subseteq\{f\in C_{ub}(E_2)\ and\ (\partial_{s_1}+\partial_{s_2})f \in C_{ub}(E_2)\}$. In order to have the other inclusion, we take $f\in \{g,g\in C_{ub}(E_2)\ and\ (\partial_{s_1}+\partial_{s_2})g \in C_{ub}(E_2)\}$, then we compute for $x=(s,(0,1))\in E_2$:
	\begin{align*}
	&r_x(t)=\\&\left|\frac{\esp_x(f(X_t)) -f(x)}{t}-\left[\alpha_{01}^{11}(f((0,s_2), (1,1))-f(s,(0,1)))
	+\beta (f(s,(0,0))-f(s,(0,1)))+(1,1).\nabla_sf(s,(0,1))\right]\right|
	\end{align*}
	From previous computations, we see the jump terms will disappear because f is uniformly continuous, and the transport term will vanish as $t\rightarrow 0$ because $(1,1).\nabla_sf\in C_{ub}(E_2)$
	\begin{align*}
	&\left|\frac{f(s+t,(0,1))-f(s,(0,1))}{t}-(1,1).\nabla_sf(s,(0,1))\right|
	\\&\leq \frac{1}{t} \int_0^t\left|(1,1).\nabla_sf(s+u,(0,1))-(1,1).\nabla_sf(s,(0,1))\right|du
	\\&\leq\sup_{0\leq u\leq t}\left|(1,1).\nabla_sf(s+u,(0,1))-(1,1).\nabla_sf(s,(0,1))\right|
	\\&\leq\epsilon
	\end{align*}
	If t small enough.\\
	Hence, $\lim_{t\rightarrow 0} \frac{\esp_{(s,(0,1))}(f(X_t)) -f(x)}{t}$ exists. As we can do exactly the same computations for all $x\in E_2$, we deduce that $\{f\in C_{ub}(E_2)\ and\ (\partial_{s_1}+\partial_{s_2})f \in C_{ub}(E_2)\}\subseteq D(\mathcal{B})$. Thus, we have the equality wanted.
\end{proof}
We can see here the need to chose $C_{ub}(E_2)$ instead of $C_{b}(E_2)$ for instance. Indeed, the uniform continuity enable us to conclude on the domain of B and on another hand it is the biggest subspace of $L^{\infty}(E_2)$ on which the derivative is the generator of a $C_0$-semigroup. If we had chosen $C_0(E_2)=\{functions\ vanishing\ at\ \infty\}$, we see immediately the semigroup associated to our process will not map $C_0(E_2)$ into itself. $T_tf$ has no reason to vanish at $\infty$. $C_{ub}(E_2)$ seems to be the space that suits. Moreover, thanks to the portmanteau lemma, the knowledge of the semigroup on $C_{ub}(E_2)$ characterizes the law of the process. We can then use the definition~\ref{prop-inv-mes-gen} to search the Laplace transforms of invariant measures.

\subsubsection*{Laplace transform}

First, we show we can write any invariant measure of the process in the form $\pi(s,v)=\sum_{k\in I} \delta_{v_k}(v) \mu^w_k \pi_k(s)$ where $(\mu^w_1,...,\mu^w_N)$ is the only invariant measure of the jump process $(V_t)$ and $\pi_k$ is a measure on $\mathcal{B}(\reels_+^2)$. Then, we prove that if the process $(X_t)_{t\geq 0}$ has at least one invariant measure of probability $\pi$, then it is unique.\\

%\begin{proposition}
%	Every $\pi$ invariant measure of $X_t$ has the same Laplace transform and it is of the following form:
%	\[(\mathcal{L}\pi_i)(\lambda_1,\lambda_2)=\int_{\reels_+^2}e^{-(\lambda_1s_1+\lambda_2s_2)}\pi_i(ds)=\frac{P^v(\lambda_1,\lambda_2)}{Q^v(\lambda_1,\lambda_2)}=F^v(\lambda_1,\lambda_2)\]
%	With, $\forall v\in I:
%	$\begin{align*}
%	\left \{
%	\begin{array}{r c l}
%	&deg(P^v,\lambda_j)<deg(Q^v,\lambda_j),\ \ j\in\{1,2\}\\
%	\\
%	&Q^v(\lambda_1,\lambda_2)=Q^v_1(\lambda_1)
%	Q^v_2(\lambda_2)Q_3^v(\lambda_1+\lambda_2)
%	\end{array}
%	\right.		
%	\end{align*}
%	Where $Q_j^v$ are polynomials with strictly negative simple roots .
%\end{proposition}
%
%To show this, we suppose $\pi$ invariant measure of probability exists. Then, we use definition~\ref{prop-inv-mes-gen} to find a system satisfied by $\mathcal{L}\pi_i$. We show this system has a unique solution which has the preceding form.\\
%First, let show the jump process alone $(V_t)_{t\geq 0}$ has a unique invariant measure  of probability.\\

It is interesting to look at the form of invariant measures for the following. Indeed, as $(V_t)$ doesn't depend on $(S_t)$, we can study its dynamic and deduce a nice decomposition of the stationary distribution of $(X_t)$.
\begin{proposition}
	The jump process alone $(V_t)_{t\geq 0}$ has a unique invariant probability measure $\vec{\mu}=(\mu^w_{00},\mu^w_{01},\mu^w_{10},\mu^w_{11})^T$. Moreover, $\mu^w_v>0,\ \forall v\in I$, and it satisfies:
	\begin{align}\label{murel2}
	Q\vec{\mu}=\begin{bmatrix}
	-\alpha_{00}^{01}-\alpha_{00}^{10}&\ \beta&\ \beta &\ 0 \\ \\
	\alpha_{00}^{01} & \ -\alpha_{01}^{11}-\beta &\ 0 & \ \beta \\ \\
	\alpha_{00}^{10} & \ 0 & -\ \alpha_{10}^{11}-\beta & \ \beta \\ \\
	0 &\ \alpha_{01}^{11}& \ \alpha_{10}^{11} & -2\beta
	\end{bmatrix}
	\begin{bmatrix}
	\mu^w_{00}\\ \\
	\mu^w_{01} \\ \\
	\mu^w_{10} \\ \\
	\mu^w_{11}
	\end{bmatrix}	
	=0
	\end{align}
\end{proposition} 
\begin{proof}
	Indeed, as each neuron is connected to each other, $(V_t)_{t\geq 0}$ is irreducible. As its state space is finite, the process is also positive recurrent so has a unique invariant probability measure $\mu^w$ by theorem1.7.7 in~\cite{norris_markov_1998}.\\
	Moreover, as each state is positive recurrent, $\mu^w_v>0,\ \forall v\in I$.\\
	The matrix $Q$ is the matrix of transition rates (Q-matrix) of $(V_t)_{t\geq 0}$. With $1=(0,0),\ 2=(0,1),\ 3=(1,0),\ 4=(1,1)$, and $Q=(q_{ij})_{1\leq i,j\leq4}$ we have Q has in the proposition. As $\mu^w$ is invariant, it belongs to the kernel of Q, which is \eqref{murel2}, Theorem 3.5.5 in~\cite{norris_markov_1998}.
\end{proof}
\vspace{1em}
From this result, we deduce that $\forall k\in I,\ \ \int_{\reels_+^2}\pi(ds,k)=\mu^w_k$. Therefore, we define $\pi_k$ as $\pi_k(A)=\frac{\int_{A}\pi(ds,k)}{\mu^w_k}$, $\forall A\in \mathcal{B}(\reels_+^2)$. Hence, $\pi(s,v)=\sum_{k\in I} \delta_{v_k}(v) \mu^w_k \pi_k(s)$. \\

Now, we previously showed the process $(X_t)_{t\geq 0}$ has at least one invariant probability measure on $E_2$, let $\pi$ be one of them and let's compute its Laplace transform to show the following proposition:
\begin{proposition}
	Assume the process $(X_t)_{t\geq 0}$ has at least one invariant measure of probability $\pi$. Then it is unique.
	%	
	%	Then, $\forall v\in I$, $\mathcal{L}\pi_v$ is defined on $\reels_+^*\times\reels_+^*$ and is a polynomial fraction as follows:
	%	\[ \mathcal{L}\pi_v(\lambda_1,\lambda_2)=\frac{P^v(\lambda_1,\lambda_2)}{Q^v(\lambda_1,\lambda_2)}\]
	%	With, $\forall v\in I:
	%	$\begin{align*}
	%	\left \{
	%	\begin{array}{r c l}
	%	&deg(P^v,\lambda_j)<deg(Q^v,\lambda_j),\ \ j\in\{1,2\}\\
	%	\\
	%	&Q^v(\lambda_1,\lambda_2)=Q^v_1(\lambda_1)
	%	Q^v_2(\lambda_2)Q_3^v(\lambda_1+\lambda_2)
	%	\end{array}
	%	\right.		
	%	\end{align*}
	%	Where $Q_j^v$ are polynomials with strictly negative simple roots.
\end{proposition} 

\begin{proof}
	We will show that all invariant measure of probability has the same Laplace transform and as the later characterizes it, see for instance Theorem 4.3 in~\cite{kallenberg_foundations_2006}, there only exists one invariant measure of probability. \\
	
	We can write $\pi$ as $\pi(A,v)=\sum_{k\in I}\pi_k(A) \otimes \mu^w_k \delta_{k}(v),\ \ \forall A\in \mathcal{B}(R_+^2)$, with $\pi_k(A)=\pi(A,k)(\mu^w_k)^{-1}$. To simplify computations, we will denote by $\mathcal{L}\pi$ be the vector of Laplace transforms of $\pi_k$. So $\forall \lambda_1, \lambda_2\in \reels_+$:
	\[\mathcal{L}\pi(\lambda_1,\lambda_2)=\left[ \begin {array}{c} {\it \mathcal{L}\pi_{00}} \left( \lambda_1,\lambda_2 \right) 
	\\ \noalign{\medskip}{\it \mathcal{L}\pi_{01}} \left( \lambda_1,\lambda_2 \right) 
	\\ \noalign{\medskip}{\it \mathcal{L}\pi_{10}} \left( \lambda_1,\lambda_2 \right) 
	\\ \noalign{\medskip}{\it \mathcal{L}\pi_{11}} \left( \lambda_1,\lambda_2 \right) \end {array}
	\right]\ \ where\ \ \forall v\in I,\ \ \mathcal{L}\pi_v(\lambda_1,\lambda_2)=\int_{\reels_+^2}e^{-(\lambda_1s_1+\lambda_2s_2)}\pi_v(ds)\]
	Just a remark, $\forall v\in I$
	\[\mathcal{L}\pi_v(0,0)=\int_{\reels_+^2}\pi_v(ds)=(\mu^w_v)^{-1}\int_{\reels_+^2}\pi(ds,v)=1\]
	As we want to compute the Laplace transform of $\pi$ which is in fact $(\lambda_1,\lambda_2)\mapsto\sum_{v\in I}\mu^w_v\mathcal{L}\pi_v(\lambda_1,\lambda_2)$, let's use the following test functions, with $\lambda=(\lambda_1,\lambda_2)$ and $\forall\ k\in I$: 
	\[
	e_{\lambda}^{k}(s,v)=e^{- (\lambda_1s_1+\lambda_2s_2)}\delta_{k}(v)
	\]
	
	By definition~\ref{prop-inv-mes-gen} of an invariant measure we get $\forall v\in I$:
	\begin{align}\label{relpimu}
	\sum_{k\in I}\int_{{\reels_+}^2}\mathcal{B}e_{\lambda}^{v}(s,k)\mu^w_k\pi_k(ds)=0
	\end{align}
	
	We then compute $\mathcal{B}e_{\lambda}^{k}(s,v)$:
	\begin{align*}
	\mathcal{B}e_{\lambda}^{00}(s,(0,0))&=(-\alpha_{00}^{01}-\alpha_{00}^{10}-(\lambda_1+\lambda_2))e^{- \lambda_1s_1- \lambda_2s_2} 
	\\
	\mathcal{B}e_{\lambda}^{00}(s,(0,1))&=\beta e^{- \lambda_1s_1- \lambda_2s_2}\\
	\mathcal{B}e_{\lambda}^{00}(s,(1,0))&=\beta e^{- \lambda_1s_1- \lambda_2s_2}\\
	\mathcal{B}e_{\lambda}^{00}(s,(1,1))&=0
	\\
	\\
	\mathcal{B}e_{\lambda}^{01}(s,(0,0))&=\alpha_{00}^{01}e^{- \lambda_1s_1}
	\\
	\mathcal{B}e_{\lambda}^{01}(s,(0,1))&=(-\alpha_{12}-\beta-(\lambda_1+\lambda_2))e^{- \lambda_1s_1- \lambda_2s_2}\\
	\mathcal{B}e_{\lambda}^{01}(s,(1,0))&=0\\
	\mathcal{B}e_{\lambda}^{01}(s,(1,1))&=\beta e^{- \lambda_1s_1- \lambda_2s_2}\\
	\\
	\mathcal{B}e_{\lambda}^{10}(s,(0,0))&=\alpha_{00}^{10}e^{- \lambda_2s_2}
	\\
	\mathcal{B}e_{\lambda}^{10}(s,(0,1))&=0\\
	\mathcal{B}e_{\lambda}^{10}(s,(1,0))&=(-\alpha_{23}-\beta-(\lambda_1+\lambda_2))e^{- \lambda_1s_1- \lambda_2s_2}\\
	\mathcal{B}e_{\lambda}^{10}(s,(1,1))&=\beta  e^{- \lambda_1s_1- \lambda_2s_2}
	\end{align*}
	\begin{align*}
	\mathcal{B}e_{\lambda}^{11}(s,(0,0))&=0
	\\
	\mathcal{B}e_{\lambda}^{11}(s,(0,1))&=\alpha_{01}^{11}e^{- \lambda_2s_2}\\
	\mathcal{B}e_{\lambda}^{11}(s,(1,0))&=\alpha_{10}^{11}e^{- \lambda_1s_1}\\
	\mathcal{B}e_{\lambda}^{11}(s,(1,1))&=(-2\beta-(\lambda_1+\lambda_2))e^{- \lambda_1s_1- \lambda_2s_2}\\
	\end{align*}
	
	So with \eqref{relpimu} and $v=(0,0)$ for instance:
	\begin{align*} 
	\sum_{k\in I}\int_{{\reels_+}^2}\mathcal{B}e_{\lambda}^{00}&(s,k)\mu^w_k\pi_k(ds)=0\\
	&\Leftrightarrow\\
	\left(-\alpha_{00}^{01}-\alpha_{00}^{10}-(\lambda_1+\lambda_2)\right)\mu^w_1&\mathcal{L}\pi_{00}(\lambda_1,\lambda_2)+\beta\mu^w_2\mathcal{L}\pi_{01}+\beta\mu^w_3\mathcal{L}\pi_{10}=0
	\end{align*}
	After computations for all $v\in I$ we get:
	\begin{align} \label{eqTL}
	M(\lambda_1,\lambda_2)\ \left[ \begin {array}{c} {\it \mathcal{L}\pi_{00}} \left( \lambda_1,\lambda_2 \right) 
	\\ \noalign{\medskip}{\it \mathcal{L}\pi_{01}} \left( \lambda_1,\lambda_2 \right) 
	\\ \noalign{\medskip}{\it \mathcal{L}\pi_{10}} \left( \lambda_1,\lambda_2 \right) 
	\\ \noalign{\medskip}{\it \mathcal{L}\pi_{11}} \left( \lambda_1,\lambda_2 \right) \end {array}
	\right] =\left[ \begin {array}{c} 0\\ \noalign{\medskip}-{\it \alpha_{00}^{01}\ \mathcal{L}\pi_{00}} \left( \lambda_1,0
	\right)\frac {{\it \mu^w_1}}{{\it \mu^w_2}} \\ \noalign{\medskip}-{\it \alpha_{00}^{10} \mathcal{L}\pi_{00}} \left( 0,\lambda_2 \right) \frac {{\it \mu^w_1}}{{\it \mu^w_3}}
	\\ \noalign{\medskip}{\it -\alpha_{01}^{11}}\,{\it \mathcal{L}\pi_{01}} \left( 0,\lambda_2 \right)\frac {{\it \mu^w_2}}{{\it \mu^w_4}} -{\it 
		\alpha_{10}^{11}}\,{\it \mathcal{L}\pi_{10}} \left( \lambda_1,0 \right)\frac {{\it \mu^w_3}}{{\it \mu^w_4}} \end {array} \right]
	\end{align}
	With:
	\begin{align*}
	M(\lambda_1,\lambda_2)=\left[ \begin {array}{cccc} -{\it \alpha_{00}^{10}}-{\it \alpha_{00}^{01}}-{\it \lambda_1}-{\it 
		\lambda_2}&{\it \beta\frac {{\it \mu^w_2}}{{\it \mu^w_1}}}&{\it \beta\frac {{\it \mu^w_3}}{{\it \mu^w_1}}}&0\\ \noalign{\medskip}0&-{\it \alpha_{01}^{11}}-{
		\it \beta}-{\it \lambda_1}-{\it \lambda_2}&0&{\it \beta\frac {{\it \mu^w_4}}{{\it \mu^w_2}}}
	\\ \noalign{\medskip}0&0&-{\it \alpha_{10}^{11}}-{\it \beta}-{\it \lambda_1}-{\it 
		\lambda_2}&{\it \beta\frac {{\it \mu^w_4}}{{\it \mu^w_3}}}\\ \noalign{\medskip}0&0&0&-2\,{\it \beta}-{\it 
		\lambda_1}-{\it \lambda_2}\end {array} \right]
	\end{align*}
	As we have $\mathcal{L}\pi(\lambda_1,0)=\left[ \begin {array}{c} {\it \mathcal{L}\pi_{00}} \left( \lambda_1,0 \right) 
	\\ \noalign{\medskip}{\it \mathcal{L}\pi_{01}} \left( \lambda_1,0 \right) 
	\\ \noalign{\medskip}{\it \mathcal{L}\pi_{10}} \left( \lambda_1,0 \right) 
	\\ \noalign{\medskip}{\it \mathcal{L}\pi_{11}} \left( \lambda_1,0 \right) \end {array}
	\right]$ and $\mathcal{L}\pi(0,\lambda_2)=\left[ \begin {array}{c} {\it \mathcal{L}\pi_{00}} \left( 0,\lambda_2 \right) 
	\\ \noalign{\medskip}{\it \mathcal{L}\pi_{01}} \left( 0,\lambda_2 \right) 
	\\ \noalign{\medskip}{\it \mathcal{L}\pi_{10}} \left( 0,\lambda_2 \right) 
	\\ \noalign{\medskip}{\it \mathcal{L}\pi_{11}} \left( 0,\lambda_2 \right) \end {array}
	\right]$.
	\\
	\\Then we can get $\mathcal{L}\pi(\lambda_1,0)$ and $\mathcal{L}\pi(0,\lambda_2)$ evaluating \eqref{eqTL} in $\lambda_1=0$ and $\lambda_2=0$:
	\[M(\lambda_1,0)\ \mathcal{L}\pi(\lambda_1,0)=\left[ \begin {array}{c} 0\\ \noalign{\medskip}-{\it \alpha_{00}^{01}\ \mathcal{L}\pi_{00}} \left( \lambda_1,0
	\right)\frac {{\it \mu^w_1}}{{\it \mu^w_2}} \\ \noalign{\medskip}-{\it \alpha_{00}^{10} \mathcal{L}\pi_{00}} \left( 0,0 \right) \frac {{\it \mu^w_1}}{{\it \mu^w_3}}
	\\ \noalign{\medskip}{\it -\alpha_{01}^{11}}\,{\it \mathcal{L}\pi_{01}} \left( 0,0 \right)\frac {{\it \mu^w_2}}{{\it \mu^w_4}} -{\it 
		\alpha_{10}^{11}}\,{\it \mathcal{L}\pi_{10}} \left( \lambda_1,0 \right)\frac {{\it \mu^w_3}}{{\it \mu^w_4}} \end {array} \right]\]
	As $\forall v\in I$, $\mathcal{L}\pi_v(0,0)=\int_{\reels^{+^2}}\pi_v(ds_1,ds_2)=1$ so:
	\[M(\lambda_1,0)\ \mathcal{L}\pi(\lambda_1,0)=\begin{bmatrix}
	0&0&0&0\\
	-\frac{\alpha_{00}^{01}\mu^w_1}{\mu^w_2}&0&0&0\\
	0&0&0&0\\
	0&0&-\frac{\alpha_{10}^{11}\mu^w_3}{\mu^w_4}&0
	\end{bmatrix}\mathcal{L}\pi(\lambda_1,0)+\begin{bmatrix}0\\0\\-\frac{\alpha_{00}^{10}\mu^w_1}{\mu^w_3}\\-\frac{\alpha_{01}^{11}\mu^w_2}{\mu^w_4}\end{bmatrix}\]
	And
	\[M(0,\lambda_2)\ \mathcal{L}\pi(0,\lambda_2)=\begin{bmatrix}
	0&0&0&0\\
	0&0&0&0\\
	-\frac{\alpha_{00}^{10}\mu^w_1}{\mu^w_3}&0&0&0\\
	0&-\frac{\alpha_{01}^{11}\mu^w_2}{\mu^w_4}&0&0
	\end{bmatrix}\mathcal{L}\pi(\lambda_1,0)+\begin{bmatrix}0\\-\frac{\alpha_{00}^{01}\mu^w_1}{\mu^w_2}\\0\\-\frac{\alpha_{10}^{11}\mu^w_3}{\mu^w_4}\end{bmatrix}\]
	Putting terms in $T_i(\lambda_i)$ in matrices marked $M_i(\lambda_i)$ we get:
	\begin{align}\label{partial-TL}
	M_1(\lambda_1) \mathcal{L}\pi(\lambda_1,0) =\left[ \begin {array}{c} 0\\ \noalign{\medskip}-{\it \alpha_{00}^{01}} \frac {{\it \mu^w_1}}{{\it \mu^w_2}} \\ \noalign{\medskip}0
	\\ \noalign{\medskip}-{\it 
		\alpha_{10}^{11}}\frac {{\it \mu^w_3}}{{\it \mu^w_4}} \end {array} \right]
	\ \ \ and\ \ \ 
	M_2(\lambda_2) \mathcal{L}\pi(0,\lambda_2) =\left[ \begin {array}{c} 0\\ \noalign{\medskip}0\\ \noalign{\medskip}-{\it \alpha_{00}^{10}} \frac {{\it \mu^w_1}}{{\it \mu^w_3}} 
	\\ \noalign{\medskip}-{\it 
		\alpha_{01}^{11}}\frac {{\it \mu^w_2}}{{\it \mu^w_4}} \end {array} \right]
	\end{align}
	With:
	\begin{align*}
	M_1(\lambda_1)=\left[ \begin {array}{cccc} -{\it \alpha_{00}^{10}}-{\it \alpha_{00}^{01}}-\lambda_1&{\frac {\beta\,\mu
			2}{\mu1}}&{\frac {\beta\,\mu3}{\mu1}}&0\\ \noalign{\medskip}{\frac {{
				\it \alpha_{00}^{01}}\,\mu1}{\mu2}}&-{\it \alpha_{01}^{11}}-\beta-\lambda_1&0&{\frac {\beta\,\mu4}{\mu2}
	}\\ \noalign{\medskip}0&0&-{\it \alpha_{10}^{11}}-\beta-\lambda_1&{\frac {\beta\,\mu4}{\mu3
	}}\\ \noalign{\medskip}0&0&{\frac {{\it \alpha_{10}^{11}}\,\mu3}{\mu4}}&-2\,\beta-\lambda_1
	\end {array} \right]
	\end{align*}
	And
	\begin{align*}
	M_2(\lambda_2)= \left[ \begin {array}{cccc} -{\it \alpha_{00}^{10}}-{\it \alpha_{00}^{01}}-\lambda_2&{\frac {\beta\,\mu
			2}{\mu1}}&{\frac {\beta\,\mu3}{\mu1}}&0\\ \noalign{\medskip}0&-{\it 
		\alpha_{01}^{11}}-\beta-\lambda_2&0&{\frac {\beta\,\mu4}{\mu2}}\\ \noalign{\medskip}{\frac 
		{{\it \alpha_{00}^{10}}\,\mu1}{\mu3}}&0&-{\it \alpha_{10}^{11}}-\beta-\lambda_2&{\frac {\beta\,\mu4}{\mu
			3}}\\ \noalign{\medskip}0&{\frac {{\it \alpha_{01}^{11}}\,\mu2}{\mu4}}&0&-2\,\beta-
	\lambda_2\end {array} \right] 	
	\end{align*}
	As a triangular superior matrix with diagonal elements strictly positive, $M$ is invertible. Moreover, $M_1$ and $M_2$ are invertible as diagonally dominant matrices whenever $(\lambda_1,\lambda_2)\in\reels_+^*\times\reels_+^*$:\\
	First line of \eqref{murel2} gives:
	\[\left(\alpha_{00}^{10}+\alpha_{00}^{01}\right)\mu^w_1=\beta\mu^w_2+\beta\mu^w_3\]
	So $\forall \lambda_1\in\reels_+^*$:
	\[\lvert M_1^{11}(\lambda_1)\rvert-\sum_{j=2}^{4}\lvert M_1^{1j}(\lambda_1)\rvert = \lambda_1>0\]
	For other lines we have $\forall \lambda_1\in\reels_+^*$:
	\[\lvert M_1^{22}(\lambda_1)\rvert-\sum_{j\neq 2}\lvert M_1^{2j}(\lambda_1)\rvert = \lambda_1>0\]
	\[\lvert M_1^{33}(\lambda_1)\rvert-\sum_{j\neq 3}\lvert M_1^{3j}(\lambda_1)\rvert = \frac{\alpha_{00}^{10}}{\mu^w_3}+\lambda_1>0\]
	\[\lvert M_1^{44}(\lambda_1)\rvert-\sum_{j\neq 4}\lvert M_1^{4j}(\lambda_1)\rvert = \frac{\alpha_{01}^{11}}{\mu^w_4}+\lambda_1>0\]
	We have similar results for $M_2$ which shows $M_1(\lambda_1)$ et $M_2(\lambda_2)$ are diagonally dominant matrices so they are invertible $\forall (\lambda_1,\lambda_2)\in\reels_+^*\times\reels_+^*$. Hence, if $\pi$ is an invariant measure for $(X_t)_{t\geq 0}$, $\forall (\lambda_1,\lambda_2)\in\reels_+^*\times\reels_+^*$:
	\[\mathcal{L}\pi(0,0)=\begin{bmatrix}
	1\\ 1\\ 1\\ 1\\
	\end{bmatrix}\]
	By~\eqref{partial-TL}:
	\begin{align}\label{eq_lap_0}
	\mathcal{L}\pi(\lambda_1,0) =	\left(M_1(\lambda_1)\right)^{-1}\left[ \begin {array}{c} 0\\ \noalign{\medskip}-{\it \alpha_{00}^{01}} \frac {{\it \mu^w_1}}{{\it \mu^w_2}} \\ \noalign{\medskip}0
	\\ \noalign{\medskip}-{\it 
		\alpha_{10}^{11}}\frac {{\it \mu^w_3}}{{\it \mu^w_4}} \end {array} \right]
	\ \ \ and\ \ \ 
	\mathcal{L}\pi(0,\lambda_2) =\left(M_2(\lambda_2)\right)^{-1}\left[ \begin {array}{c} 0\\ \noalign{\medskip}0\\ \noalign{\medskip}-{\it \alpha_{00}^{10}} \frac {{\it \mu^w_1}}{{\it \mu^w_3}} 
	\\ \noalign{\medskip}-{\it 
		\alpha_{01}^{11}}\frac {{\it \mu^w_2}}{{\it \mu^w_4}} \end {array} \right]
	\end{align}
	By~\eqref{eqTL}
	\[\mathcal{L}\pi(\lambda_1,\lambda_2)=(M(\lambda_1,\lambda_2))^{-1}\left[ \begin {array}{c} 0\\ \noalign{\medskip}-{\it \alpha_{00}^{01}\ \mathcal{L}\pi_{00}} \left( \lambda_1,0
	\right)\frac {{\it \mu^w_1}}{{\it \mu^w_2}} \\ \noalign{\medskip}-{\it \alpha_{00}^{10} \mathcal{L}\pi_{00}} \left( 0,\lambda_2 \right) \frac {{\it \mu^w_1}}{{\it \mu^w_3}}
	\\ \noalign{\medskip}{\it -\alpha_{01}^{11}}\,{\it \mathcal{L}\pi_{01}} \left( 0,\lambda_2 \right)\frac {{\it \mu^w_2}}{{\it \mu^w_4}} -{\it 
		\alpha_{10}^{11}}\,{\it \mathcal{L}\pi_{10}} \left( \lambda_1,0 \right)\frac {{\it \mu^w_3}}{{\it \mu^w_4}} \end {array} \right]\]
	We conclude using the fact the Laplace transform of a law determines it, so $\pi$ is unique.
\end{proof}

\newpage
%\bibliographystyle{abbrv}
%\bibliography{biblio}
%

\end{document}